\documentclass[a4paper,11pt]{article}
\usepackage{makeidx}
\usepackage[english]{babel}
\usepackage[latin1]{inputenc}
\usepackage{amsfonts}
\usepackage{bm}
\usepackage{amssymb}
\usepackage{latexsym}
\usepackage{amsmath}
\usepackage{amscd}
\usepackage{amsthm}
\usepackage{rotating}
\usepackage{graphicx}
\usepackage{pifont}
\usepackage{curves}
\usepackage{fancyhdr}
\usepackage{epsfig}
\usepackage{pstricks}
\usepackage{pst-tree}
\usepackage{subfigure}
\usepackage{epic}
\usepackage{alltt}
\usepackage[all]{xy}
\usepackage{appendix}

\theoremstyle{plain}
\newtheorem{theorem}{Theorem}[section]
\newtheorem{prop}[theorem]{Proposition}
\newtheorem{lemme}[theorem]{Lemma}
\newtheorem{corol}[theorem]{Corollary}

\theoremstyle{definition}

\newtheorem{hyp}[theorem]{Assumption}

\newtheorem{rque}[theorem]{Remark}

\title{Age-structured Trait Substitution Sequence Process and Canonical Equation}

\author{Sylvie Méléard$^\heartsuit$, Viet Chi Tran$^\sharp$\\
 {\footnotesize $^\heartsuit$ CMAP, Ecole Polytechnique, route de Saclay, 91128 Palaiseau Cedex, Sylvie.Meleard@polytechnique.edu}
\\
{\footnotesize $^\sharp$ Laboratoire Paul Painlevé, Université Lille 1, 59655 Villeneuve d'Ascq Cedex, Chi.Tran@math.univ-lille1.fr}}

\date{\today}

\numberwithin{equation}{section}

\newcommand{\Co}{\mathcal{C}}

\newcommand{\N}{\mathbb{N}}

\newcommand{\R}{\mathbb{R}}
\newcommand{\X}{\mathcal{X}}

\newcommand{\lbrac}{\left[\!\left[}
\newcommand{\rbrac}{\right]\!\right]}

\begin{document}

\maketitle

\begin{abstract}
We are interested in a stochastic model of trait and age-structured
population undergoing mutation and selection. We start with a continuous
time, discrete individual-centered population process. Taking the large
population and rare mutations limits under a well-chosen time-scale
separation condition, we obtain a jump process that generalizes the Trait Substitution
Sequence process describing Adaptive Dynamics for populations without age structure. Under the additional assumption of small mutations,  we derive an age-dependent ordinary differential equation that extends the Canonical Equation. These evolutionary approximations have never been introduced to our knowledge. They are based on ecological phenomena represented by PDEs that generalize the Gurtin-McCamy equation in Demography. Another particularity is that they involve a fitness function, describing the probability of invasion of the resident population by the mutant one, that can not always be computed explicitly. Examples illustrate how adding an age-structure enrich the modelling of structured population by including life history features such as senescence. In the cases considered, we establish the evolutionary approximations and study their long time behavior and the nature of their evolutionary singularities when computation is tractable. Numerical procedures and simulations are carried.
\end{abstract}


\section{Introduction}\label{sectionresultats}

\noindent \textit{Structured populations} are populations in which individuals
differ according to variables that affect their reproductive or
survival capacities. These variables  can be phenotypic, genotypic
or behaviorial \textit{traits} that are assumed to be hereditarily
transmitted from a
 parent to its descendants unless a mutation occurs. They can also be
 \textit{position}, \textit{sex} or \textit{age}. In this article, we
 are interested in modelling the adaptive evolution of  trait and age-structured
 populations,
describing the ability of a population to generate and select
trait diversity at an individual level and through the expression
of the individual ages and traits. Individual-centered models
allow us to give a realistic description of these phenomena
  and to obtain  macroscopic equations, at a large population scale.

\noindent As emphasized by Charlesworth \cite{charlesworth}, it is important
to take age structure into account.
    Most populations consist in sets of individuals born over a range of
     past times, with behaviors that depend on their
     ages. As mentioned in
     \cite{charlesworth},  there exist for many species a juvenile and an adult
periods. The first period is devoted to maturation and growth,
 and the individual starts reproducing only
 during the second period. See   the \textit{Tribolium} population (cf. \cite{henson}) as an illustration.
 For many species, and mammals in particular, fecundity and survival functions are
decreasing with age (cf. \cite{promislow}). Some other
species exhibit long postponement of reproduction. The pink salmon
\textit{Oncorhynchus Gorbuscha} breeds at about 2 years old, which
is approximatively its life expectation (cf. \cite{charlesworth, kingsbury}), and the cicada
\textit{Magicicada} reproduces and dies at 13-17 years old (\cite{charlesworth}) for instance.

\noindent Adding an age-structure allows us to study life history traits, such as age at maturity,
or to consider senescence phenomena, which describe the fixation of
deleterious genes or the decrease of reproductive and survival
capacities in the course of life. Many questions concerning these age-structured
populations can be raised. How does the age influence the trait evolution? For a given trait, how does the probability of fixation or elimination by natural selection depend on its age of expression? How does the age-structure of a population influence the selective pressure exerted on individuals? Which
age distributions will appear at equilibrium for a given trait?

\noindent In this paper, our main interest is to generalize to
age-structured populations
  the recent theory of \textit{Adaptive Dynamics}.
 First models of Adaptive Dynamics
  have been introduced by Hofbauer and Sigmund \cite{hofbauersigmund}, Metz \textit{et al.}
  \cite{metzgeritzmeszenajacobsheerwaarden}, Dieckmann and Law \cite{dieckmannlaw}
  and rigorous
   microscopic derivations have been obtained by Champagnat \cite{champagnat, champagnat3},
   Champagnat \textit{et al.} \cite{champagnatferrieremeleard, champagnatferrieremeleard2}.
   In these models, the population is only structured by the trait of interest.

\noindent We start with the description of a discrete stochastic individual-centered model
of population
 structured by trait and age. We consider approximations of the microscopic process
 under a large population
 asymptotic and obtain in the limit a deterministic partial
 differential equation (PDE) involving trait and age,
 that generalizes classical demographic PDEs
  (see McKendrick \cite{mckendrick}, Von Foerster \cite{vonfoerster}, Gurtin and McCamy
  \cite{gurtinmaccamy},
   Webb \cite{webb}, Murray \cite{murray}, Charlesworth \cite{charlesworth}, Thieme \cite{thieme}). This gives us information on the long time behavior of the population. We are then able to characterize rare mutation rates for which it
is possible, under an additional assumption of non-coexistence in
   the long term of two different traits, to separate the time scale of ecology describing the demographic variations in the population and the time
   scale of evolution linked to the occurrence of mutations. The latter hypothesis generalizes the "Invasion implies
   Fixation" assumption of the Adaptive Dynamics theory.
   We prove in this case that the microscopic process converges to
    the so-called \textit{Age-structured Trait Substitution Sequence Process} that jumps from
    a monomorphic
    equilibrium (where all the individuals carry the same trait) to another. This
     process extends to populations
     with age-structure the \textit{Trait Substitution Sequence} (TSS) introduced by Metz \textit{et al.}
     \cite{metzgeritzmeszenajacobsheerwaarden}. A difficulty in our case lies in the fact that the
     fitness function describing the invasion of a mutant trait in the resident population and which appears in the generator
      of the TSS process can not be computed explicitly. Taking the limit of the TSS when the mutation steps tend to zero gives us an age-dependent ordinary differential equation (ODE),
  generalizing the \textit{Canonical Equation}
   proposed by Dieckmann and Law \cite{dieckmannlaw}. We show that an explicit expression for the fitness gradient
    appearing in the ODE is available despite of the implicit definition of the fitness function for the  TSS.
     \noindent To our knowledge, the equations that we establish have not been introduced in the
 biological literature yet
 and are more general than the ones  previously proposed
 by Dieckmann \textit{et al.} \cite{dieckmannheinoparvinen}, Ernande \textit{et al.}
 \cite{ernandedieckmanheino},
 Metz \cite{metzencyclop}, Parvinen \textit{et al.}
 \cite{parvinendieckmannheino}.\\
\noindent Our results point out that in the adaptive dynamics limit, the age-structured TSS and Canonical Equation involve averaged functions in age. The age-structure does not mainly affect the qualitative behavior of the population but plays a role in the trade-offs determining the Evolutionary Stable Strategies (ESS). With age-structure, we gain in realism and refine the results by including effects due to life histories and age-dependent behaviors.\\

\noindent Sections \ref{sectionprocessusmicroevoltion} to \ref{sectionpreztss}
are devoted to the presentation of the three models of interest: the microscopic model,
the TSS model, and the Canonical Equation.
 Then, we state the limit theorems linking these models (Theorems \ref{theoremtss} and
 \ref{theoremequationcanoniquefoncionnelle}). Several examples are considered. in Section \ref{sectionexemple}, we begin with a logistic model with senescence in the birth rate and size dependence in the competition term. The long time behaviors of the TSS process and of the solution of the Canonical Equation are studied. We highlight differences with the case without age-structure. In our model, senescence acts as a penalization term that favors fast reproduction to growth. We then consider a model belonging to a class that we call \textit{age-logistic}. Computation becomes rapidly complicated. In Section \ref{sectionexemple2} we investigate an example where the competition kernel is nonlocal, asymmetric in trait (cf. Kisdi \cite{kisdi}) and decreasing in the age of the competitors. The explicit expression of the fitness gradient allows us to compute the ESS value and to show that coexistence is possible. The existence of an "evolutionary branching" is suggested by numerical simulations.

\section{Microscopic model}\label{sectionprocessusmicroevoltion}

We start with a stochastic microscopic model of age and trait
structured asexual population describing the
dynamics at the individual level. We take into account births, either clonal
or with mutation and death, natural or due to the competition with
other individuals. A large population limit is studied. We
provide by this way a microscopic justification of PDEs generalizing models
introduced in Demography (\cite{charlesworth, mckendrick, murray, thieme, webb}). This gives us a better
understanding of the large time behavior of the microscopic process and
allows us to define the rare mutations asymptotics leading to Adaptive
Dynamics limits. To our knowledge, there are only a few models
dealing with trait and age-structured populations, all deterministic
(Rotenberg \cite{rotenberg},
 Mischler \textit{et al.} \cite{mischlerperthameryzhik}) and the stochastic models taking age
  into account do not consider trait evolution, see Kendall \cite{kendall}, Athreya and Ney \cite{athreyaney},
  Doney \cite{doney},
   Oelschläger \cite{oelschlager}, Tran \cite{trangdesdev}.
    We generalize all these models by introducing a
   dependence between
  trait and age at the individual level and by
   taking into account mutation and
   competition between individuals, which yields nonlinearity in
   the limiting phenomena.

\subsection{Microscopic description}\label{sectiondescription}

Individuals are characterized by a trait $x$ belonging to a
compact set $\X$ of $\R^{d}$ and by their physical age $a\in
\R_+$. We set $\widetilde{\X}:=\X\times \R_+.$
 The population is discrete and described by the point measure
  \begin{equation}
  Z_t(dx,da)=\sum_{i=1}^{\langle Z_t,1\rangle}\delta_{(x_i(t),a_i(t))}.\label{defz}
  \end{equation}
  Each individual is represented by a Dirac mass on its trait and age,
   and $\langle  Z_t,1\rangle $  is  the population size at time $t$.
We denote by $\mathcal{M}_F(\widetilde{\X})$ the set of finite
measures on $\widetilde{\X}$. For $Z\in \mathcal{M}_F(\widetilde{\X})$ and for
a real-valued bounded measurable function  $f$,
 we define $\langle Z,f\rangle=\int_{\widetilde{\X}}fdZ$.

\noindent
 An individual with trait $x\in \X$
and age $a\in \R_+$ in the population described by $\ Z$
 gives birth to a new individual with rate $b(x,a)$.
 With probability $p\in [0,1]$ the new individual is a mutant of age 0 and trait $x+h$, where $h$ is drawn
 from a probability distribution $\ k(x,h)\, dh$ with support on $\X-\{x\}=\{y-x\,|\,y\in \X\}$ (so that $x+h\in \X$).
  With probability $1-p\in [0,1]$, the new individual is clonal, with age 0 and trait $x$.
The individual dies with rate $d(x,a)+ Z U(x,a)$ where
$d(x,a)$ is the natural death rate and $Z
U(x,a)=\int_{\widetilde{\X}}U((x,a),(y,\alpha))Z(dy,d\alpha)$.
Here, $U$ is a real-valued kernel describing the interaction
 exerted by an individual with trait $y$ and age $\alpha$ on an individual with trait $x$
 and age $a$. During their life, individuals age with
velocity 1, so that  the age at time $t\in \R_+$ of an individual
born at time $c\leq t$ is $a=t-c$. (Notice that the birth date $c$
varies from an individual to another).

\noindent Let us now describe the generator of the
$\mathcal{M}_F(\widetilde{\X})$-valued Markov process $Z$. This
generator $\ L\ $ sums the aging phenomenon and  the ecological
dynamics of the population. Let
$\mathcal{C}_b^{0,1}(\widetilde{\mathcal{X}},\R)$
  be the space of continuous bounded real-valued functions on $\widetilde{\X}$
 with bounded continuous derivatives with respect to the age variable.
  As developed in Dawson \cite{dawson}
Theorem 3.2.6, the set of cylindrical functions defined for each
$\mu\in \mathcal{M}_F(\widetilde{\X})$ by $F_f(\mu)=F(\langle
\mu,f\rangle)$, with $F\in C^1_b(\R)$ and $f\in
C^{0,1}_b(\widetilde{\X})$, generates the set of bounded
measurable functions on $\mathcal{M}_F(\widetilde{\X})$. For such function,
\begin{multline}
  LF_f(\mu)=
   \langle \mu,  \partial_a f(.)\rangle F'_f(\mu)\\
+
 \int_{\widetilde{\X}} \left[(d(x,a)+\mu U(x,a))
  (F_f(\mu-\delta_{(x,a)})-F_f(\mu))+ b(x,a)(1-p)
  (F_f(\mu+\delta_{(x,0)})-F_f(\mu))\right.\\
  +  \left. b(x,a)p
 \int_{\R^d}(F_f(\mu+\delta_{(x+h,0)})-F_f(\mu))k(x,h)dh\right]\mu(dx,da).
\end{multline}

\noindent In the sequel, we make the following assumption.
\begin{hyp}\label{hypothesetaux} {\it The functions $\ b$, $d$ and $U$
are assumed to be of class $\Co^1$ and $x\mapsto k(x,h)$ is
Lipschitz continuous, uniformly in $h$. We suppose that
 there exist strictly positive constants $\bar{b},\,\underline{d},\,\bar{d},\,\bar{U},\,\underline{U}$
  such that $\forall (x,a)\in\widetilde{\X}$,
 $\forall (y,\alpha)\in \widetilde{\X}$,
$$0  \leq  b(x,a)  \leq    \bar{b}, \quad
\underline{d}  \leq d(x,a)  \leq \bar{d}, \quad\underline{U}
 \leq   U((x,a),(y,\alpha))  \leq  \bar{U},
\quad \int_{\mathbb{R}^d} (\sup_x k(x,h))dh  <+\infty.$$}
\end{hyp}

\bigskip

\noindent Under Assumption \ref{hypothesetaux}, the process $\ Z$
can be obtained as the unique strong solution of a stochastic
differential equation driven by a multivariate Poisson point
measure corresponding to the dynamics described above (we can
adapt Fournier Méléard \cite{fourniermeleard} and Tran
\cite{trangdesdev} concerning respectively trait-structured and
age-structured cases).

\bigskip
\noindent Let us mention different forms for the birth rate that can be found in the literature. The period devoted to
reproduction is very different
 from one specie to another.
 For some
species, there is no reproduction in the first
  growth period (as juvenile or larva states, cf. Henson
  \cite{henson}).
 The parameters of interest are the maturity age $a_M$ and the constant  birth rate $b_1$,
 which are
 the heritable traits. The birth rate for an individual with traits $(a_M,b_1)$
 and age $a$ can be $b(a_M,b_1,a)= b_1 {\bf 1}_{\{a\geq a_M\}}.$\\
\noindent  Many other species reproduce early in life, and the reproduction function is decreasing with age. The heritable traits
   submitted to mutation are then  the  initial
  reproduction
  rate $b_1$ and the senescence parameter $b_2$. As in (cf. Webb \cite{webb} p.39 Ex. 2.1),
    the birth rate of an
  individual with traits $(b_1,b_2)$ and age $a$
  can be $b(b_1,b_2,a)=b_1 e^{- b_2 a}.$\\
\noindent Conversely, for other species (\cite{webb} p.41 Ex. 2.2), reproduction happens at the end of life and the birth rate
of an individual with age $a$ can be modelled as
$b(b_1,b_2,a)=b_1 (1 - e^{- b_2 a}).$ The individual trait will be $(b_1,b_2)$, with $\ b_1$ the
reproduction rate at maturation and $b_2$ the maturation
parameter.

\medskip
\noindent
 The competition kernel $\ U$ describes the
competition intensity between two individuals and can depend on
their respective ages and traits. The simplest case is the
(density dependent) {\it logistic} one, for which $\
U((x,a),(y,\alpha))=\eta(x,a)$. The interaction exerted on an
individual is then proportional to the total number of individuals
in the population, and thus nonlocal. Another interesting case is
when $U((x,a),(y,\alpha))=\eta(x,a)U(x,y)$.
Then
 interaction is nonlocal in age and local in the trait space.
  Such population will be called {\it age-logistic}.

\bigskip
\noindent In Sections 5 and 6, we develop examples of asexual age and size-structured populations. We first consider a logistic population with senescence in the birth
rate. Computation can be developed in details. We compare it with
the corresponding model without senescence and with a model with
age-logistic interaction.  In the last model, we
investigate a case for which the competition kernel is local in
trait and age.

\medskip
\noindent In these examples, the trait space is the size space
$[0,4]$ and the size trait is heritable except  when a mutation
occurs. With probability $p\in ]0,1[$, the descendant of an
individual with trait $x$ is a mutant with trait $x+h$. Mutation
amplitude $h$ is distributed on $[-x,4-x]$ following the centered
law $k(x,h)dh$, so that $x+h$ remains in $[0,4]$. With probability
$1-p$, the descendant is clonal.

\bigskip
\noindent {\bf Example 1:}
 The individuals reproduce at
rate
\begin{equation}
b(x,a)=x(4-x)e^{-a}.\label{tauxnaissanceexemple}
\end{equation}
Dependence in $x$ explains as follows: if the individual is small,
it has not enough energy
 to reproduce with high rate. If
 it is large, descendants' size will also be large by  hereditary
 transmission,
 and more energy is required to produce such descendants than smaller ones.
 The term $e^{- a}$ expresses a senescence phenomenon and the reproduction rate is  decreasing
 with age.
The logistic death
  rate of an individual with trait $x$ is given whatever its age  by
\begin{equation}
d(x,\langle Z,1\rangle)=\frac{1}{4}+0.001\cdot(4-x)\langle
Z,1\rangle.\label{tauxmortexemple}
\end{equation}The term 1/4 is the natural death rate and  the
logistic  term has intensity $0.001 (4-x)$, meaning that big
individuals are less sensitive to competition. The trait
maximizing the birth rate is $x=2$. Since the logistic
competition favors individuals with size close to  4, we can
expect that the optimal traits, taking birth and death rates into
account, are lying in the interval [2,4].

\noindent In Section 5,  computation will allow us to obtain the
value  of the trait equilibrium (evolutionary stable strategy) in
which the population will stabilize in the evolution time scale.
We will compare this model with the similar  one with
$b(x)=x(4-x)$, (a size-structured population without age
structure) and will show that the senescence phenomenon appears as
a penalization factor, reducing the equilibrium trait value.
We also compare the first model with the one in which the logistic interaction term $0.001(4-x)\langle Z,1\rangle$ is replaced by
\begin{equation}a\int_{\widetilde{\X}}U(x,y)Z(dy,d\alpha),\quad \mbox{ where }\quad  U(x,y)= C\left(1-\frac{1}
{1+\nu\exp\left(-k(x-y)\right)}\right)\label{kisdi}\end{equation} is
Kisdi's asymmetric competition function \cite{kisdi} ($C=0.002$, $\nu=1.2$ and $k=4$). Its sigmoïd
shape models the fact that the competition on an individual is
mainly exerted by larger ones. We
see how intricate computation can be.\\

\begin{figure}[!ht]
\hspace{-0.5cm}\begin{tabular}[!ht]{ccc}
(a) & (b) & (c)\\
\begin{minipage}[b]{.33\textwidth}\centering
\vspace{0.3cm}
\includegraphics[width=0.8\textwidth,height=0.20\textheight,angle=0,trim=1cm 2cm 1cm 1cm]{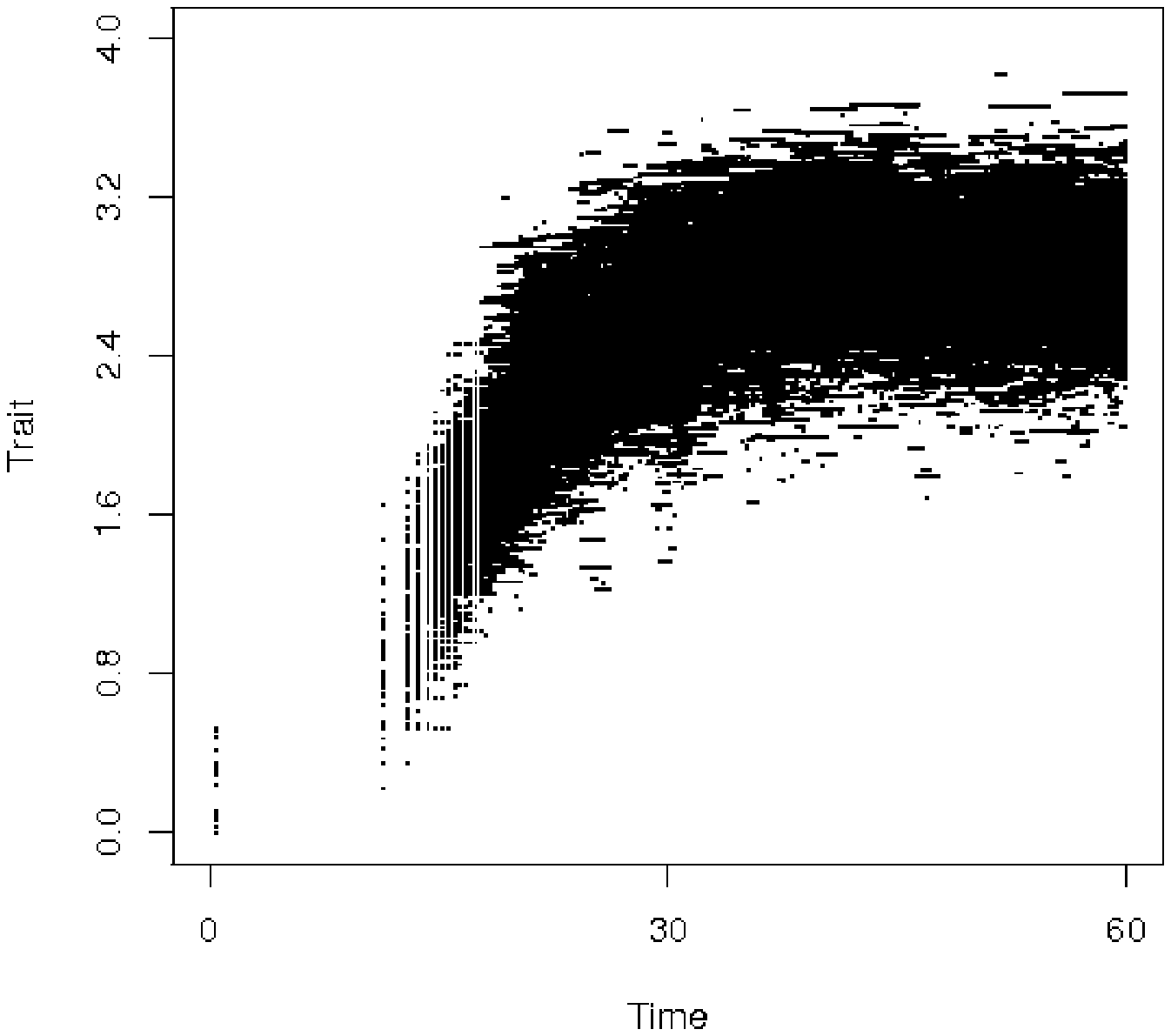}
\end{minipage}
 &
\begin{minipage}[b]{.33\textwidth}\centering
\vspace{0.3cm}
\includegraphics[width=0.8\textwidth,height=0.20\textheight,angle=0,trim=1cm 2cm 1cm 1cm]{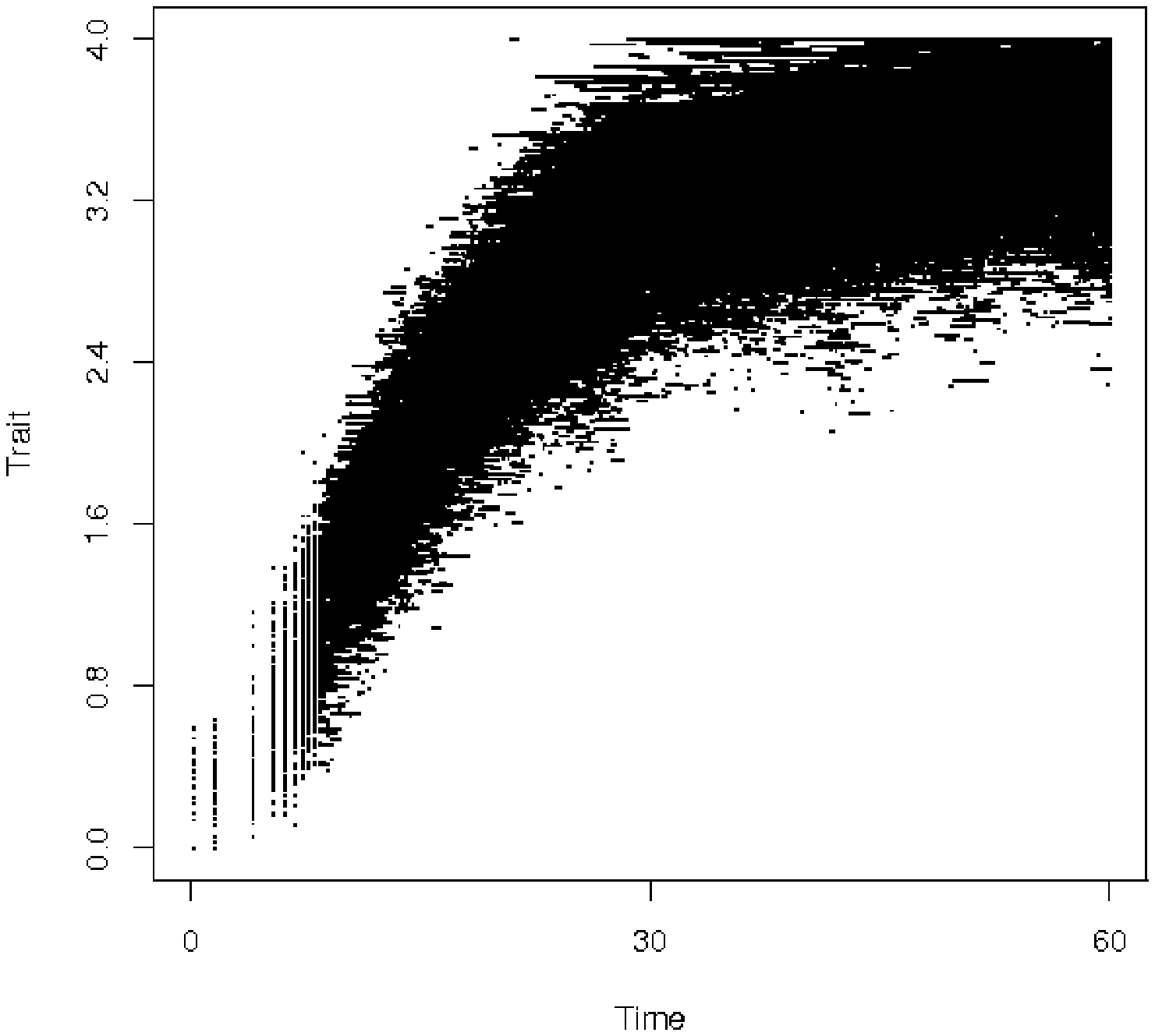}
\end{minipage}
 &
\begin{minipage}[b]{.33\textwidth}\centering
\vspace{0.3cm}
\includegraphics[width=0.8\textwidth,height=0.20\textheight,angle=0,trim=1cm 2cm 1cm 1cm]{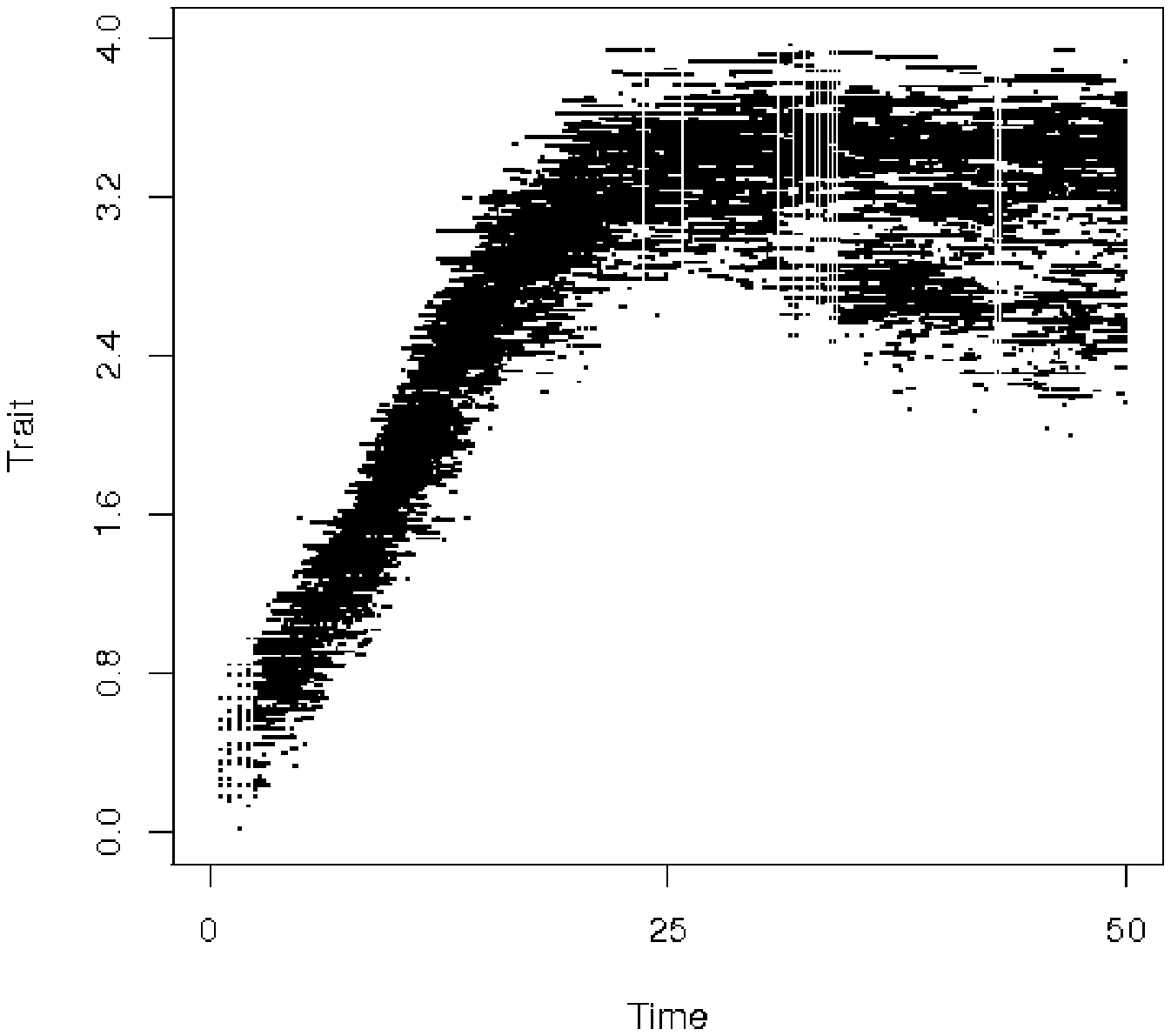}
\end{minipage}
\end{tabular}
\caption{\textit{Simulation of the microscopic process, using an individual-based algorithm: the traits in the population are represented in function of time. We start with
 $n=2000$ individuals. The initial traits are drawn uniformly between 0 and 1.3, and the ages are drawn in an exponential distribution of parameter 2. For the mutations, $p=0.13$ and $k$ is the gaussian
law with mean $0$ and variance $0.15$, conditioned to stay in
$[0,4]$. (a): Logistic population with age structure (\ref{tauxnaissanceexemple}, \ref{tauxmortexemple}). (b): Logistic population without age structure and (\ref{tauxmortexemple}). (c): Age-structured population with nonlocal interaction kernel (Example 2).}}\label{figmicrosc}
\end{figure}

\noindent {\bf Example 2:} We investigate a model where the competition
kernel is a function of the traits and ages of the competitors. An individual with trait $x$
and age $a$
  reproduces at rate $b(x)=x(4-x).$
 There is no natural death  and the action
exerted by a competitor of trait $y$ and age $\alpha$   is
described by  the interaction  function $\eta(a)
\left(1+e^{-\alpha}\right) U(x,y)$ of {\it separable
multiplicative form}, with $U$ defined in \eqref{kisdi}.   The
term  $\ 1+e^{-\alpha}\ $ is a senescence parameter diminishing
 the interaction intensity with the competitor's
age. When $\eta$ is a constant function, we recover models studied
by \cite{busenbergiannelli,perthameryzhik} with a deterministic
approach.
 In Section 6, we investigate the case
$ \eta(a)=a$. \\

\noindent Simulations using an individual-based algorithm are given in Figure \ref{figmicrosc} (c). They present the trait-support of the point measure Z defined in (\ref{defz}) for the three examples previously described. We remark a qualitative difference between cases (a) and (b) and case (c). In the two first plots, the trait-support is very dense. Conversely, in Figure \ref{figmicrosc} (c), the population separates into
smaller groups once it has reached a certain trait value which
seems to be approximatively 3.2. This phenomenon is known as
\textit{branching} in the biological literature.

\subsection{Large population and rare mutation renormalizations}

We are interested in studying approximations of the age-trait
dynamics presented in Section \ref{sectiondescription} under a large initial population and  rare mutation
assumption.

\medskip
\noindent The processes of interest are the renormalized population
processes $(Z^n)_{n\in \N^*}$
  given by:
\begin{equation}\label{processusmicrorenorm}\forall n\in \N^*,\,\forall t\in \R_+,\,
Z^n_t(dx,da)=\frac{1}{n}\sum_{i=1}^{n\langle
Z^n_t,1\rangle}\delta_{(x_i(t),a_i(t))}.
\end{equation}
The assumptions are as follows.
 \medskip\noindent
\begin{hyp}\label{hypconditioninitcvedp}{\it
(i) $\exists \varepsilon>0,\quad\sup_{n\in
\N^*}\mathbb{E}\left(\langle
Z^n_0,1\rangle^{2+\varepsilon}\right)< +\infty$.
\\
(ii) The sequence $Z^n_0$ converges in law in
$\mathcal{M}_F(\widetilde{\X})$ to a finite measure $\xi_0$.\\
(iii) The  ecological
 parameters stay unchanged, except the interaction kernel $U_n={U\over n}$. This
 hypothesis corresponds heuristically to the expression of
 a resource constraint. The death rate of an individual $(x,a)$ of $Z^n_t$
  is thus equal to
  $d(x,a)+ Z^n_t U(x,a)$.
 \\
 (iv) The mutation probability is given by $u_n p$, with $u_n,\ p \in
[0,1]$ and $\ u_n$ is assumed to decrease to $0$ when $n \to
+\infty$.}
\end{hyp}

\bigskip




  \noindent Adapting
tightness-compactness results in Fournier and Méléard
\cite{fourniermeleard}, one can prove that  the sequence $(Z^n)_{n\in \N^*}$ converges to a deterministic limit  when $n\to
+\infty$. More precisely:

\begin{prop}\label{prop26}  (cf. Tran \cite{chithese} (Section 3.2)) Under Assumptions \ref{hypothesetaux} and
\ref{hypconditioninitcvedp}, the sequence  $(Z^n)_{n\in \N^*}$
converges in $\mathbb{D}(\R_+,\mathcal{M}_F(\widetilde{\X}))$ to a
deterministic continuous process $\xi\in
\Co(\R_+,\mathcal{M}_F(\widetilde{\X}))$ characterized as the
unique solution of the evolution equation: $\forall f\in
\Co_b^{0,1}(\widetilde{\X},\R)$, $\forall t\in \R_+$,
\begin{multline}
\langle \xi_t,f\rangle= \langle \xi_0,f\rangle +   \int_0^t \int_{\widetilde{\X}}
\left[\partial_{a} f(x,a)+ f(x,0) b(x,a)
-  f(x,a)(d(x,a)+\xi_s U(x,a))\right]\xi_s(dx,da)\, ds,\label{formefaible}
\end{multline}
\end{prop}
\noindent Let us roughly give the main ideas of the proof. Since
every individual is weighted by $1/n$, the size of population
jumps converges to zero, and the limiting values of $(Z^n)_{n\in \N^*}$ are
continuous processes. Moreover, for $f\in \Co_b^{0,1}(\widetilde{\X},\R)$, $\langle Z^n_t,f\rangle$ writes as the sum of a
finite variation term and a random term whose square is of order
$1/n$. The random part
disappears when $n\to +\infty$ and the limit is deterministic. In
the finite variation part, mutations also disappear in the limit,
since $(u_n)_{n\in \N^*}$ tend to $0$, leading to \eqref{formefaible}. No
diversity appears since there is no mutation. Uniqueness of the solution of (\ref{formefaible}) implies uniqueness of the limiting value of $(Z^n)_{n\in \N^*}$.

\medskip
\noindent  In cases where the support of the trait-marginal of
$\xi_0$ is a singleton $\{x_0\}$, respectively a pair $\{x_0,y\}$, the population remains monomorphic, respectively dimorphic. Equation
\eqref{formefaible} is then parameterized by these values and
relations between (\ref{formefaible}) and classical partial
differential equations have been studied by Tran
(\cite{trangdesdev}
 Propositions 3.4 and 3.6).

\begin{prop}\label{propconvergence1chap4} (i) Assume that  $\xi_0(dx,da)=\delta_{x_0}(dx)
m_0(x_0,a)da$. Then for every $t\geq 0$,
$\xi_t(dx,da)=\delta_{x_0}(dx)m(x_0,a,t)da\ $  and the function $\
m(x_0,a,t)\ $  is the unique weak
 function solution of the  partial differential equation
 parameterized by $x_0$:
\begin{align}\frac{\partial m}{\partial t}(x_0,a,t) +& \frac{\partial m}{\partial a}
(x_0,a,t) = - \left(d(x_0,a)+\int_{\R_+} U((x_0,a),(x_0,\alpha))
m(x_0,\alpha,t)d\alpha\right)
m(x_0,a,t),\nonumber\\
m(x_0,0,t) = & \int_0^{+\infty} b(x_0,a)m(x_0,a,t)da,\quad m(x_0,a,0)=m_0(x_0,a).\label{pde2}
\end{align}
(ii) Assume that  $\ \xi_0(dx,da)=\delta_{x_0}(dx) m_0(x_0,a)da
+\delta_{y}(dx) m_0(y,a)da$. Then for every $t\geq 0$,
  $$\xi_t(dx,da)=\delta_{x_0}(dx) m(x_0,a,t)da
+\delta_{y}(dx) m(y,a,t)da$$ where $\ (m(x_0,a,t),m(y,a,t))\ $ is
the unique weak solution of
\begin{align} & \frac{\partial m}{\partial t}(x_0,a,t) + \frac{\partial m}{\partial a}
(x_0,a,t) = \nonumber\\
 & -   \left(d(x_0,a)+\int_{\R_+} \left(U((x_0,a),(x_0,\alpha))
m(x_0,\alpha,t)+U((x_0,a),(y,\alpha))
m(y,\alpha,t)\right)d\alpha\right)
m(x_0,a,t),\nonumber \\
 & \frac{\partial m}{\partial t}(y,a,t) + \frac{\partial m}{\partial
a} (y,a,t) = \nonumber\\
& - \left(d(y,a)+\int_{\R_+} \left(U((y,a),(x_0,\alpha))
m(x_0,\alpha,t)+U((y,a),(y,\alpha))
m(y,\alpha,t)\right)d\alpha\right)
m(y,a,t),\nonumber \\
 & m(x_0,0,t) =  \int_0^{+\infty} b(x_0,a)m(x_0,a,t)da,\quad
m(y,0,t) =  \int_0^{+\infty} b(y,a)m(y,a,t)da, \nonumber\\
 & m(x_0,a,0)= m_0(x_0,a),\quad m(y,a,0)=m_0(y,a).\label{pde3}
 \end{align}
\end{prop}

\medskip \noindent
Equation \eqref{pde2}  generalizes Demography equations, as
 McKendrick-Von Foerster or Gurtin-McCamy Equations
  (see \cite{mckendrick, vonfoerster, gurtinmaccamy}).

    \bigskip
\noindent In the sequel, we will make the following assumptions
concerning the  long time behavior of the solutions of
\eqref{pde2} and \eqref{pde3}.

\begin{hyp}\label{hypsolutionstationnaire}

 {\it  Let $x_0 \in \X$ and $m_0(x_0,a)da\in
\mathcal{M}_F(\mathbb{R}_+)$. The solution of (\ref{pde2}) admits
a unique nontrivial stable stationary solution
$\widehat{m}(x_0,a)$ such that
 $m(x_0,a,t)da$ converges  for the  weak convergence topology in
 $\mathcal{M}_F(
 {\mathbb{R}_+})$ to
 $\widehat{m}(x_0,a)da$ when $t\rightarrow +\infty$.
We will denote by
\begin{equation} \label{mchap}\widehat{M}_{x_0}=\int_0^{+\infty}\widehat{m}(x_0,a)da\end{equation}
the mass of the stationary age measure and by $\widehat{\xi}_{x_0}$
the space-time stationary measure
$\widehat{\xi}_{x_0}(dx,da)=\delta_{x_0}(dx)\widehat{m}(x_0,a)da$.}
\end{hyp}

\bigskip
\noindent
 Under Assumption
\ref{hypsolutionstationnaire}, $\widehat{m}(x_0,a)$ is solution of
\begin{equation}{\partial \widehat{m}\over \partial
a}(x_0,a)= - \widehat{d}(x_0,a,x_0)\widehat{m}(x_0,a),\quad \widehat{d}(x_0,a,x_0):=d(x_0,a)+\int_{R_+}U((x_0,a),(x_0,\alpha))
\widehat{m}(x_0,\alpha)d\alpha,\label{solstationnaire}
\end{equation}with the boundary condition:
\begin{equation}
\widehat{m}(x_0,0)=\int_0^{+\infty}
b(x_0,a)\widehat{m}(x_0,a)da.\label{boundarycondition}\end{equation}Solutions to (\ref{solstationnaire}) have the form
\begin{equation}
\widehat{m}(x_0,a)=\widehat{m}(x_0,0)\exp\left(-\int_0^a \widehat{d}(x_0,\alpha,x_0)d\alpha\right).\label{formesolutionsstationnaires}
\end{equation}From (\ref{formesolutionsstationnaires}) and (\ref{boundarycondition}), we obtain the following balance condition
\begin{equation}\int_0^{+\infty}b(x_0,a)e^{-\int_0^{+\infty}
\widehat{d}(x_0,\alpha,x_0)d\alpha}da=1.\label{eqbalance}\end{equation}

\noindent A necessary condition to get Assumption
\ref{hypsolutionstationnaire} is then that
\begin{equation}
R_0(x_0):=\int_0^{+\infty}b(x_0,a)e^{-\int_0^a
d(x_0,\alpha)d\alpha}da>1.\label{cssurvie}
\end{equation}

\noindent The term $R_0(x_0)$, called  net reproduction rate, is
the integral in age of the birth rate weighted by the survival
probability in absence of competition. It's  the well known
threshold between  sub and super-criticality in age-structured
models (see \cite{athreyaney, doney, gurtinmaccamy, webb}). If the birth and death rates $b(x)$ and $d(x)$ do not
depend on age, $R_0(x)=b(x)/d(x)> 1$ if and only if $b(x)>d(x)$,
which is the standard super-criticality condition.

\bigskip
\noindent With the notation of Section \ref{sectiondescription}, in the age-logistic case and for a monomorphic population with trait $x_0$, the death
rate in (\ref{pde2}) equals
$d(x_0,a)+\eta(x_0,a)U(x_0,x_0)\int_{\mathbb{R}_+}
m(x_0,\alpha,t)d\alpha$. Equation (\ref{pde2}) is then a
Gurtin-McCamy equation parameterized by $x_0$ and Condition
\eqref{cssurvie} is also sufficient  to obtain Assumption
\ref{hypsolutionstationnaire} as proved in Webb \cite{webb},
Section 5.4.

\medskip
\noindent Let us now introduce assumptions describing the
"invasion implies fixation" principle, for the age-dependent
system defined by (\ref{pde3}). Assertion (ii) impedes the
co-existence of the two traits $x_0$ and $y$ in the long time.

\begin{hyp}\label{dimstan}
{\it (i)  Let $x_0, y\in \X$ with $R_0(x_0)>1$. One of the
following assumptions holds:
\begin{align}
\mbox{Either: }\int_{\R_+}b(x_0,a)e^{-\int_0^a
\widehat{d}(x_0,\alpha,y)d\alpha}da>1 \mbox{ and }
\int_{\R_+}b(y,a)e^{-\int_0^a \widehat{d}(y,\alpha,x_0)d\alpha}da<1\label{hypnoncoex1}\\
\mbox{Or: }\int_{\R_+}b(x_0,a)e^{-\int_0^a
\widehat{d}(x_0,\alpha,y)d\alpha}da<1 \mbox{ and }
\int_{\R_+}b(y,a)e^{-\int_0^a
\widehat{d}(y,\alpha,x_0)d\alpha}da>1,\label{hypnoncoex2}
\end{align}
where \begin{equation}
\widehat{d}(x,a,y):=d(x,a)+\int_0^{+\infty}U((x,a),(y,\alpha))
\widehat{m}(y,\alpha)d\alpha\label{simplifmortgelee}
\end{equation}is the death rate of an individual $(x,a)$ in the population
at equilibrium $\widehat{\xi}_y$.
\medskip

\noindent (ii) The solution of (\ref{formefaible}) with the
dimorphic initial condition  $\xi_0(dx,da)
=\delta_{x_0}(dx)m_0(x_0,a)da+\delta_{y}(dx)m_0(y,a)da$ converges
as $t\to \infty$ to $\widehat{\xi}_{x_0}$ if (\ref{hypnoncoex1})
is satisfied and to
 $\widehat{\xi}_y$ if (\ref{hypnoncoex2}) is satisfied (for the weak convergence topology in
$\mathcal{M}_F(\widetilde{\X})$).}
\end{hyp}

\bigskip
\noindent For a logistic dimorphic population with traits $x_0$ and $y$ ($R_0(x_0)>1$),
Assumption \ref{dimstan}
 is fulfilled as soon as $U(x_0,x_0)U(y,y)-U(x_0,y)U(y,x_0)\leq 0$
and one of the following assertion is satisfied:
\begin{align*}
 &
U(x_0,x_0)\widehat{M}_{x_0}-U(x_0,y)\widehat{M}_y<0 \quad \mbox{
and }\quad U(y,y)\widehat{M}_y-U(y,x_0)\widehat{M}_{x_0}>0,
\nonumber\\
\mbox{ or } & U(x_0,x_0)\widehat{M}_{x_0}-U(x_0,y)\widehat{M}_y>0
\quad \mbox{ and }\quad U(y,y)\widehat{M}_y-U(y,x_0)
\widehat{M}_{x_0}<0.
\end{align*}
(Let us recall that $\widehat{M}_{x_0}$ has been defined in
\eqref{mchap}).

\section{Age Structured Trait Substitution Sequence Process}\label{sectionpreztss}

\noindent In order to obtain the Adaptive Dynamics limit, we consider
here the same  mutation scale as in the work of Champagnat
\cite{champagnat3} without age-structure. We assume that
\begin{equation}\label{equationun}
\forall\  V>0,\quad  \exp(-Vn)=o(u_n)\quad \mbox{ and
}u_n=o\left(\frac{1}{n \log n}\right).
\end{equation} This scaling is here again the right one
to derive from the microscopic process a jump process generalizing the
 Trait Substitution Sequence (see Metz \textit{et al.} \cite{metzgeritzmeszenajacobsheerwaarden}, Champagnat \cite{champagnat3}) for age-structured populations. It is obtained from a fine study of the different time steps
in the invasion process of the resident population by a mutant one. This study has been done by Champagnat in \cite{champagnat3} and adapted to
 include age structure in Tran's thesis \cite{chithese}. Let us roughly recall the
main ideas justifying \eqref{equationun}. They are summarized in Figure \ref{fig1}. If no mutation occurs, then with a probability that tends to 1 when $n\rightarrow +\infty$, a monomorphic population with initial size $n$ enters a given neighborhood of its equilibrium after a sufficiently large time that does not depend on $n$. By large deviation results, the microscopic process stays in this neighborhood during an exponential time $e^{nV}$.
Mutation time is of order $1/(nu_n)$, and an invasion period is proved to be of order $\log n$.  Thus the condition $u_n=o(1/(n\log n))$ in (\ref{equationun}) implies that the mutations are sufficiently rare so that
  the population has returned to its monomorphic equilibrium when the next mutation occurs.
   The condition $\exp(-Vn)=o(u_n)$ tells us that the mutations occur sufficiently often so that
    a mutant appears before a rare event drives the resident population
    far from its equilibrium. Hence, if we change time and consider $(Z^n_{./(nu_n)})_{n\in
  \N^*}$ at the mutation time scale, we will obtain in the limit, when $n\rightarrow +\infty$, a process where the transition periods have disappeared and where only the sequence of equilibrium states remains.
   Under our hypotheses preventing the trait co-existence, the latter
   have a singleton trait-support and then reduce to age-measures.
 Therefore, the limiting  process jumps from an age-measure to
  another, each of these parameterized by a trait. This provides a generalization of the TSS proposed
  by Metz \textit{et al.} \cite{metzgeritzmeszenajacobsheerwaarden}. The result is stated in Theorem \ref{theoremtss} and proved in Appendix A.

\begin{figure}[ht]
 \begin{center}
  \begin{picture}(350,180)(-20,-10)
   \put(0,0){\vector(1,0){320}} \put(0,0){\vector(0,1){150}}
  \put(-10,-10){0}
 \put(-20,155){population size}
\put(310,-10){time}
      \qbezier[27](0,30)(25,32)(40,60)
     \qbezier[27](40,60)(55,115)(80,120)
     \dottedline(0,120)(320,120)
     \put(-25,117){$\langle \widehat{\xi},1\rangle$}
     \qbezier[27](0,20)(27,30)(40,50)
    \qbezier[27](40,50)(55,105)(80,110)
   \dottedline(80,110)(320,110)
   \qbezier[27](0,40)(25,42)(35,70)
  \qbezier[27](35,70)(55,125)(80,130)
 \dottedline(80,130)(320,130)
\dottedline{0.5}(0,30)(4,33)(5,30)(9,35)(12,32)(15,30)(18,35)(20,40)(21,43)(24,40)(26,47)(27,46)(29,50)(33,52)(35,54)
(38,59)(40,63)(43,68)(45,67)(47,75)(48,74)%
(50,80)(52,87)(55,93)(56,98)(57,97)(58,105)(59,110)(60,109)(61,114)(64,111)%
(66,112)(69,114)(71,120)(72,118)(75,117)(77,120)(80,121)(81,122)(83,124)(86,118)%
(88,120)(90,116)(91,115)(93,124)(95,122)(98,123)(100,120)(101,116)(103,118)(106,120)%
(107,124)(109,126)(110,125)(112,121)(113,117)(115,118)(118,120)(120,119)(123,122)%
(125,124)(126,123)(128,125)(129,120)(132,112)(134,113)(136,105)(138,98)(141,90)%
(143,91)(145,80)(148,82)(149,75)(151,67)(153,69)(154,57)(156,50)(158,53)%
(160,42)(161,44)(164,38)(166,40)(167,33)(170,34)(173,24)(175,26)(177,20)%
(181,15)(182,16)(184,10)(186,3)(189,4)(191,1)
(193,0)
\dottedline{3}(132,122)(134,123)(136,119)(138,122)(141,124)%
(143,118)(145,117)(148,121)(149,122)(151,118)(153,120)(154,119)(156,123)(158,116)%
(160,119)(161,124)(164,122)(166,123)(167,120)(170,124)(173,119)(175,118)(177,116)%
(181,119)(182,121)(184,119)(186,120)(189,122)(191,123)
(193,120)(196,118)(199,119)%
(202,120)(203,122)(205,120)(208,118)(210,119)(213,121)(215,122)(217,121)%
(220,123)(222,124)(224,123)(226,119)(229,118)(231,120)(233,121)(236,122)%
(239,119)(240,117)(242,118)(244,118)(247,121)(249,119)(252,122)(254,122)(255,123)%
(257,121)(258,121)(261,118)(262,116)(264,113)(267,110)(270,99)(272,100)%
(273,84)(275,89)(278,75)(279,69)(281,72)(283,57)(288,63)(290,56)
    \dottedline{3}(267,-40)(267,140)
    \put(268,-10){$\mathcal{T}^n$}
         \put(80,-35){\vector(1,0){184}}
        \put(200,-28){$e^{nV}$}
 \put(322,117){$\langle \xi_t,1\rangle$}
\put(292, 53){$\langle Z^n_t,1\rangle$}
 \dottedline{3}(80,-40)(80,140)
\dottedline{0.5}(110,0)(115,6)(118,11)(120,10)(122,13)(123,19)(125,18)(128,24)(130,29)(132,25)(134,32)(136,38)(137,37)(139,46)(142,49)(145,48)(147,56)(149,62)(153,60)(156,67)(157,70)(162,69)(166,78)(167,75)(170,84)(172,86)(175,85)
(178,92)(180,100)(183,97)(186,108)(188,106)(191,115)(193,120)(196,128)(200, 127)(203,135)(205,139)(207,140)(209,138)
(213,142)(215,141)
     \dottedline{3}(193,-8)(193,140)
    \dottedline{3}(110,-8)(110,140)
      \put(110,-6){\vector(1,0){83}}
    \put(150,-15){$\log n$}
   \put(95,-20){$\tau \sim \frac{C}{nu_n}$}
\end{picture}
 \end{center}
\vspace{0.5cm}
 \caption{{\small\textit{Large time behaviour of the microscopic process $Z^n$ and of its deterministic approximation $\xi$. On compact time intervals and for sufficiently large $n$, the behaviour of $Z^n$ follows the one of its approximation $\xi$. On a large time scale, if no mutation occurs, $Z^n$ leaves the neighborhood of the stationary equilibrium of $\xi$ at $\mathcal{T}^n$ after an exponentially long period and drives the population to extinction. However, before it gets extinct, the population can be invaded by a mutant in a time scale of order $1/nu_n$. Under the "Invasion implies fixation" Assumption \ref{dimstan}, if the mutant population does not die, it replaces the resident population after a transition period of order $\log n$.}}}\label{fig1}
\end{figure}
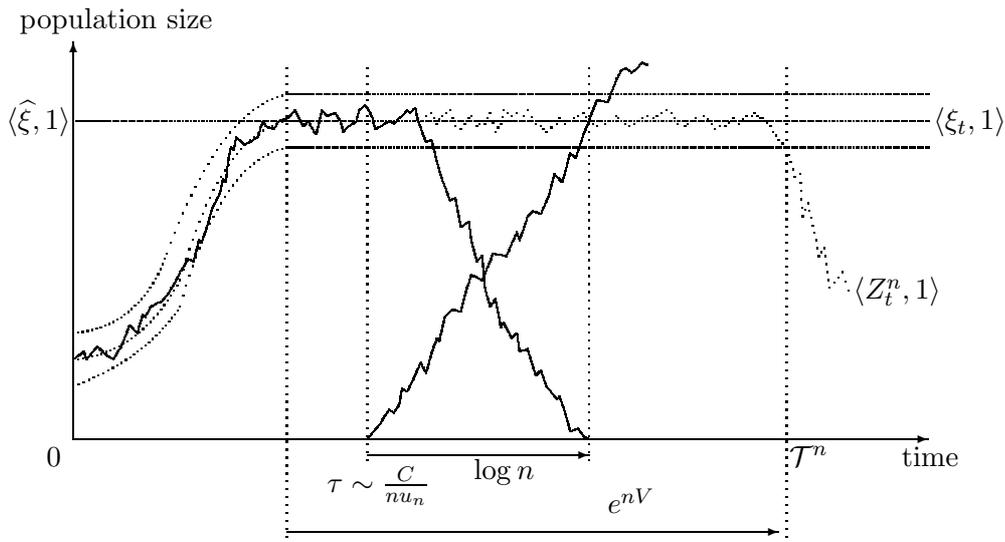

\begin{theorem}\label{theoremtss}Let us consider $(Z^n)$ defined by
 (\ref{processusmicrorenorm}) with a monomorphic trait support $\{X_0\}$  and satisfying
 Assumptions
\ref{hypothesetaux}, \ref{hypconditioninitcvedp} with
$\xi_0(dx,da)=\delta_{X_0}(dx)m_0(X_0,a)da$. Assume moreover that
$(u_n)_{n\in \N}$  satisfies (\ref{equationun}) and that Assumptions
\ref{hypsolutionstationnaire} and \ref{dimstan} are satisfied for
each $x_0,y\in \X$ with $R_0(x_0)>1$.
 Then, for every $t\in \R_+$,
 the sequence $(Z^n_{t/(nu_n)})_{n\in \N^*}$ converges in law
 in $\mathcal{M}_F(\widetilde{\X})$ to the measure $Z_t$ defined
 by
 \begin{equation} \label{TSS}
Z_t(dx,da)=\delta_{X_t}(dx)
\widehat{m}\left(X_t,a\right)da,\end{equation}
 where $X$ is a
Markov jump process
 of generator defined for $\phi\in \mathcal{B}_b(\X,\R)$, $x\in \X$ by
\begin{equation}\label{generateursaut}
L\phi(x)=\int_{\R^d}
\left[\left(\phi(x+h)-\phi(x)\right)
p\left(\int_{\R_+}b(x,a)\widehat{m}(x,a)da\right)(1-z_0(x+h,x))\right]k(x,h)dh,
\end{equation}where $z_0(y,x)$ is the smallest solution in $[0,1]$ of:
\begin{align}
 z  =  \int_0^{+\infty} e^{(z-1)\int_0^a b(y,u)du}\widehat{d}(y,a,x)e^{-\int_0^a
 \widehat{d}(y,u,x)du}da.
 \label{equationdeterminationsurvieprez}
\end{align} The death rate $\widehat{d}(y,a,x)$ has been defined in (\ref{simplifmortgelee}).
\end{theorem}

\begin{corol}\label{coroltss}The sequence $(Z^n_{./(nu_n)})_{n\in \N^*}$ converges to $Z$ in
 $\mathbb{D}(\R_+,\mathcal{M}_F(\widetilde{\X}))$ in the sense of finite-dimensional distributions.
\end{corol}

\noindent The process $X$ is called \textit{Age Structured Trait
Substitution Sequence Process}.

\medskip

\noindent We have already given the main ideas of Theorem \ref{theoremtss}'s
proof. Let
 us comment on the different terms of the generator (\ref{generateursaut}). \\
\textbf{1.} The term $ p\,k(x,h)
\int_0^{+\infty}b(x,a)\widehat{m}(x,a)da $ is the rate at which
the monomorphic
 population of trait $x$, in its equilibrium, generates a mutant of trait $x+h$. This term
 equals $p b(x) \widehat{M}_x k(x,h)$ in the absence of age-structure.\\
\textbf{2.} The term $(1-z_0(x+h,x))$ is called \textit{fitness}
of the mutant trait $x+h$ in the population of trait $x$ and is the probability that the mutant descendants invade the
resident population. In most cases,  this probability $z_0(x+h,x)$
can not be computed explicitly conversely to the case without
age-structure. Nevertheless, its implicit definition given in
Equation (\ref{equationdeterminationsurvieprez}) makes possible
its numerical computation as developed in the examples of
Sections \ref{sectionexemple} and \ref{sectionexemple2}.

\noindent Let us now focus on the way we establish the equation
defining the fitness function $z_0(y,x_0)$ where $x_0\in \X$ is
the resident trait and $y$ the mutant trait. When the mutant $y$ appears in the monomorphic resident
population with trait $x_0$, we can neglect in a first
approximation the mutant population and the deviations of the
resident population from its equilibrium. The mutant population is
then compared to a linear age-structured birth and death process
with parameters $b(y,a)$ and $\widehat{d}(y,a,x_0)$. For this process, the extinction probability is the smallest solution $z_0(y,x_0)$ in $[0,1]$ of (\ref{equationdeterminationsurvieprez}). Next
proposition proved in Appendix \ref{sectiontraitsubstitutionsequence} identifies
 $z_0(y,x_0)$ as the extinction probability of
the real mutant progeny in large population.

\begin{prop}\label{propextinction} Consider the process $Z^n$ as in Theorem
\ref{theoremtss} and starting from
 $Z^n_0(dx,da)=\delta_{x_0}(dx) q_0^n(x_0,da)+\frac{1}{n} \delta_{(y,0)}(dx,da)$, where $\delta_{x_0}(dx)q_0^n(x_0,da)$ is a point measure weighted by $1/n$ with support in $\{x_0\}\times \R_+$ and converging to $\widehat{\xi}_{x_0}$ in $\mathcal{M}_F(\widetilde{\X})$. We denote
 by $\mathbb{P}^n_{x_0,q^n_0,y}$ its law. Let $\tau$ be the first time at which a mutation
occurs,  $\theta$  the first time of return to a monomorphic
population and $V$ the survival trait at  time $\theta$. Then, under the assumptions of
Theorem \ref{theoremtss},\\
i)
$\lim_{n\rightarrow+\infty}\mathbb{P}^n_{x_0,q^n_0,y}\left(\theta<\tau,\,
 V=y\right)=1-z_0(y,x_0),$ \\
ii)
$\lim_{n\rightarrow+\infty}\mathbb{P}^n_{x_0,q^n_0,y}\left(\theta<\tau,\,
 V=x_0\right)=z_0(y,x_0),$
 where $z_0(y,x)$ is the extinction probability of a linear age-structured birth
 and death process
with parameters $b(y,a)$ and $\widehat{d}(y,a,x_0)$.
\end{prop}

\begin{figure}[ht]
\begin{center}
\begin{minipage}[b]{.46\textwidth}\centering
\begin{picture}(100,180)(-20,0)
   \put(-15,0){\vector(0,1){160}}
  \put(-10,165){time}
 \put(0,0){\line(0,1){88}}
\put(50,40){\line(0,1){51}}
    \put(20,70){\line(0,1){34}}
   \put(10,80){\line(0,1){10}}
  \put(30,98){\line(0,1){20}}
 \put(40,85){\line(0,1){48}}
\put(60,81){\line(0,1){68}}
    \put(70,68){\line(0,1){52}}
   \put(80,52){\line(0,1){12}}
  \dottedline{3}(0,40)(50,40)
 \dottedline{3}(0,70)(20,70)
\dottedline{3}(0,80)(10,80)
    \dottedline{3}(20,98)(30,98)
   \dottedline{3}(20,85)(40,85)
  \dottedline{3}(50,52)(80,52)
 \dottedline{3}(50,68)(70,68)
\dottedline{3}(50,81)(60,81)
\end{picture}
\vspace{-0.3cm}
\end{minipage}
\hspace{0cm}
\begin{minipage}[b]{.46\textwidth}\centering
\begin{pspicture}(1.5,0)
    \pstree[treemode=U]{\Tcircle{}}
            {
              \Tcircle{}
             \pstree{\Tcircle{}}
               {
                \Tcircle{}
             \Tcircle{}}
         \pstree{\Tcircle{}}
            {
             \Tcircle{}
            \Tcircle{}
           \Tcircle{}}}
\end{pspicture}
\vspace{0.5cm}
\end{minipage}
\end{center}
 \caption{{\small\textit{An age-structured process
in continuous time and its discrete time generation-tree.}}}\label{figuregaltonwatson}
\end{figure}
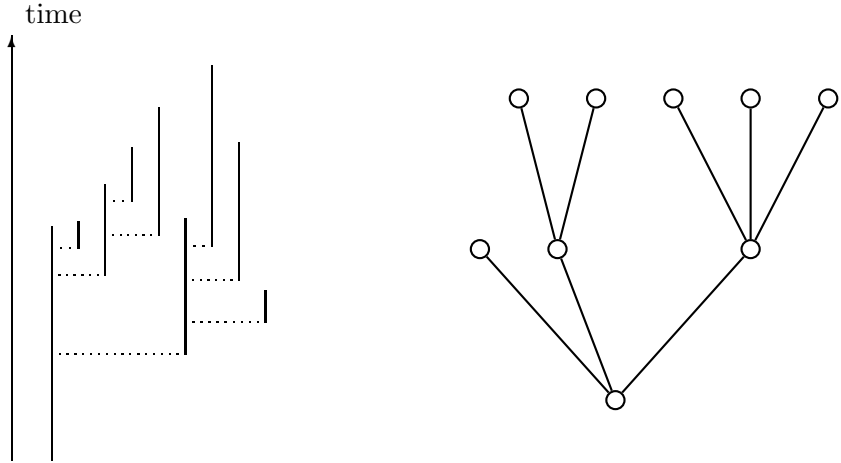

\noindent A proof of \eqref{equationdeterminationsurvieprez} is adapted
from Doney \cite{doney}. Noticing that
 the extinction of the mutant progeny for the linear
age-structured birth and death process is equivalent to the
extinction of
 the discrete-time Galton-Watson
 process corresponding to the underlying generation process
(see Figure \ref{figuregaltonwatson}), $z_0(y,x_0)$ is characterized as the smallest fixed point on
$[0,1]$ of the generating function of the number of children and
satisfies:
\begin{align}
 z  = & G(z)=\int_0^{+\infty} e^{(z-1)\int_0^a b(y,u)du}\widehat{d}(y,a,x_0)e^{-\int_0^a \widehat{d}(y,u,x_0)du}da
 \label{equationdeterminationsurvie}\\
= & 1+\int_0^{+\infty}(z-1)b(y,a)e^{(z-1)\int_0^a b(y,u)du-\int_0^a \widehat{d}(y,u,x_0)du}da,
\label{equationdeterminationsurvie2}
\end{align}
Equation (\ref{equationdeterminationsurvie2}) is obtained from
(\ref{equationdeterminationsurvie}) by integration by parts.
\medskip

\noindent A main issue is to know whether extinction happens almost surely
or not.
 \begin{prop}\label{proprappelg}Let $x_0,y\in \X$.  \\
(i) If \begin{equation}\label{conditionenoncezo}\int_0^{+\infty}
b(y,a)e^{-\int_0^a \widehat{d}(y,u,x_0)du}da> 1,\end{equation}then
$z_0(y,x_0)<1$ and defines a $\Co^1$
 function in both variables.\\
(ii) Else, $z_0(y,x_0)=1$.\\
(iii) Finally, when $y=x_0$, $ z_0(x_0,x_0)=1.$
\end{prop}

\noindent Of course, $z=1$ is an obvious solution of
 (\ref{equationdeterminationsurvie2}). Each other solution is a zero of $F$ where
 \begin{equation}\label{equationdeterminationsurvie3}
 F(z,y,x)=\int_0^{+\infty}b(y,a)e^{(z-1)\int_0^a b(y,u)du-
\int_0^a \widehat{d}(y,u,x_0)du}da-1. \end{equation}
 The proof is
therefore an immediate consequence of the following lemma.

\begin{lemme}\label{g}
For each pair of  traits $x,y$, the equation $\  F(z,y,x)=0$
admits a unique solution $z=g(y,x)$ on $\R_+$ of class $C^1$ in
both variables and such that $g(x,x)=1$.
\end{lemme}

\begin{proof}
The function $F$ is of class $\Co^1$ with a positive partial
derivative with respect to $z$. It is thus strictly increasing in $z$. Since $\int_0^{+\infty}b(y,a)e^{-\int_0^a b(y,u)du}da
 =e^{-\int_0^{+\infty}b(y,u)du}-1<0$, then $F(0,y,x)<0$.
Moreover $\lim_{z\rightarrow +\infty}F(z,y,x)=+\infty$. We deduce
that for every $x,y\in \X$, there is a unique solution $z=g(y,x)$
on $\R_+$ to $\  F(z,y,x)=0$
  which is equal
to 1 if and only if $\int_0^{+\infty} b(y,a)e^{-\int_0^a
\widehat{d}(y,u,x_0)du}da= 1$. This condition is true for $y=x_0$
by \eqref{eqbalance}. Using the Implicit Function Theorem we
obtain moreover that $g$ is of class $\Co^1$.
\end{proof}

\noindent In both cases, we can notice that:
\begin{equation}
z_0(y,x_0)=g(y,x_0)\wedge 1\mbox{ and }1-z_0(y,x_0)=\left[1-g(y,x_0)\right]_+,
\label{introductiong}
\end{equation}where $[.]_+$ denotes the positive part.

\begin{rque}\label{rquenedependpasdea}For a population without age-structure
 Equation (\ref{equationdeterminationsurvie}) can be solved
explicitly and we recover the infinitesimal generator
introduced by Metz et \textit{al.}
\cite{metzgeritzmeszenajacobsheerwaarden} with: \begin{equation}1-z_0(y,x_0)=
\left[\frac{b(y)-d(y)-U(y,x_0)\widehat{M}_{x_0}}{b(y)}\right]_+.\label{z0sansage}\end{equation}
\end{rque}

\section{Canonical Equation for an Age-structured population}\label{sectionprezec}

We are interested in the limit of the \textit{Age-structured Trait
Substitution Sequence Process} when the mutation step tends to
zero. Generalizing the approach of Dieckmann and Law
\cite{dieckmannlaw} and Champagnat \cite{champagnat} in the case
without age, we consider, for $\varepsilon>0$, the infinitesimal
generator $L^\varepsilon$ defined for every
$\phi\in\mathcal{B}_b(\X,\R)$ and $x\in\X$ by:
\begin{equation}\label{generateursautreormalise}
L^\varepsilon \phi(x) =\frac{1}{\varepsilon^2}\int_{\R^d}
\left(\phi(x+\varepsilon h)-\phi(x)\right) p\int_{\R_+}b(x,a)
\widehat{m}(x,a)da (1-z_0(x+\varepsilon
h,x))k(x,h)dh.\end{equation}

\noindent When $\varepsilon\rightarrow 0$, the sequence of such renormalized
TSS-processes converges to the solution of an ODE that generalizes
the Canonical Equation introduced by Diekmann and Law
\cite{dieckmannlaw}:

\begin{theorem}\label{theoremequationcanoniquefoncionnelle}
Under Assumptions \ref{hypothesetaux},
\ref{hypsolutionstationnaire}, \ref{dimstan}, the sequence
$(X^{\varepsilon})_{\varepsilon>0}$ converges in probability, for
the Skorohod topology on $\mathbb{D}\left(\R_+,\X\right)$ to the
solution of the following ODE:
\begin{equation}\label{equationcanonique}
\frac{dx}{dt}=p\int_{\R_+}b(x,a)\widehat{m}(x,a)da\int_{\R^d} h
\,D_h^1 z_0(x,x)\, k(x,h)dh,
\end{equation}\begin{align*}
\mbox{where }\,\,\,\frac{dx}{dt}=\lim_{h\rightarrow 0}\frac{|x(t+h)-x(t)|}{h},\quad
\mbox{ and }\quad D_{h}^{1}
z_0(x,x):=\lim_{\begin{array}{c}\varepsilon\rightarrow 0\\
\varepsilon>0\end{array}}\frac{z_0(x,x)-z_0(x+\varepsilon
h,x)}{\varepsilon}.\end{align*}
\end{theorem}

\noindent The Proof of Theorem \ref{theoremequationcanoniquefoncionnelle},
based on a tightness-uniqueness argument,  can be adapted from
Theorem 1 in Champagnat \cite{champagnat}.

\bigskip

\noindent Let us remark that $D_h^1z_0(x,x)$ is  nonnegative since
$z_0(x+\varepsilon h, x)\in [0,1]$ and  $z_0(x,x)=1$ by Proposition \ref{proprappelg}. Using (\ref{introductiong}), one gets
\begin{equation}
D_h^1z_0(x,x)=\left[\lim_{{\scriptsize\begin{array}{c}\varepsilon\rightarrow 0\\
\varepsilon>0\end{array}}}\frac{1-g(x+\varepsilon
h,x)}{\varepsilon}\right]_+.
\label{gradientfitness}\end{equation}Hence, as in the classical
case, the evolution follows the directions
 where the  fitness gradient $D_h^1z_0(x,x)$ is positive, and along which the extinction probability   $y\mapsto z_0(y,x)$ is a decreasing function.


\noindent Although $z_0$ is implicitly defined, an
explicit expression of $D_h^1z_0(x,x)$ can be
established.
\begin{prop} Let us consider the scalar case $d=1$ for sake of
simplicity. Under
Assumptions \ref{hypothesetaux}, \ref{hypsolutionstationnaire},
\ref{dimstan}, for $h>0$,  $D_1^hz_0(x,x)= h
[-\partial_1g(x,x)]_-$ and for $h<0$, $D_1^hz_0(x,x)= h
[-\partial_1g(x,x)]_+$, with $g$ defined in Lemma \ref{g}, and
\begin{multline}
\partial_1 g(x,x)=-\int_{\R_+} \left(\partial_1b(x,a)-b(x,a)
\int_0^a \left(\partial_1d(x,\alpha)+\int_{\R_+}\partial_1
U((x,\alpha),(x,u))\widehat{m}(x,u)du\right)d\alpha\right)\\
\times e^{-\int_0^a \widehat{d}(x,\alpha,x)d\alpha}da
\times \left(  \int_0^{+\infty} b(x,a)\left(\int_0^a b(x,\alpha)d\alpha \right)
e^{-\int_0^a \widehat{d}(x,\alpha,x) d\alpha}da\right)^{-1}.\label{resultatgradientfitness}
\end{multline}
\end{prop}

\begin{proof}We replace $y$ by $x+\varepsilon$ in (\ref{equationdeterminationsurvie3})
and consider the expansion with respect to $\varepsilon$, using:
\begin{align*}& b(x+\varepsilon,a)=  b(x,a)+\varepsilon \partial_1 b(x,a)+o(\varepsilon),\\
& d(x+\varepsilon ,a)=  d(x,a)+\varepsilon  \partial_1 d(x,a)+o(\varepsilon),\\
& U((x+\varepsilon,a),(x,\alpha))= U((x,a),(x,\alpha))+\varepsilon
\partial_1 U((x,a),(x,\alpha))+o(\varepsilon),\\
& g(x+\varepsilon ,x)=1+\varepsilon  \partial_1
g(x,x)+o(\varepsilon).\end{align*} Identifying the terms of order
$\varepsilon$, we obtain (\ref{resultatgradientfitness}).
\end{proof}

\begin{rque}
(i)  When $\partial_1 g(x,x)>0$, $D_1^1 z_0(x,x)=0$ and $D_{-1}^1 z_0(x,x)=\partial_1 g(x,x)$.
 When $\partial_1 g(x,x)<0$, $D_1^1 z_0(x,x)=-\partial_1g(x,x)$ and $D_{-1}^1 z_0(x,x)=0$.\\
(ii) In the scalar case without age-structure,
(\ref{resultatgradientfitness}) allows us to recover the
expression of the classical fitness gradient:
\begin{align*}
D_1^1z_0(x,x)= & \left[\frac{b'(x)(d(x)+U(x,x)\widehat{M}_x)-b(x)(\partial_1 d(x)+
\partial_2U(x,x)\widehat{M}_x)}{b(x)^2}\right]_+\nonumber\\
= & \left[\left.\partial_y\left(\frac{b(y)-(d(y)+U(y,x)\widehat{M}_x)}
{b(y)}\right)\right|_{y=x}\right]_+.
\end{align*}
\end{rque}

\section{Example 1}\label{sectionexemple}

We now present several examples for which we specify the adaptive dynamics approximations and study their behavior. We used \textbf{R} for simulations which illustrate our purpose, and {\sc Maple} for formal calculus, when computation becomes too technical.

\subsection{A logistic age and size-structured population}

We develop in this section a simple example highlighting the
difficulties that appear when considering a trait and age-structured population and for which
computations can be carried explicitly. The birth and death rates that are used have been specified and explained in (\ref{tauxnaissanceexemple}) and (\ref{tauxmortexemple}). Compared with models with only trait structure, the specificity here lies in the introduction of a senescence term. The latter, even if it is simple, introduces a notion of \textit{life history} that will have an effect on the traits selected through evolution.

\subsubsection{Monomorphic equilibrium in large populations}

\noindent The large population approximation is given by Proposition \ref{propconvergence1chap4}:
\begin{align}\frac{\partial m}{\partial t}(x,a,t) +& \frac{\partial m}{\partial a}(x,a,t) =
- \left(\frac{1}{4}+0.001(4-x)\int_{\R_+} m(x,\alpha,t)d\alpha\right)m(x,a,t)\nonumber\\
m(x,0,t) = & \int_0^{+\infty} x(4-x)e^{- a}m(x,a,t)da.\label{ex1pde}
\end{align}From (\ref{cssurvie}), there exists a non trivial stationary solution if and only if:
\begin{align}
R_0(x)=\int_0^{+\infty}x(4-x)e^{-5a/4}da>1  \ \Leftrightarrow  & \ \frac{4x(4-x)}{5}>1\nonumber\\
 \Leftrightarrow &  \ x\in \left]2-\frac{\sqrt{11}}{2},2+\frac{\sqrt{11}}{2}\right[\approx ]0.35,3.65[.
 \label{ex1condviabilite}
\end{align}
Under (\ref{ex1condviabilite}), Assumption
\ref{hypsolutionstationnaire} is satisfied.

\begin{prop}\label{propsolstationex1}If (\ref{ex1condviabilite}) is satisfied, then (\ref{ex1pde}) admits
 a unique nontrivial stationary solution
\begin{equation}
\widehat{m}(x,a)=\frac{\left(x(4-x)-5/4\right)\left(x(4-x)-1\right)}{0.001(4-x)}
\exp\left(-(x(4-x)-1)a\right).\label{ex1solutionstationnontriviale}
\end{equation}
\end{prop}
\begin{proof}
Any stationary solution is of the form
$\widehat{m}(x,a)=\widehat{m}(x,0)\exp\left(-(1/4+0.001(4-x)
 \widehat{M}_x)a\right)$ by (\ref{formesolutionsstationnaires}). Plugging this expression into (\ref{boundarycondition}) gives
\begin{align}
\widehat{m}(x,0) =&   \widehat{m}(x,0)\int_0^{+\infty} x(4-x)e^{-\left(\frac{5}{4}+0.001(4-x)
 \widehat{M}_x\right)a}da= \widehat{m}(x,0)\frac{x(4-x)}{\frac{5}{4}+0.001(4-x)\widehat{M}_x}.\label{equilibrepopres}
\end{align}Since we are looking for a nontrivial solution, $\widehat{m}(x,0)\not=0$ and necessarily:
\begin{equation}
\widehat{M}_x=\frac{x(4-x)-5/4}{0.001(4-x)}.\label{ex1widehatm}
\end{equation}The definition of
$\widehat{M}_x$ implies $\widehat{M}_x=\widehat{m}(x,0)/(1/4+0.001(4-x)\widehat{M}_x)$ which gives $\widehat{m}(x,0)$.
\end{proof}

\subsubsection{Invasibility}\label{sectionex1invasibility}

\noindent With our choice of competition kernel, Assumptions \ref{dimstan} are satisfied.
 Let us consider the invasion phenomena. An explicit expression can not be obtained for $z_0(y,x)$,
 but we can compute it numerically (see Figure \ref{figtracef}), which allows us to simulate the
 paths of the TSS process (see Figure \ref{figurebandelettes}). Even if $z_0(y,x)$ remains implicit, the study of the domain of invasibility, that consists in pairs $(x,y)$ of traits $y$ with positive fitness in a resident population of trait $x$, can be carried explicitly. It brings information on the long time behavior of the Age structured TSS process.

\begin{prop}
Let $x\in [0,4]$ satisfy (\ref{ex1condviabilite}). The mutant
traits $y$ which  can invade the monomorphic resident
 population with trait $x$ belong to $]\min(x,f(x)),\max(x,f(x))[$, where:
\begin{equation}
f(x):=4-\frac{5/4}{4-x}.\label{fexemple}\end{equation}\end{prop}

\begin{proof} From Proposition \ref{proprappelg}, we know that $z_0(y,x)<1$
(\textit{i.e.} the invasion of the resident population by the
mutant one is possible) if and only if $\int_0^{+\infty}
b(y,a)e^{-d(y,\widehat{M}_x)a}da>1$. Since
\begin{align}
\int_0^{+\infty} b(y,a)e^{-d(y,\widehat{M}_x)a}da= & \int_0^{+\infty}
b(y,a)e^{-d(y,\widehat{M}_y)a
+\left(d(y,\widehat{M}_y)-d(y,\widehat{M}_x\right)a}da,\label{ex1conditioninvasion}
\end{align}and  $$\int_0^{+\infty} b(y,a)e^{-d(y,\widehat{M}_y)}da=1$$ by the
 balance condition
(\ref{eqbalance}), Equation (\ref{ex1conditioninvasion}) is
satisfied if and only if
 $d(y,\widehat{M}_y)-d(y,\widehat{M}_x)>0$. This is equivalent to:\begin{align}
  \widehat{M}_y-\widehat{M}_x>0
\Leftrightarrow   & \frac{x(4-x)-5/4}{0.001(4-x)}<\frac{y(4-y)-5/4}{0.001(4-y)}\nonumber\\
\Leftrightarrow & x(4-x)(4-y)-\frac{5}{4}(4-y)<y(4-x)(4-y)-\frac{5}{4}(4-x)\nonumber\\
\Leftrightarrow & (x-y)(4-x)(4-y)-\frac{5}{4}(4-y-4+x)<0\nonumber\\
\Leftrightarrow & (x-y)\left[(4-x)(4-y)-\frac{5}{4}\right]<0.\label{ex1condition2}
\end{align}
\noindent \underline{Case 1: if $x>y$} (\ref{ex1condition2}) becomes:
\begin{align*}
(4-x)(4-y)<\frac{5}{4} \Leftrightarrow & y>4-\frac{5/4}{(4-x)},
\end{align*}and hence $f(x)<y<x$, with $f$ defined in (\ref{fexemple}).\\
\noindent \underline{Case 2: if $x<y$}, we obtain with similar
computation that $ x<y< f(x)$.
\end{proof}

\begin{corol}Let $x_0\in [0,4]$ be an initial condition that satisfies (\ref{ex1condviabilite}).
 The TSS process $(X_t)_{t\in \R_+}$ starting from $x_0$ converges almost
 surely to the unique fixed point $x^*$ of $f$ on $[0,4]$ when $t\rightarrow +\infty$.
\end{corol}

\begin{proof}The function $f$ is a decreasing concave function satisfying
\begin{align}
\forall x\in [0,4[,\, f\circ f(x)= & 4-\frac{5/4}{4-\left(4-\frac{5/4}{4-x}\right)}
=  4-\frac{5/4}{\frac{5/4}{4-x}}=x.\nonumber
\end{align}
Let's remark that $\ f$ has a unique fixed point  in $[0,4]$ given
by
\begin{equation}x^*:=4-\frac{\sqrt{5}}{2}\thickapprox 2.88.\label{xstar}\end{equation}

\noindent If $x<y<f(x)$, then $x=f\circ f(x)<f(y)<f(x)$, and if $x>y>f(x)$
then
 $x=f\circ f(x)>f(y)>f(x)$. Thus, the new
interval of invasible traits
$[\min(y,f(y)),\max(y,f(y))]\subsetneq
[\min(x,f(x)),\max(x,f(x))]$.

\noindent Hence the sequence
$[\min(X_t,f(X_t)),\max(X_t,f(X_t))]$ is almost
  surely a strictly decreasing sequence for the inclusion and it converges to
  $x^*$.
Since a monomorphic population with trait $x^*$ cannot be invaded,
$x^*$ is an Evolutionary Stable Strategy (ESS) in the sense
developed by Diekmann \cite{diekmann}, Metz \textit{et al.}
\cite{metzgeritzmeszenajacobsheerwaarden} and Geritz \textit{et
al.} \cite{geritzmetzkisdimeszena}.
\end{proof}

\begin{figure}[!ht]
\begin{center}\begin{tabular}[!ht]{cc}
(a) & (b)\\
\begin{minipage}[b]{.33\textwidth}
\vspace{0.5cm}
\centering
\includegraphics[width=0.8\textwidth,height=0.20\textheight,
angle=0,trim=1cm 1cm 1cm 1cm]{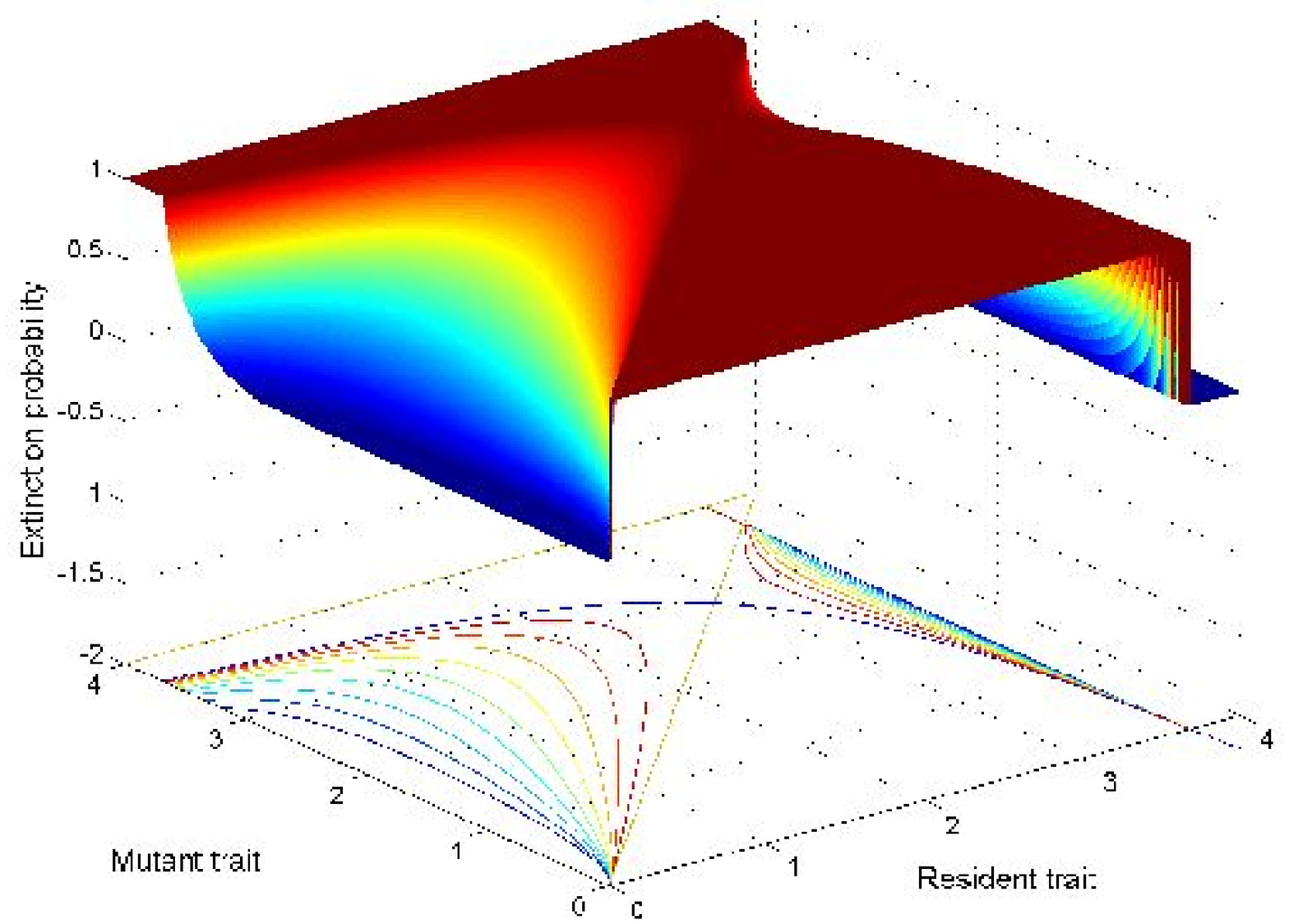}
\end{minipage} &
\begin{minipage}[b]{.33\textwidth}\centering
\includegraphics[width=0.7\textwidth,height=0.15\textheight,
angle=0,trim=1cm 1cm 1cm 1cm]{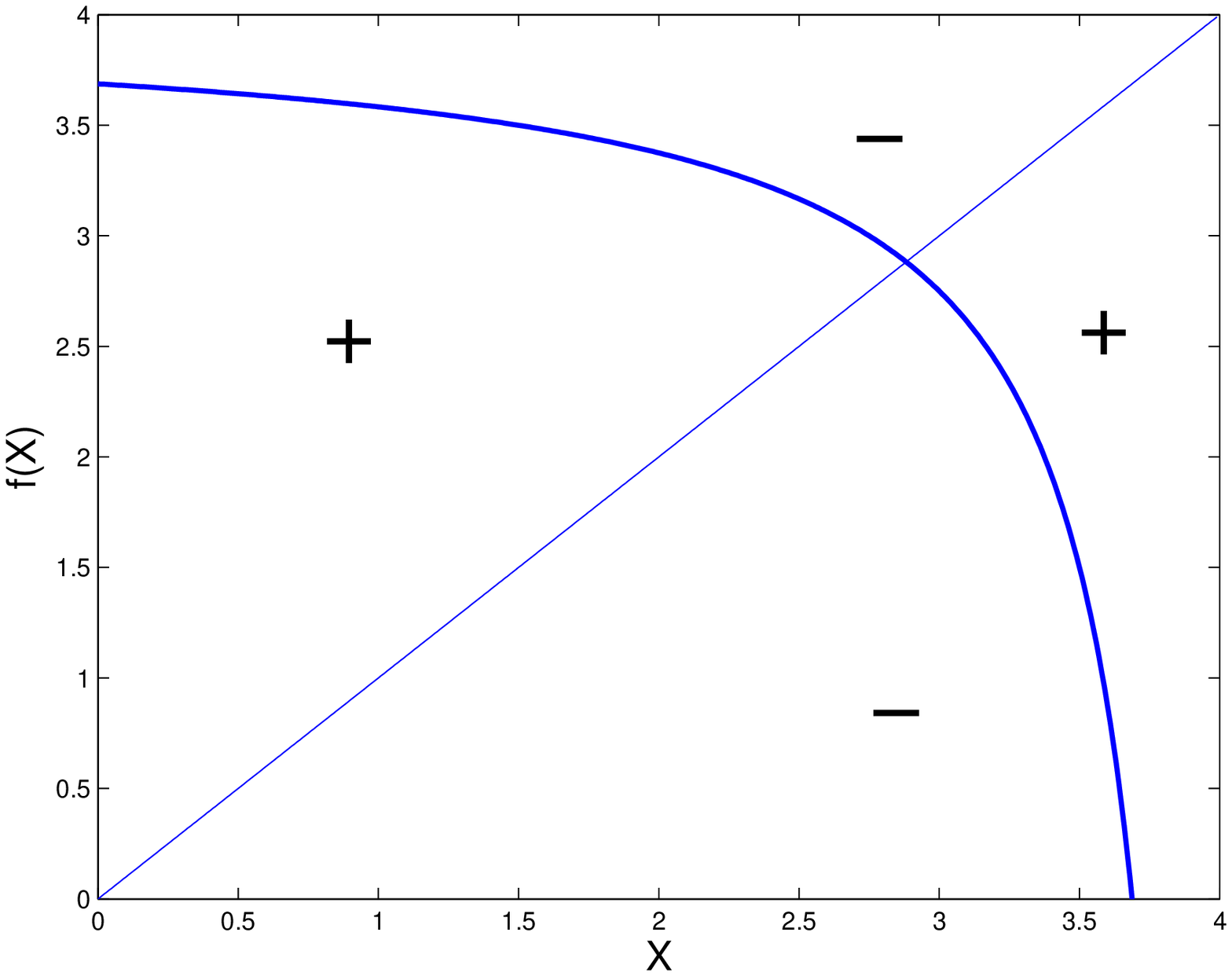}
\end{minipage}
\end{tabular}
\end{center}
\vspace{0cm} \caption{\textit{(a): Probability of extinction as a
function of the resident and mutant traits. For every possible
values of resident and mutant traits on a grid in $[0,4]$, we have
computed the solution of (\ref{equationdeterminationsurvie})
numerically. (b): Graph of function $f$ defined in
(\ref{fexemple}). It gives the regions of the left figure where
$z_0(y,x)=1$ and $z_0(y,x)<1$. This graph is called Pairwise
Invasibility Plot (PIP).}}\label{figtracef}
\end{figure}

\subsubsection{Simulations of the Age-structured TSS}

The \textit{Age-structured TSS} $Z$  defined by \eqref{TSS} is a
measure-valued process  jumping from an equilibrium measure to
another. Each equilibrium measure is characterized by a trait $x$
and by the  density $\widehat{m}(x,a)$ defined in
(\ref{ex1solutionstationnontriviale}). Simulations are given in
Figure \ref{figurebandelettes} below.

\begin{figure}[!ht]
\begin{center}\begin{tabular}[!ht]{cc}
(a) & (b)\\
\begin{minipage}[b]{.33\textwidth}\centering
\vspace{0.3cm}\includegraphics[width=0.8\textwidth,height=0.20\textheight,angle=0,trim=1cm 1cm 1cm 1cm]
{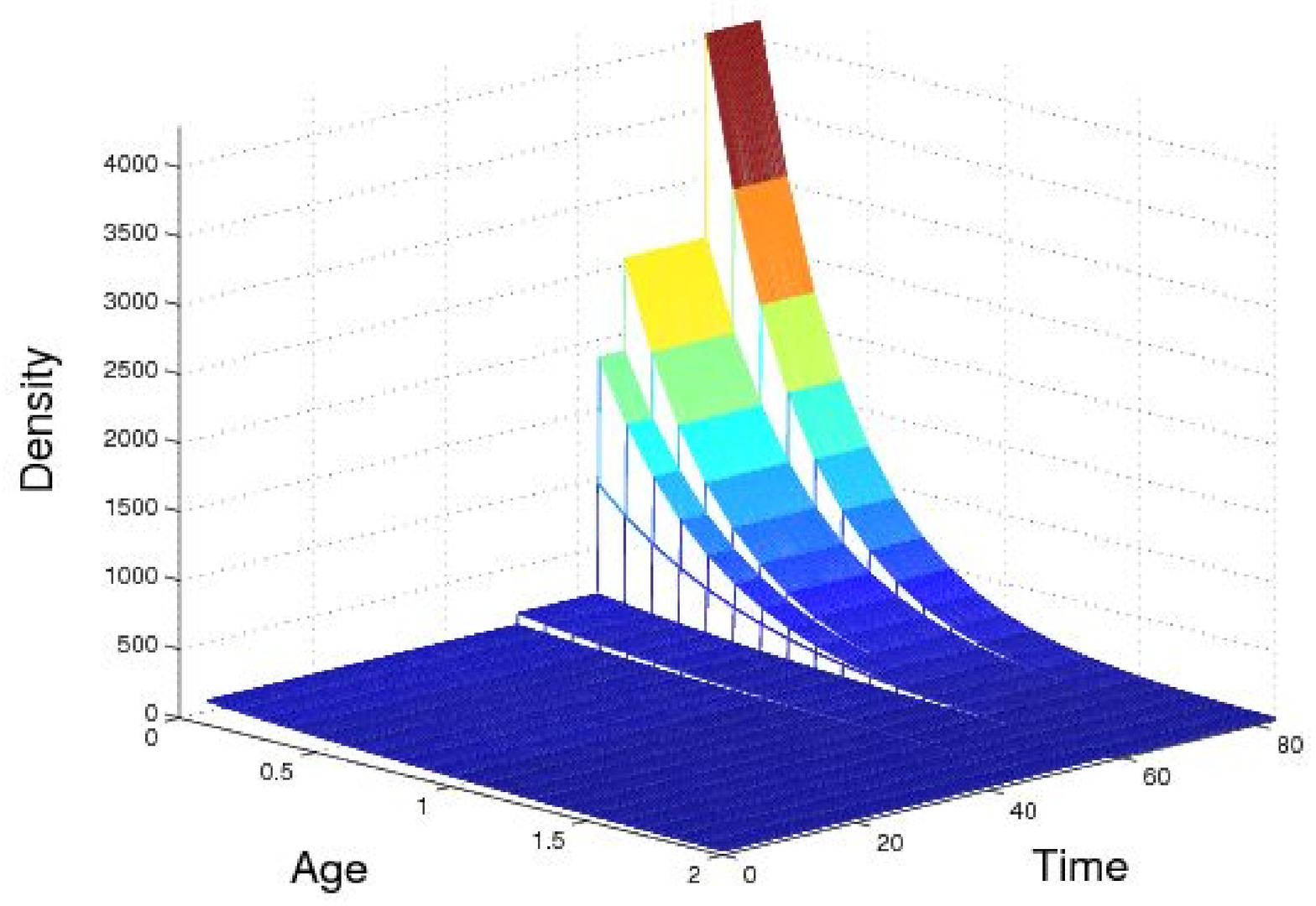}
\end{minipage} &
\begin{minipage}[b]{.33\textwidth}\centering
\includegraphics[width=0.8\textwidth,height=0.20\textheight,angle=0,trim=1cm 1cm 1cm 1cm]{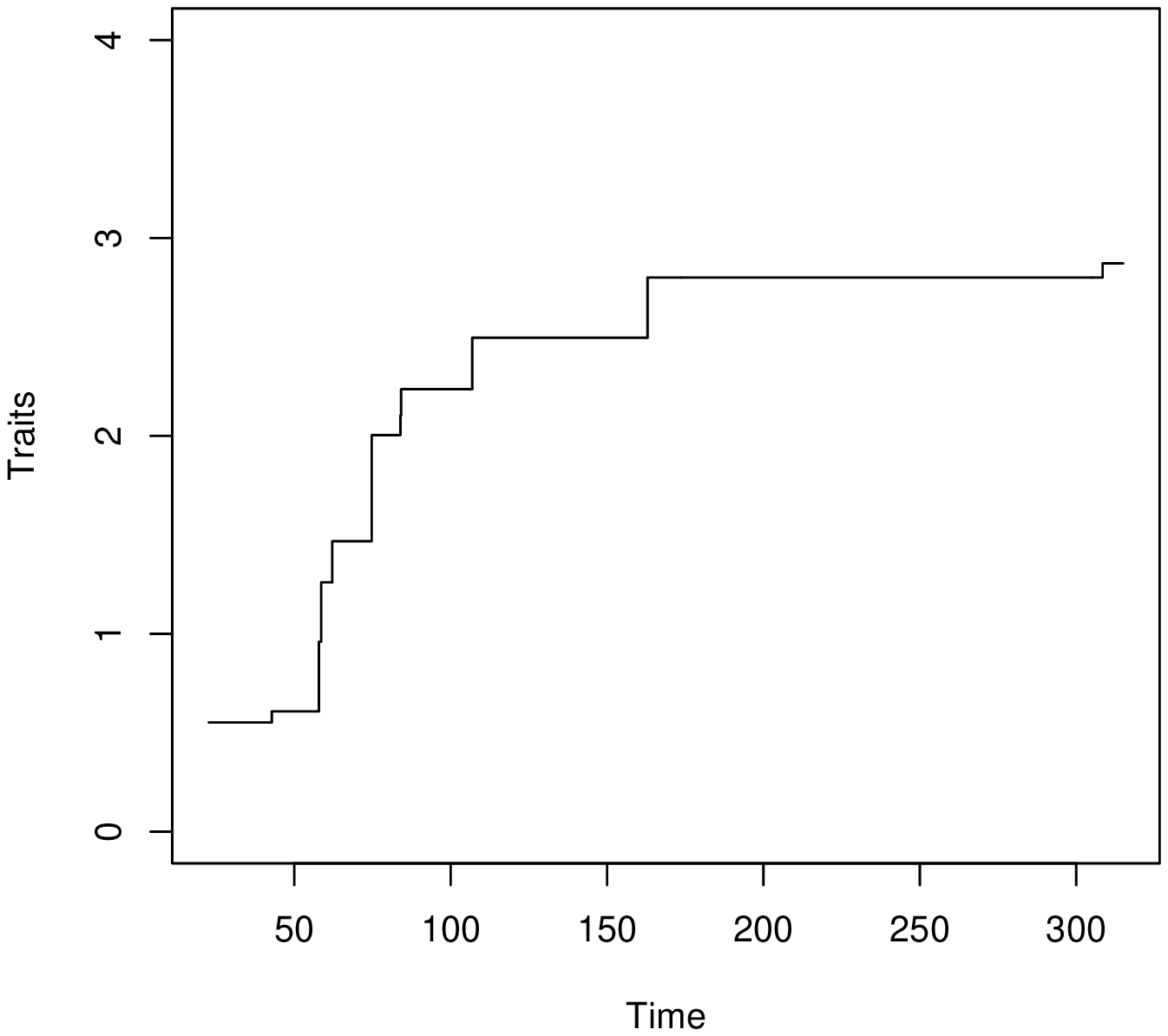}
\end{minipage}\end{tabular}\end{center}
\caption{\textit{Age-structured Trait Substitution Sequence
process: the measure-valued process jumps from one equilibrium
age-distribution to another one. (a): Each strip describes an
equilibrium $\widehat{m}(x,a)$ as age-function. These equilibria are given in (\ref{ex1solutionstationnontriviale}). (b): the
trait-valued process $X$ describing the trait
jumps.}}\label{figurebandelettes}
\end{figure}

\noindent On the Figure \ref{figurebandelettes},
  the initial resident trait is $x_0=0.552$ and the size of the population
   at its equilibrium is $\widehat{M}_{x_0}\thickapprox 189.47$.
In the beginning, successful invasions favours increasing traits and
   $\widehat{M}_{x}$ increases too, as well as the jump rate. For instance,
   after 4
   successful invasions, the trait is approximatively $x=2$ and  $\widehat{M}_2=1375$. When the trait approaches $x^* \thickapprox 2.88$, the interval of
possible invading traits decreases and the extinction
 probability of possible successful mutants tends to one making invasions rare.

\subsubsection{Age-structured Canonical Equation}

When  the mutation step decreases to zero and  time accelerates as
in Section \ref{sectionprezec}, the evolution is described by the
age-structured Canonical Equation \eqref{equationcanonique}.

\begin{prop}The fitness gradient $D_1^1z_0(x,x)$ is obtained from (\ref{gradientfitness})
with:
\begin{equation}
\partial_1 g(x,x)=\frac{(4 x^2  - 32 x + 59)(x^2  - 4 x - 1)}{4x^2(x-4)^3}
\label{fitnessgradientex1}
\end{equation}
and the Canonical Equation has the following explicit form:
\begin{multline}
\frac{dx}{dt}= p \left(\left[\partial_1 g(x,x)\right]_-\int_{\R_+}h^2k(x,h)dh -
 \left[\partial_1g(x,x)\right]_+ \int_{\R_-}h^2k(x,h)dh \right) \\ \times
  \frac{(-x^2+4x-5/4)(-x^2+4x-1)}{0.001(4-x)}=:\varphi(x).\label{canonicalequationexemple1}
\end{multline}
\end{prop}
\begin{proof} The fitness $z_0(x+\varepsilon,x)$ is the smallest
solution in $[0,1]$ of  Equation
(\ref{equationdeterminationsurvie2}), which writes
\begin{align}
z-1= & (z-1)\int_0^{+\infty} (x+\varepsilon )(4-x-\varepsilon
)e^{(z-1)(x+\varepsilon ) (4-x-\varepsilon
)(1-e^{-a})-\left(\frac{5}{4}+(4-x-\varepsilon )
\frac{x(4-x)-5/4}{4-x}\right)a}da.\label{deveqcanon1}
\end{align}
For the case where
\begin{eqnarray*}
1&\geq& \int_0^{+\infty}b(x+\varepsilon ,a)e^{-d(x+\varepsilon
,x)a}da=\int_0^{+\infty}(x+\varepsilon ) (4-x-\varepsilon )
e^{-\left(\frac{5}{4}+(4-x-\varepsilon )
\frac{x(4-x)-5/4}{4-x}\right)a}da \nonumber\\
&=& \frac{(x+\varepsilon )(4-x-\varepsilon
)}{\frac{5}{4}+(4-x-\varepsilon ) \frac{x(4-x)-5/4}{4-x}},
\end{eqnarray*} the unique solution is  $z_0(x+\varepsilon ,x)=1$.
Else, (\ref{deveqcanon1})
admits a solution in $[0,1[$. Writing that $g(x+\varepsilon ,x)-1=\varepsilon
  \partial_1 g(x,x)+o(\varepsilon)$ (since $g(x,x)=1$), and expanding the integrand of (\ref{deveqcanon1})
  with respect to $\varepsilon$ gives:
\begin{align*}
1= & 1+ \int_0^{+\infty}\left\{ \left[e^{-x(4-x)a}\left(\partial_1 g(x,x)x^2(4-x)^2(1-e^{-a})
\right.\right.\right.\nonumber\\
- & \left.\left.\left.\left(\frac{5}{4}-x(4-x)\right)xa+4-2x\right)\right]\varepsilon
 +o(\varepsilon)\right\}da\nonumber\\
= & 1+ \left[\frac{1}{4}\frac{\partial_1 g(x,x)\left(4x^5-48x^4+192x^3-256x^2\right)
+48x^3-183x^2+204x-4x^4+59}{x(x-4)^2(x^2-4x-1)}\right]\varepsilon+o(\varepsilon).
\end{align*}The bracket vanishes for $\partial_1 g(x,x)$ given in (\ref{fitnessgradientex1}). From
\begin{align*}
\int_{\R}hD_h^1z_0(x,x)k(x,h)dh= & \int_{\R_+}h^2 D_1^1z_0(x,x)k(x,h)dh+
\int_{\R_-}h|h| D_{-1}^1z_0(x,x)k(x,h)dh\\
= & \int_{\R_+}h^2 \left[-\partial_1 g(x,x)\right]_+
k(x,h)dh-\int_{\R_-}h^2 \left[\partial_1 g(x,x)\right]_+ k(x,h)dh,
\end{align*}
we obtain (\ref{canonicalequationexemple1}).
\end{proof}

\noindent The graph of $x\to \partial_1 g(x,x)$ is drawn in Figure \ref{figeqcanonique} (b). We can verify that
the fitness gradient vanishes at $x^*\thickapprox 2.88$. For
$x<x^*$ the fitness gradient is positive, implying that the size
tends to increase to $x^*$. Similarly, for $x>x^*$,
the fitness gradient is negative and evolution reduces the size to $x^*$.


\begin{figure}[!ht]
\hspace{-1cm}\begin{tabular}[!ht]{ccc}
(a) & (b) & (c) \\
\begin{minipage}[b]{.35\textwidth}\centering
\includegraphics[width=0.75\textwidth,height=0.21\textheight,
angle=0,trim=1cm 1cm 1cm 1cm]{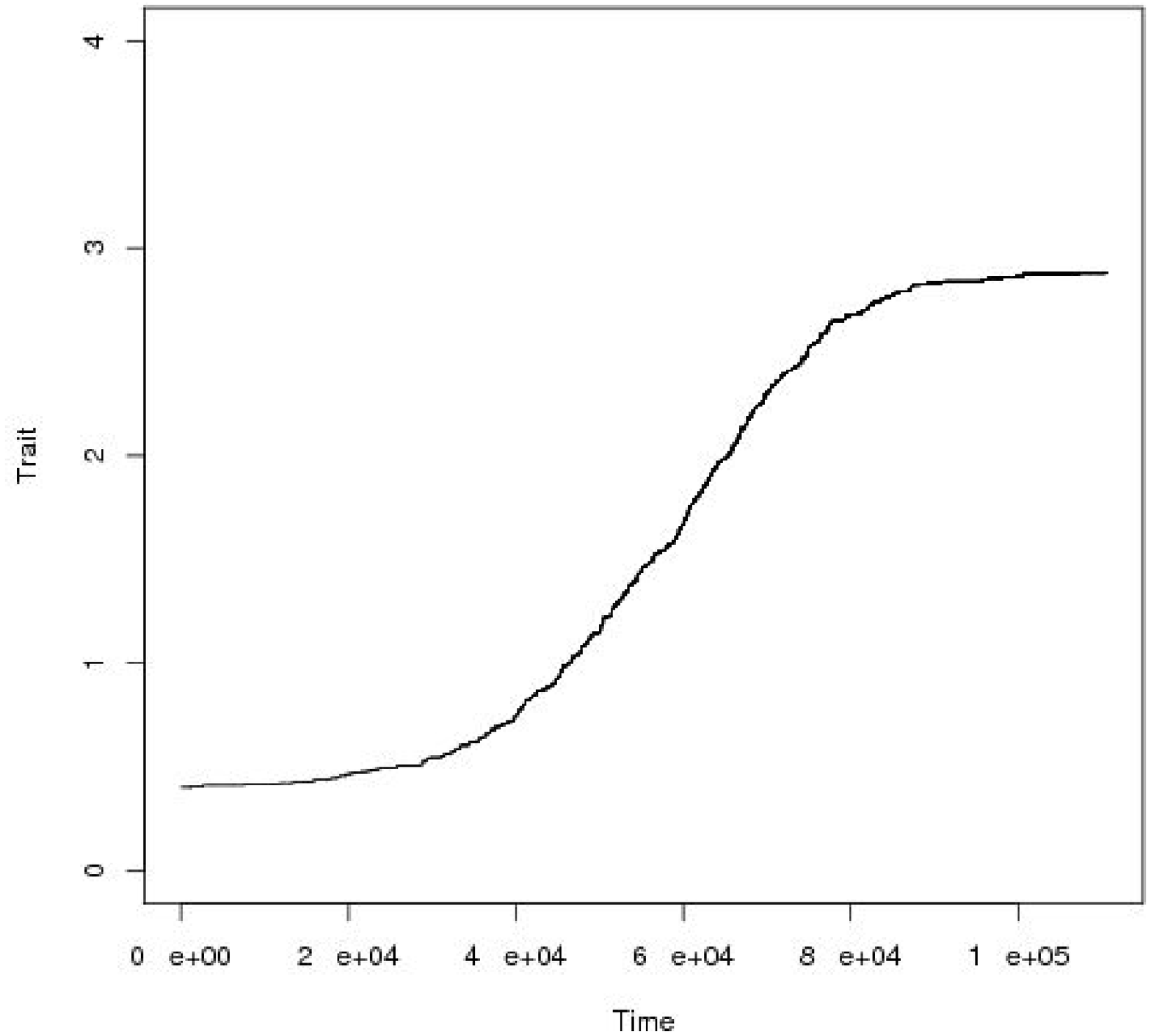}
\end{minipage} &
\begin{minipage}[b]{.31\textwidth}\centering
\vspace{0.5cm}
\includegraphics[width=0.83\textwidth,height=0.18\textheight,
angle=0,trim=1cm 0cm 0cm 1cm]{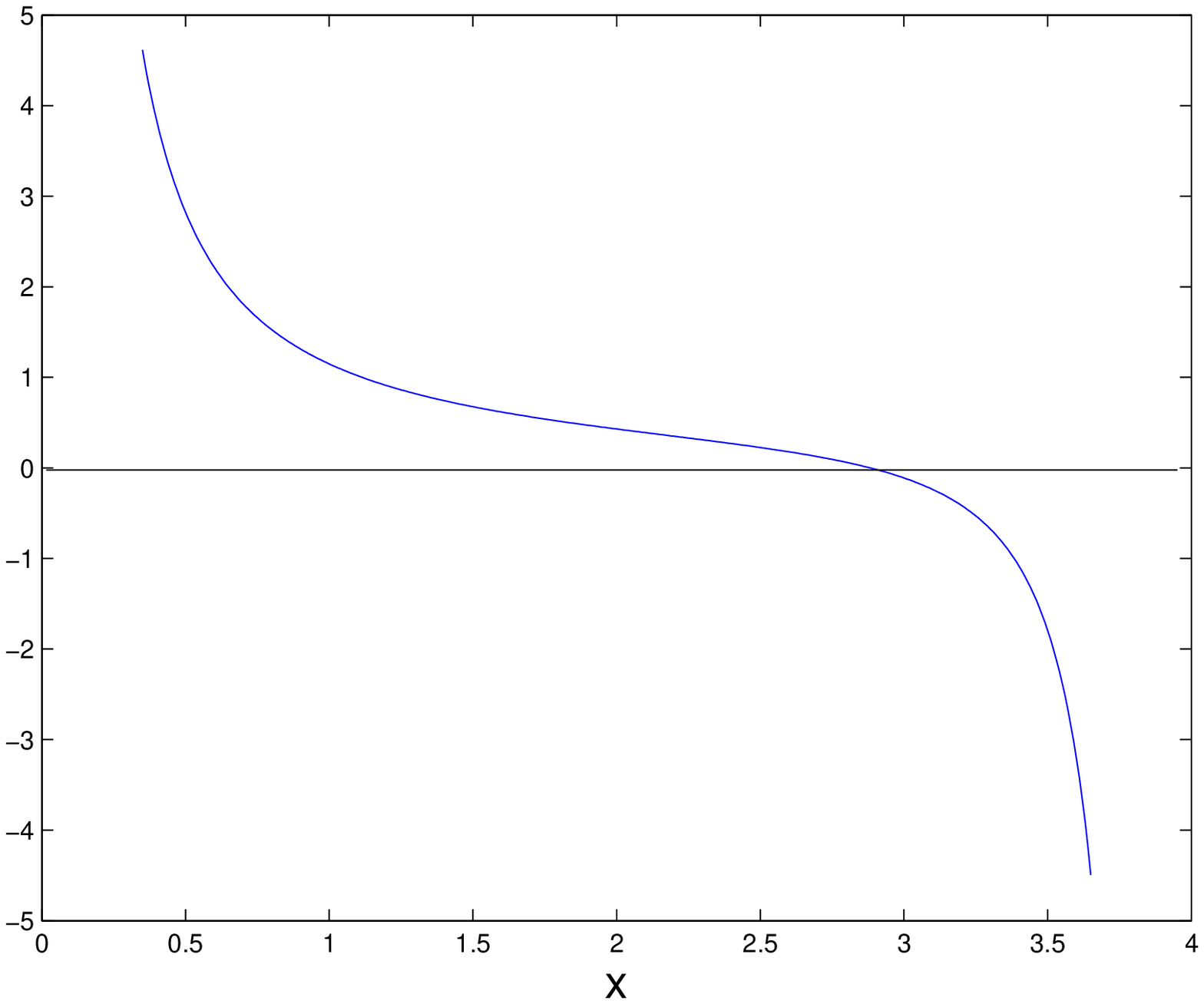}
\end{minipage} &
\begin{minipage}[b]{.31\textwidth}\centering
\vspace{0.5cm}
\includegraphics[width=0.83\textwidth,height=0.18\textheight,
angle=0,trim=1cm 0cm 0cm 1cm]{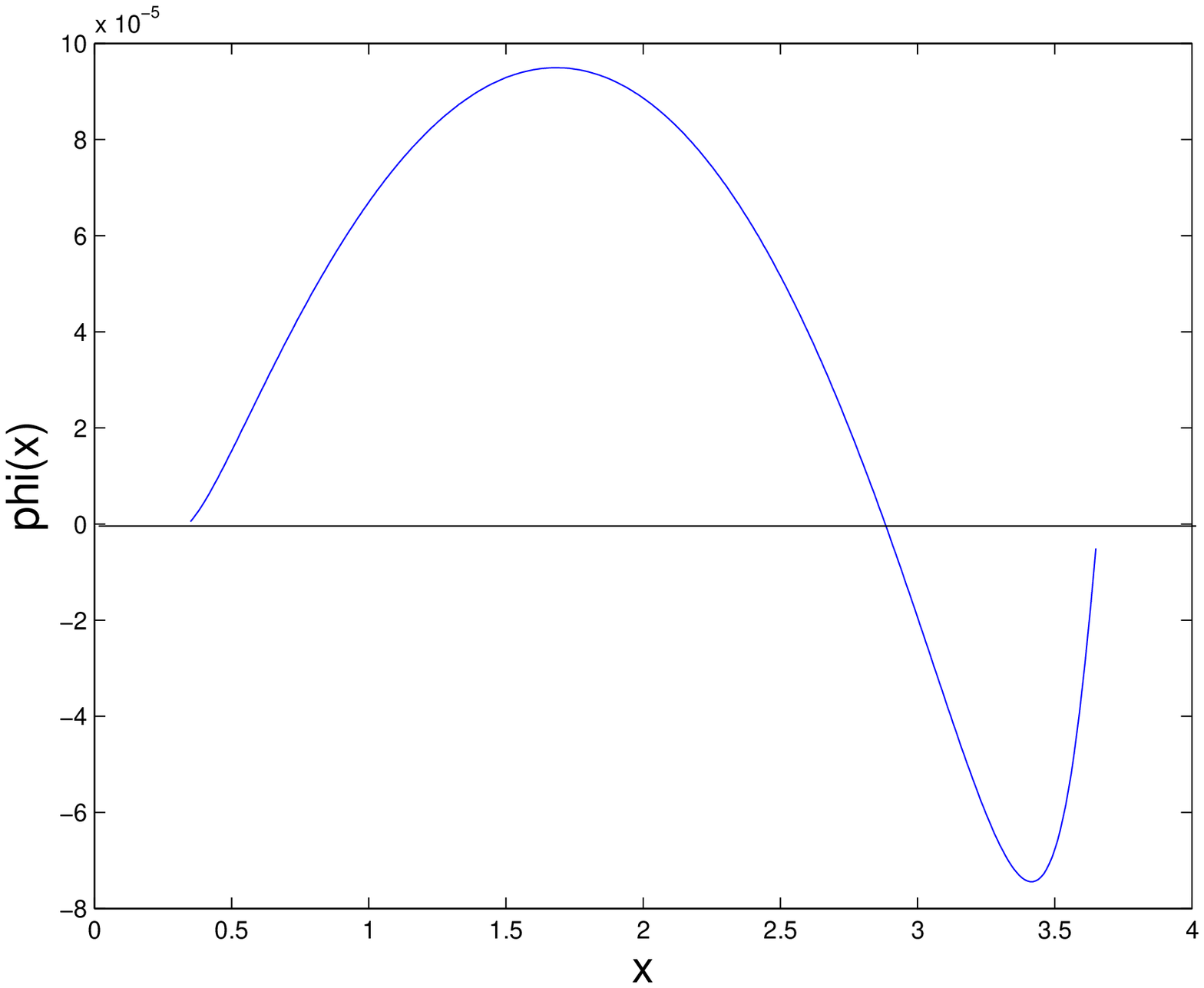}
\end{minipage}
\end{tabular}
\caption{\textit{(a): The Age-structured Canonical Equation,
obtained from the Trait Substitution Sequence process in the limit of small mutation.
(b): Graph of the function $x\mapsto -\partial_1 g(x,x)$
defined in
 (\ref{fitnessgradientex1}). (c): The function $\varphi$ defined in
 (\ref{canonicalequationexemple1}).}}\label{figeqcanonique}
\end{figure}

\subsection{Comparison with a penalized logistic population without age-structure}

We compare the previous example with a logistic population, whithout age-structure, where the birth rate is $b(x)=x(4-x)$ (senescence is absent) and the death rate is $d(x,M)=1/4+0.001(4-x)M$. The large population approximation corresponds to the logistic equation
\begin{equation*}
\frac{dm}{dt}(x,t)=\left(x(4-x)-\frac{1}{4}\right)m(x,t)-0.001(4-x)m^2(x,t),\quad m(x,0)=m_0(x),
\end{equation*}which admits the following unique solution in $\Co^1(\R_+,\R_+)$:
\begin{equation}
m(x,t)=\frac{\left(x(4-x)-\frac{1}{4}\right)m_0(x)e^{(x(4-x)-1/4)t}}{\left(x(4-x)-\frac{1}{4}\right)+0.001(4-x)m_0(x)e^{(x(4-x)-1/4)t}} \stackrel{t\rightarrow+\infty}{\rightarrow}\widehat{m}(x):=\frac{\left[x(4-x)-\frac{1}{4}\right]_+}{0.001(4-x)}.\label{equilibresanssenescence}
\end{equation}
The infinitesimal generator of the TSS process takes here the following explicit form:
\begin{align*}
L\phi(x)= & \int_{\R}\left(\phi(x+h)-\phi(x)\right)px\frac{\left[x(4-x)-1/4\right]_+}{0.001}\left[\left(1-\frac{1}{4(4-x)(4-x-h)}\right)\frac{h}{x+h}\right]_{+}k(x,h)dh\end{align*}
and the Canonical Equation becomes:
\begin{align*}
\frac{dx}{dt}= & \int_{\R }\frac{px}{0.001}\left[x(4-x)-\frac{1}{4}\right]_+\left[\frac{1}{x}-\frac{1}{4x(4-x)^2}\right]_+ h^2k(x,h)dh.
\end{align*}
A mutant $y$ can invade the monomorphic resident population of trait $x$ at equilibrium if
\begin{align*}
   \frac{y(4-y)-\frac{1}{4}-(4-y)\frac{x(4-x)-1/4}{4-x}}{y(4-y)}>0  \Leftrightarrow & (y-x)(4-y)(4-x)-\frac{1}{4}(y-x)>0\\
 \Leftrightarrow & y\in \left]\min(x,f_2(x)),\max(x,f_2(x))\right[
\end{align*}where $f_2(x)=4-(1/4)/(4-x)$. The same conclusion as in Section \ref{sectionex1invasibility} holds. The trait $x^* =7/2$ is an ESS to which the TSS and the Canonical Equation converge.\\

\noindent Let us comment on these results. First, we point out that the ESS is larger in absence of senescence. In this case, the size $\widehat{M}_x$ (\ref{equilibresanssenescence}) of the population at equilibrium for a given trait $x$ is larger since the birth rate does not decrease with age. The stronger competition then favors large sizes in the trade-off between growth and reproduction. Another reason for this difference is that senescence more or less reduces the reproduction period to the beginning of life. Without senescence, the individual has more flexibility in its reproduction strategy. It may choose a size that ensures him a longer life and that allows him to give birth at more spaced intervals.\\

\noindent Let us also notice that the model with senescence is similar to a \textit{penalized} version of the logistic model of this section, with a stronger natural death rate
$$b(x)=x(4-x),\quad d(x,\langle Z,1\rangle)=\frac{5}{4}+0.001(4-x)\langle Z,1\rangle,$$
in the sense that these models lead to the same equilibrium sizes $\widehat{M}_x$ (\ref{ex1widehatm}) and ESS $x^*$ (\ref{xstar}).

\subsection{Age-logistic population interacting through the Kisdi interaction kernel}\label{sectionagelogistic}

\noindent We replace the logistic death rate by an age-logistic death rate where individuals interact through Kisdi's interaction kernel (\ref{kisdi}). In this example, computation becomes rapidly intricated. \\

\noindent The constant $R_0$ is the same as in (\ref{ex1condviabilite}). Equation (\ref{solstationnaire}) defining $\widehat{m}(x,a)$ becomes here
\begin{align}
\frac{\partial \widehat{m}}{\partial a}(x,a)=\left(\frac{1}{4}+ \frac{C\nu a}{1+\nu}\widehat{M}_x\right)\widehat{m}(x,a).\label{ex1trstationnaire1}
\end{align}As we have seen before, this type of equation is called age-logistic since the death rate in monomorphic populations is proportional to the size of the population. Plugging the solutions of (\ref{ex1trstationnaire1}) into the boundary equation of (\ref{boundarycondition}) gives for a non trivial equilibrium:
\begin{align}
1=  \int_0^{+\infty}x(4-x)e^{-\frac{5}{4}a-\frac{1}{2}a^2 \frac{C\nu \widehat{M}_x}{1+\nu}}da
=    x(4-x)e^{\frac{25(1+\nu)}{32C\nu \widehat{M}_x}}\sqrt{\frac{2\pi(1+\nu)}{C\nu \widehat{M}_x}}\left(1-\Phi\left(\frac{5}{4}\sqrt{\frac{1+\nu}{C\nu\widehat{M}_x}}\right)\right)\label{ex1trpart1}
\end{align}where $\Phi$ is the distribution function of the standard Gaussian law. Since the integral term defines a continuous and strictly decreasing function of $\widehat{M}_x$, there exists for the traits satisfying (\ref{ex1condviabilite}) a unique solution $\widehat{M}_x$ to (\ref{ex1trpart1}). There is however no explicit expression, and we use numerical computation to obtain the approximation presented in Figure \ref{figeqcanonique2}.

\begin{figure}[!ht]
\begin{center}\hspace{0cm}
\begin{minipage}[b]{.33\textwidth}\centering
\includegraphics[width=0.8\textwidth,height=0.20\textheight,
angle=0,trim=1cm 1cm 1cm 1cm]{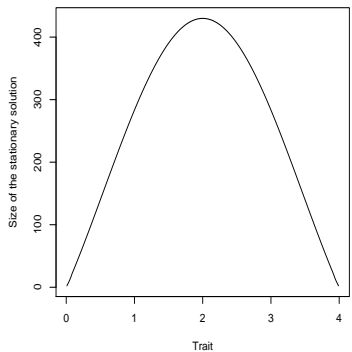}
\end{minipage}
\caption{\textit{A numerical resolution of Equation (\ref{ex1trpart1}), giving for each trait $x$ an approximation of the size $\widehat{M}_x$.}}\label{figeqcanonique2}
\end{center}
\end{figure}

\noindent Once $\widehat{M}_x$ has been computed, it is possible to obtain an approximation of the unique solution of (\ref{ex1trstationnaire1}), (\ref{ex1trpart1}) given as a function of $\widehat{M}_x$ by:
\begin{equation}
\widehat{m}(x,a)=\frac{\widehat{M}_x\sqrt{C\nu\widehat{M}_x}}{e^{\frac{1+\nu}{32 C\nu \widehat{M}_x}}\sqrt{2\pi(1+\nu)}\left(1-\Phi\left(\frac{1}{4}\sqrt{\frac{1+\nu}{C\nu \widehat{M}_x}}\right)\right)}\exp\left(-\frac{a}{4}-\frac{C\nu \widehat{M}_x}{2(1+\nu)}a^2\right).
\end{equation}
\noindent To simulate approximations of the Age-structured TSS-jump process, we have to compute the fitness $z_0(y,x)$ of a mutant $y$ in the monomorphic resident population of trait $x$. This can be obtained by solving numerically an equation in which the term $\widehat{M}_x$ is replaced by its numerical approximation $\widetilde{M}_x$:
\begin{equation}
1=y(4-y)\int_0^{+\infty}\exp\left(-\frac{5}{4}a+(z-1)y(4-y)(1-e^{-a})-\frac{U(y,x)\widetilde{M}_x}{2}a^2\right)da.
\end{equation}

\begin{figure}[!ht]
\begin{center}\begin{tabular}[!ht]{cc}
(a) & (b) \\
\begin{minipage}[b]{.36\textwidth}\centering
\vspace{0.5cm}
\includegraphics[width=0.85\textwidth,height=0.22\textheight,
angle=0,trim=1.1cm 0cm 0cm 1.1cm]{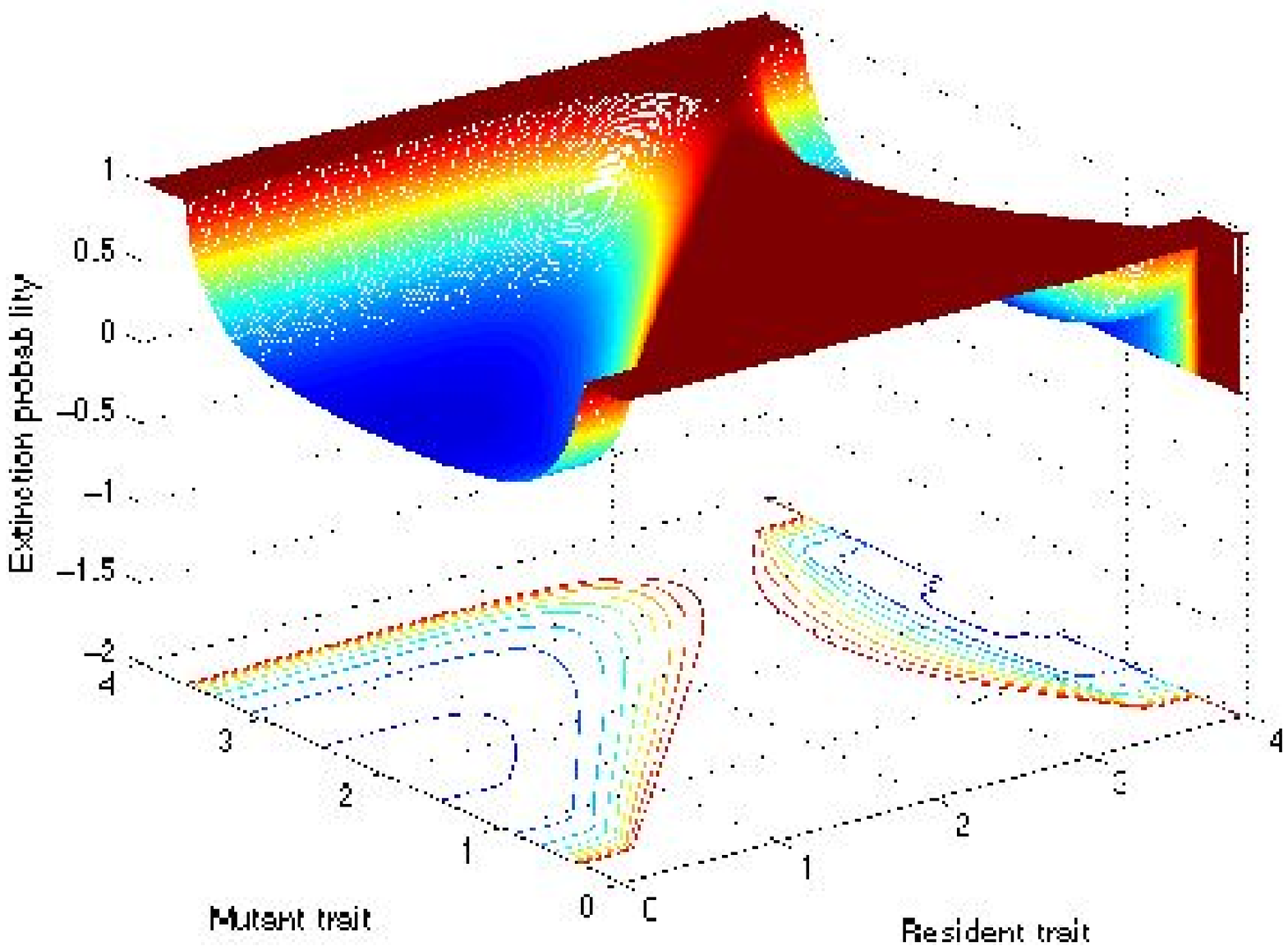}
\end{minipage} &
\hspace{-0.3cm}\begin{minipage}[b]{.40\textwidth}\centering
\includegraphics[width=0.73\textwidth,height=0.22\textheight,
angle=0,trim=1cm 1cm 1cm 1cm]{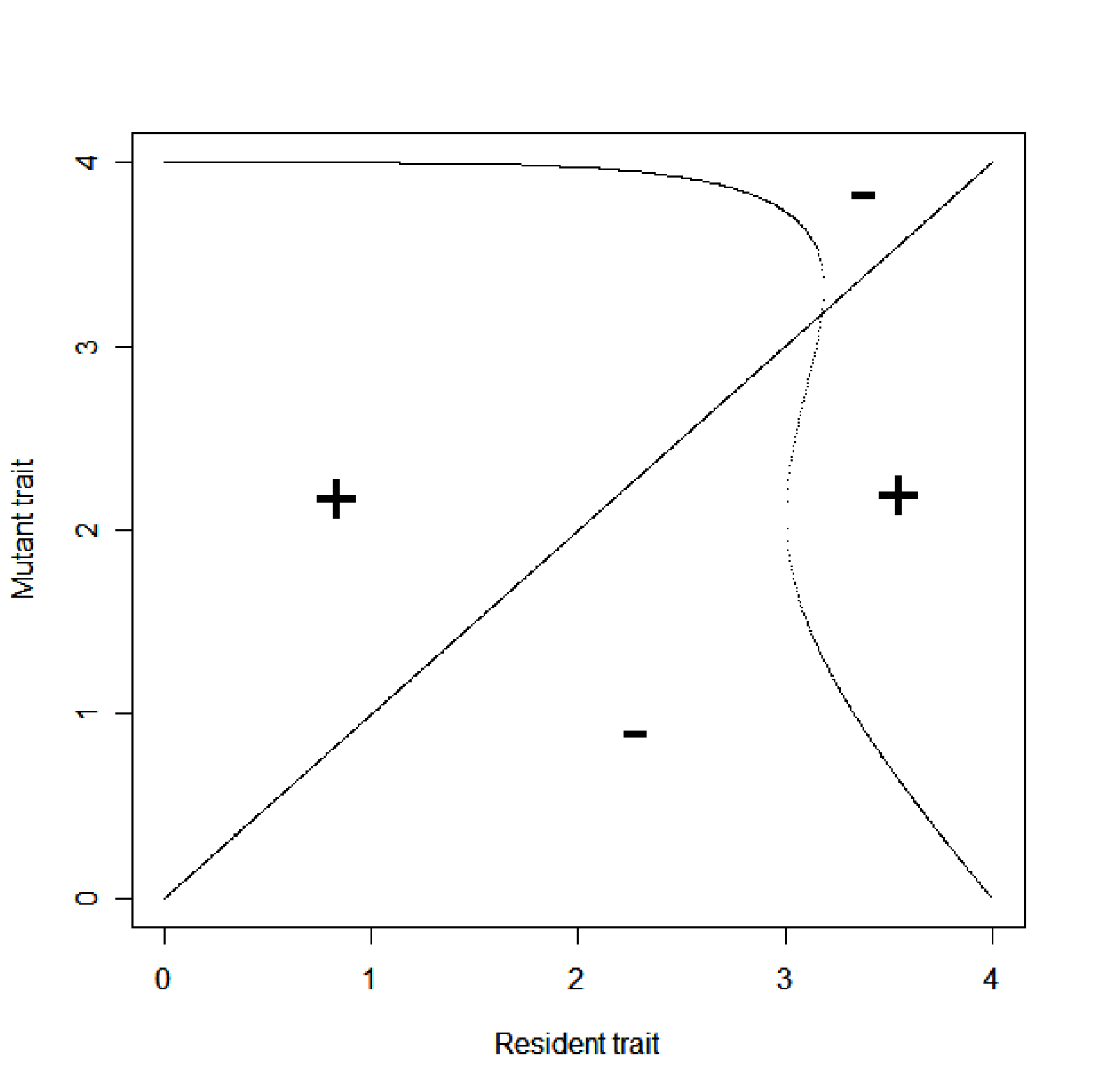}
\end{minipage}
\\
(c) & (d) \\
\begin{minipage}[b]{.36\textwidth}\centering
\includegraphics[width=0.85\textwidth,height=0.22\textheight,
angle=0,trim=1.1cm 0cm 0cm 1.1cm]{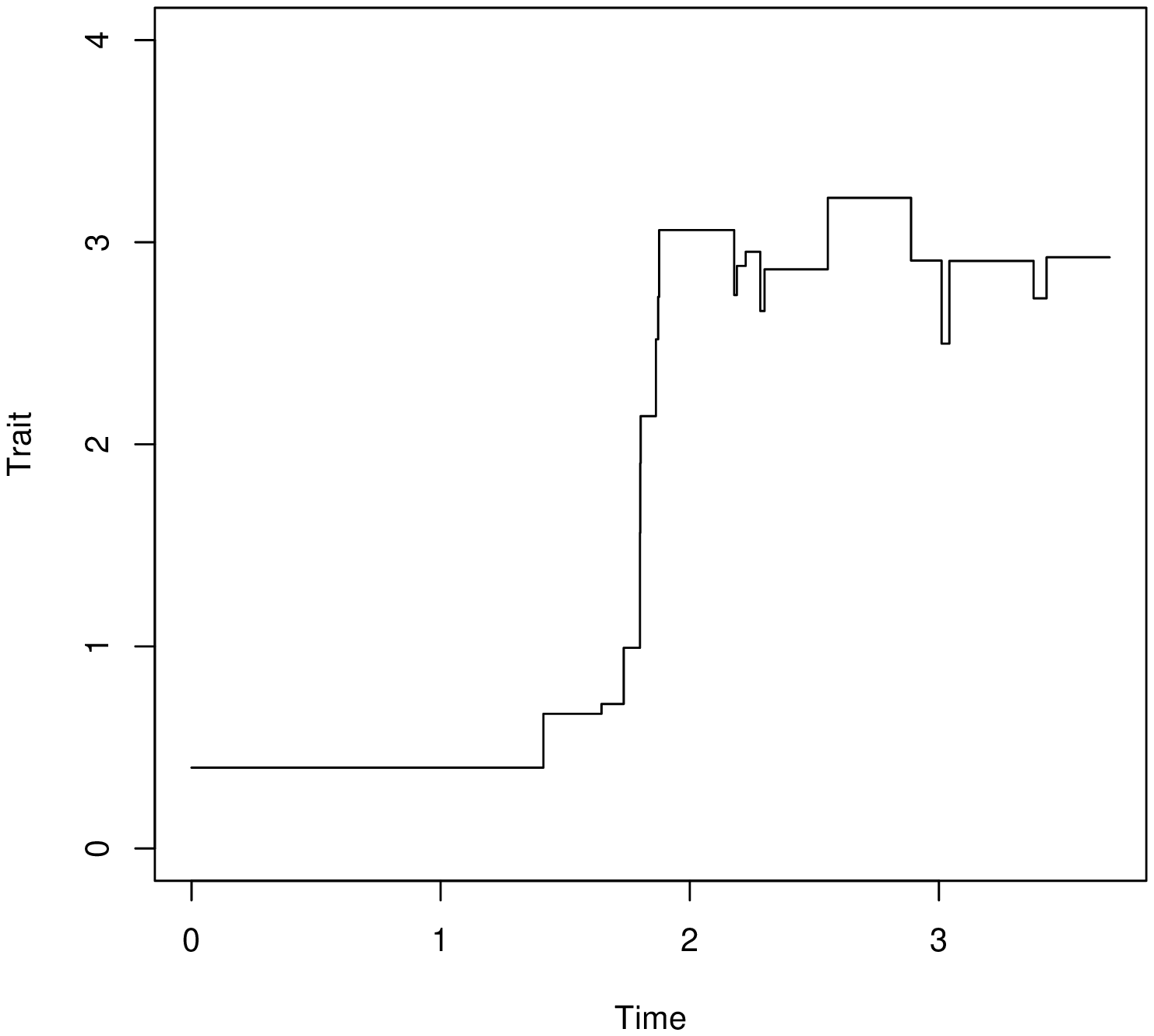}
\end{minipage} &
\begin{minipage}[b]{.36\textwidth}\centering
\includegraphics[width=0.85\textwidth,height=0.22\textheight,
angle=0,trim=1.1cm 0cm 0cm 1.1cm]{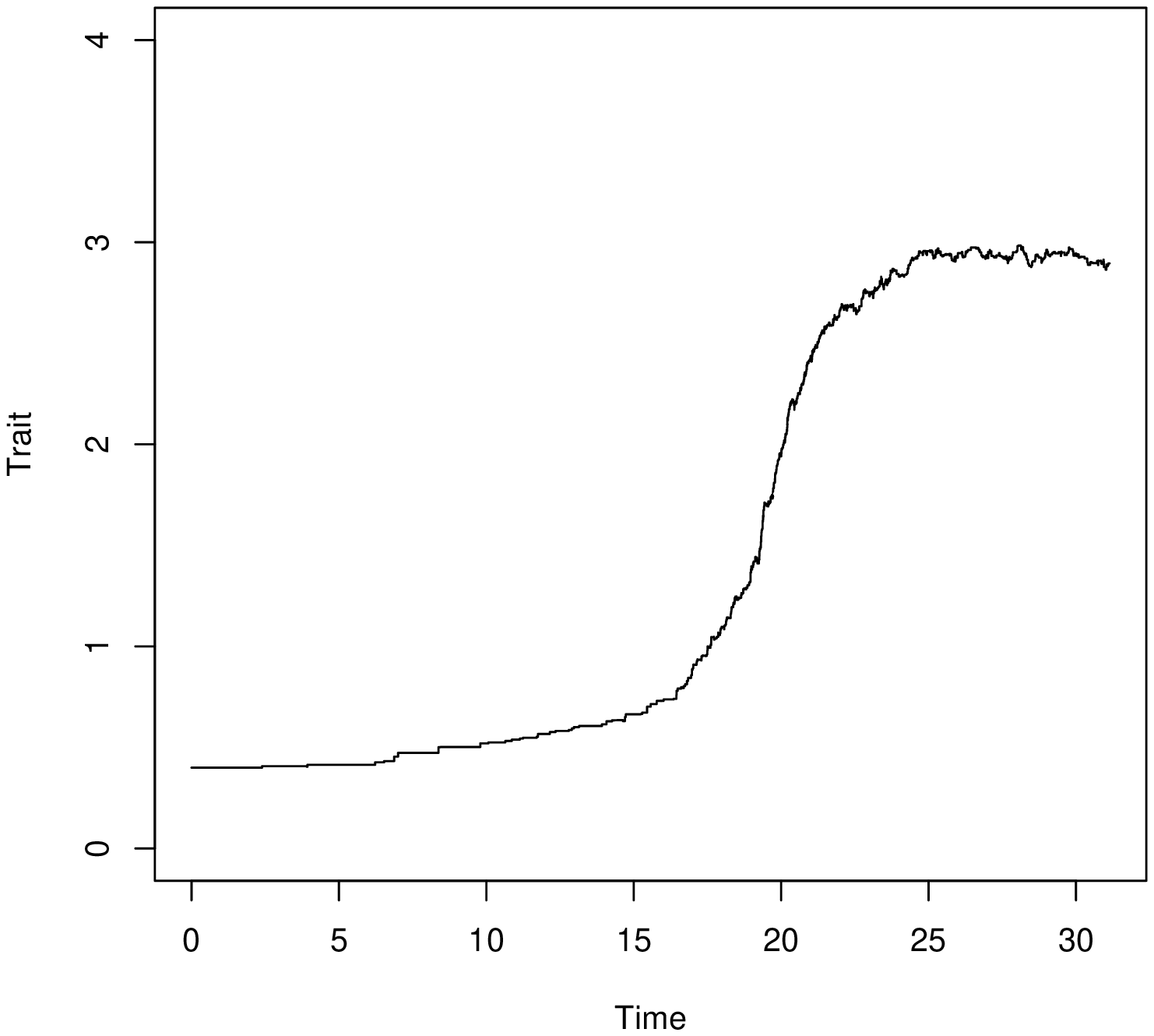}
\end{minipage}
\end{tabular}
\end{center}
\vspace{-0.5cm}\caption{\textit{(a): Probability of extinction as a function of the resident and mutant traits. (b): Associated PIP. The domain of invasibility is delimited by the first bissector and by a curve that is not the graph of a function any more. (c): Simulations of an Age-structured TSS path. (d): Approximation of the Age-structured Canonical Equation obtained in the limit of the TSS when mutations are small.}}\label{figex1tr}
\end{figure}

\noindent Conversely to the example corresponding to the rates (\ref{tauxnaissanceexemple}), (\ref{tauxmortexemple}), the TSS process does not seem to converge in long time here. The twisted curve in the PIP (Figure \ref{figex1tr}) shows indeed that we are in a case of \textit{mutual invasibility} (see \cite{diekmann, metzgeritzmeszenajacobsheerwaarden, geritzmetzkisdimeszena}).

\section{Example 2}\label{sectionexemple2}

In this last example, we investigate a model where the death rate has the separable multiplicative form commented in Section \ref{sectiondescription}\begin{equation}d(x,a,Z)=\int_{[0,4]\times \R_+}a(1+e^{-\alpha})U(x,y)Z(dy,d\alpha),\label{tauxmortexemple2}\end{equation}
$U$ being the Kisdi kernel defined in (\ref{kisdi}). The birth rate is here $b(x)=x(4-x)$. Here, it will be possible to carry an investigation of the branching phenomenon that is observed.

\subsection{Monomorphic equilibrium in large population}

In this example, the partial differential equation defined in
Proposition \ref{propconvergence1chap4} is given by
\begin{align}
  & \frac{\partial m}{\partial t}(x,a,t)+\frac{\partial m}{\partial a}(x,a,t)= -  a m(x,a,t)\int_0^{+\infty}(1+e^{-\alpha})U(x,x)m(x,\alpha,t)d\alpha,\label{eq1ex2}\\
   & m(x,0,t)=\int_0^{+\infty}x(4-x)m(x,a,t)da.\label{eq2ex2}
\end{align}

\begin{prop}Equations (\ref{eq1ex2})-(\ref{eq2ex2}) admit a unique nontrivial
stationary solution:
\begin{equation}
\widehat{m}(x,a)=  \frac{\pi x^3(4-x)^{3}(1+\nu)}{2C\nu
 \left[1+2e^{\frac{1}{\pi x^2(4-x)^2}}\left(1-\Phi\left(\frac{2}{\pi x^2(4-x)^2}\right)
 \right)\right]}
\exp\left(-\frac{\pi x^2(4-x)^2 a^2}{4}  \right),\label{solutionstationnaireex2}
\end{equation}where $\Phi$ is the distribution function of the standard gaussian law.
\end{prop}

\begin{proof}(\ref{formesolutionsstationnaires}) becomes here
\begin{align}
\widehat{m}(x,a)=  \widehat{m}(x,0)e^{-\int_0^a \alpha
d\alpha \widehat{E}(x)}=\widehat{m}(x,0)e^{{-a^2\over
2} \widehat{E}(x)},\quad
\widehat{E}(x)=
\int_0^{+\infty}(1+e^{-\alpha})U(x,x)\widehat{m}(x,\alpha)d\alpha.
\label{defechapeau}
\end{align}Plugging (\ref{defechapeau}) in (\ref{eq2ex2}), we obtain
\begin{align}
1= \int_0^{+\infty}x(4-x)e^{-\frac{\widehat{E}(x)
a^2}{2}}da\Leftrightarrow\quad 1=\frac{x(4-x)}
{2}\sqrt{\frac{2\pi}{\widehat{E}(x)}}\quad \Leftrightarrow \quad
\widehat{E}(x)= \frac{\pi x^2(4-x)^2}{2}.\label{echapeau}
\end{align}Finally, from (\ref{defechapeau})
\begin{align}
\widehat{m}(x,0)= & \frac{\widehat{E}(x)}{\int_0^{+\infty}(1+e^{-\alpha})U(x,x)e^{-\frac{\widehat{E}(x)
 \alpha^2}{2}}d\alpha}
=    \frac{\widehat{E}(x)^{3/2}(1+\nu)}{C\nu \sqrt{2\pi}
\left[\frac{1}{2}+e^{\frac{1}
{2\widehat{E}(x)}}\left(1-\Phi\left(\frac{1}{\widehat{E}(x)}\right)\right)\right]}.
\label{mchapeau0}
\end{align}We deduce the announced result from (\ref{defechapeau}),
(\ref{echapeau}) and (\ref{mchapeau0}).
\end{proof}

\noindent In Appendix \ref{annexestabiliteex2}, we show that the nontrivial equilibrium (\ref{solutionstationnaireex2}) is asymptotically stable.

\subsection{Invasibility}

\begin{figure}[!ht]
\begin{center}\begin{tabular}[!ht]{cc}
(a) & (b) \\
\begin{minipage}[b]{.33\textwidth}\centering
\vspace{0.5cm}
\includegraphics[width=0.8\textwidth,height=0.20\textheight,angle=0,trim=1cm 1cm 1cm 1cm]
{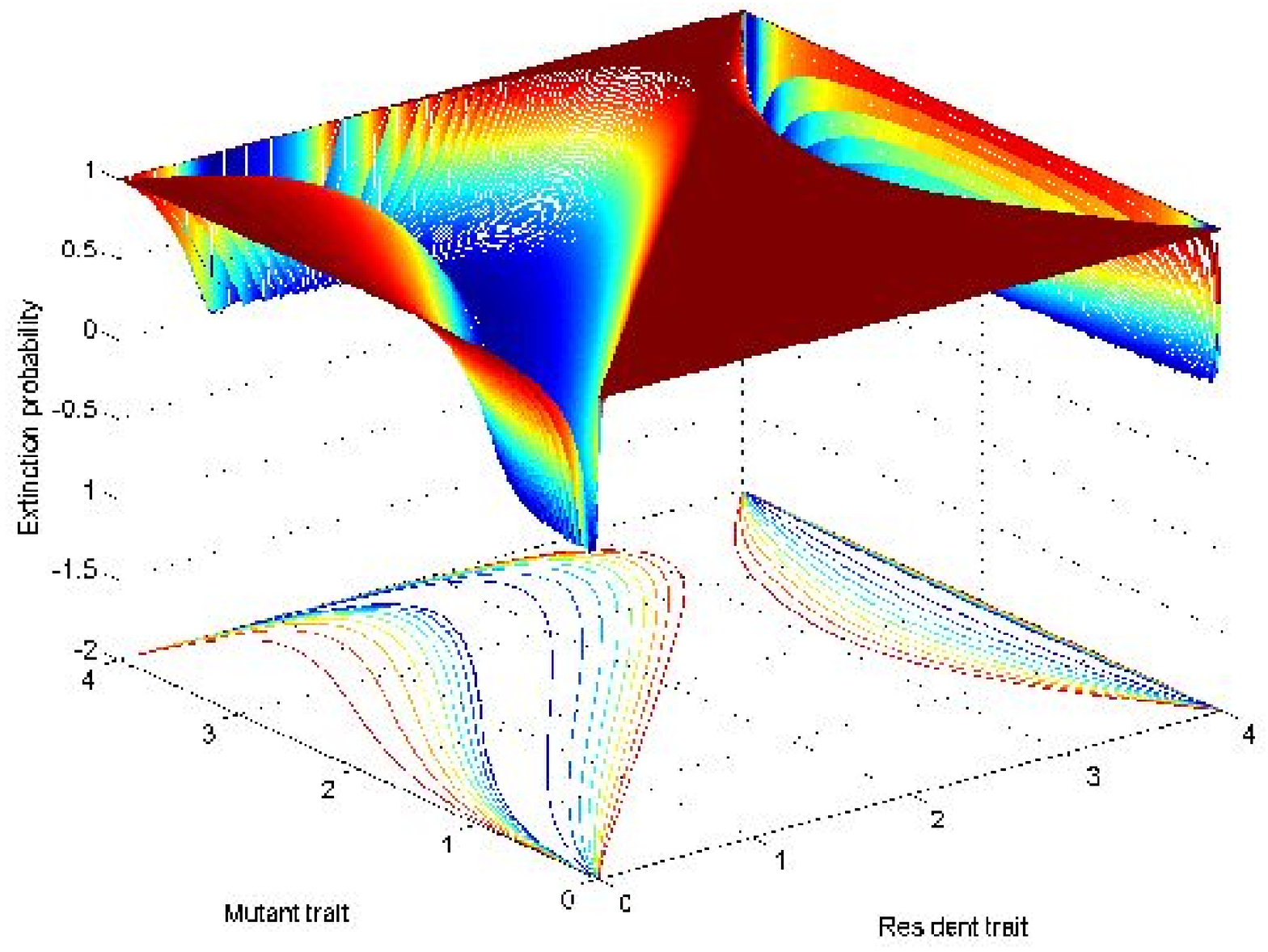}
\end{minipage} \vspace{0.5cm}&
\hspace{-0.3cm}\begin{minipage}[b]{.40\textwidth}\centering
\vspace{0.3cm}\includegraphics[width=0.73\textwidth,height=0.22\textheight,
angle=0,trim=1cm 0cm 1cm 1cm]{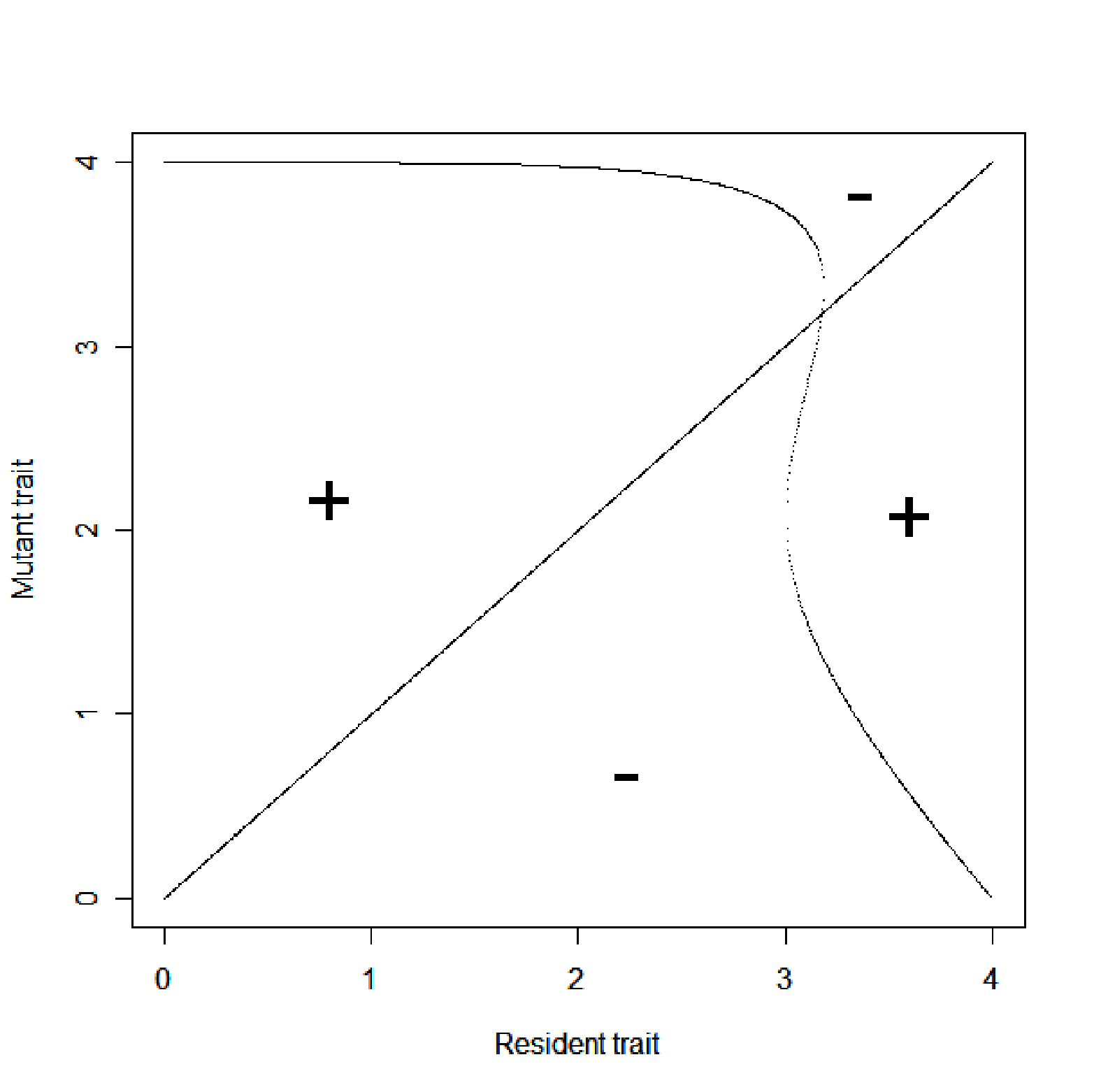}
\end{minipage}\\
(c) & (d) \\
\begin{minipage}[b]{.33\textwidth}\centering
\includegraphics[width=0.8\textwidth,height=0.20\textheight,angle=0,trim=1cm 1cm 1cm 0cm]
{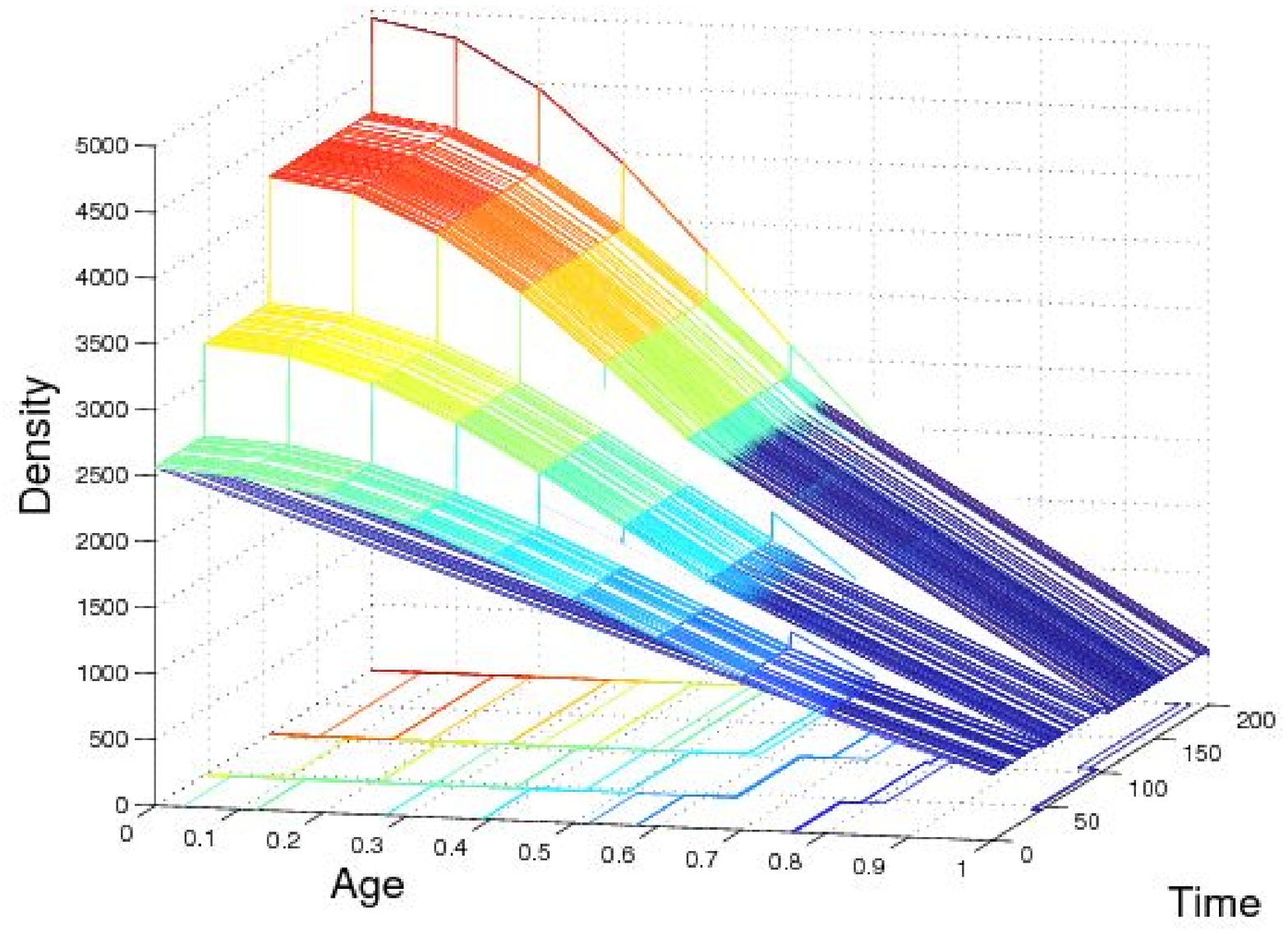}
\end{minipage}
 &
\begin{minipage}[b]{.33\textwidth}\centering
\includegraphics[width=0.8\textwidth,height=0.20\textheight,angle=0,trim=1cm 1cm 1cm 1cm]
{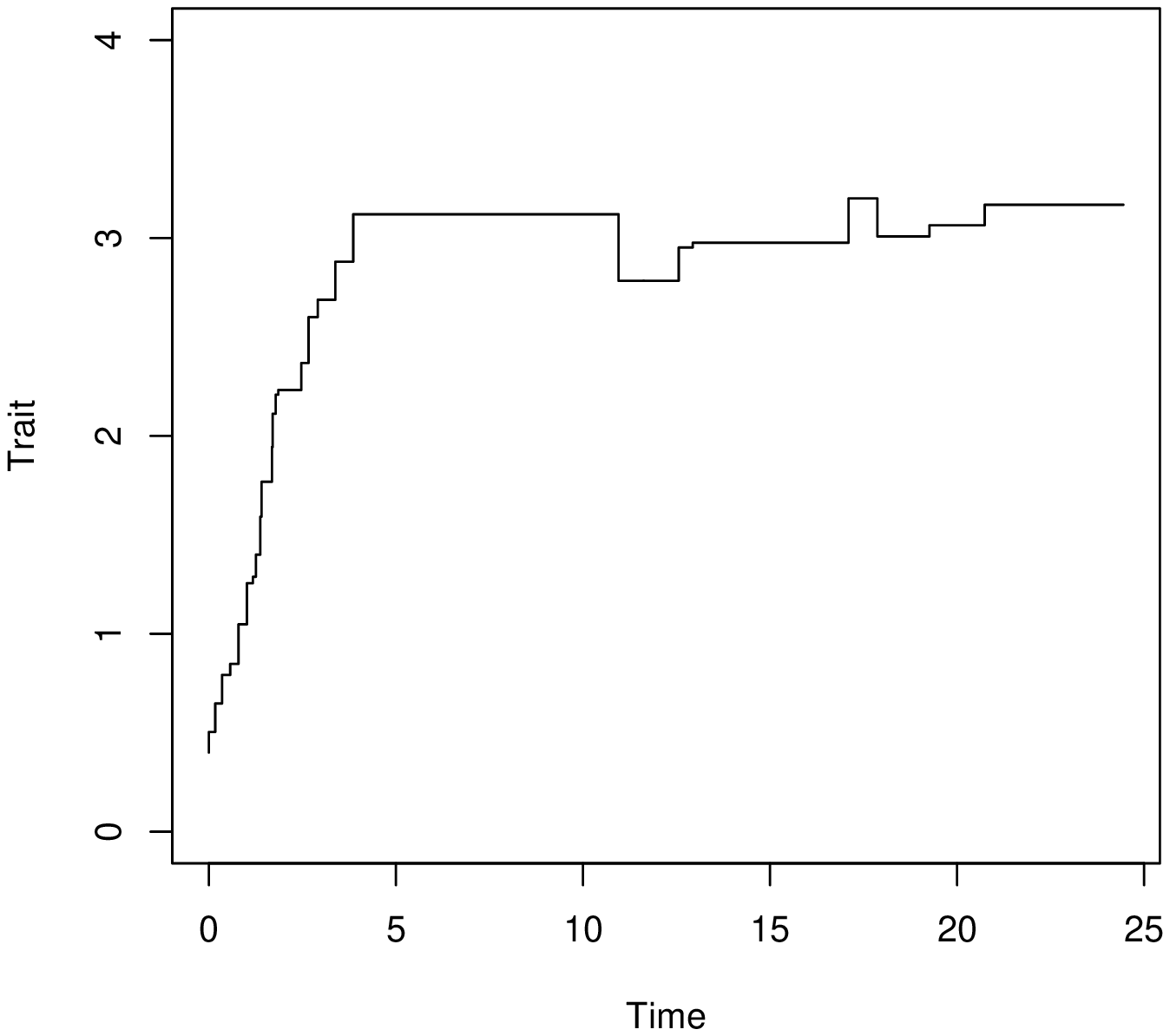}
\end{minipage}
\end{tabular}\end{center}
\caption{\textit{(a): Extinction probability as  function of the resident
and mutant traits. (b): PIP showing the regions of possible invasion, obtained by solving
(\ref{critereinvasionexemple2}). (c): Simulation of the Age-structured
TSS process. (d): we show the successive equilibrium age distributions, and right,
 the sequence of corresponding traits.}}\label{figtracefex2}
\end{figure}

Invasion of the resident population of trait $x$ at equilibrium by a mutant of trait
$y$ is possible if and only if\begin{equation}
\int_0^{+\infty}y(4-y)e^{-\frac{a^2
\widehat{E}(y,x)}{2}}da>1,\quad \mbox{ where }\widehat{E}(y,x)=
  \int_0^{+\infty}(1+e^{-\alpha})U(y,x)\widehat{m}(x,\alpha)d\alpha.
\end{equation}The balance equation \eqref{eqbalance} is
\begin{equation}
1= \int_0^{+\infty}y(4-y)e^{-\frac{\widehat{E}(y)a^2}{2}}da,
\label{balanceconditionex2}
\end{equation} and hence, the probability of extinction $z_0(y,x)$ is strictly less than 1
 if and only if:
\begin{align}
\widehat{E}(y,x)<\widehat{E}(y)\label{conditionexemple2}
\end{align}
Since $
\widehat{E}(y,x)=\widehat{E}(x)\frac{(1+\nu)e^{-k(y-x)}}{1+\nu
e^{-k(y-x)}}$, (\ref{conditionexemple2})
becomes:
\begin{align}
x^2(4-x)^2\frac{(1+\nu)e^{-k(y-x)}}{1+\nu e^{-k(y-x)}}<y^2(4-y)^2.\label{critereinvasionexemple2}
\end{align}
This inequality can be solved numerically and the PIP is given in
Figure \ref{figtracefex2} (b).\\

\noindent In the case where invasion is possible, the fitness of the mutant
of trait $y$ in the resident population of trait $x$ is the
smallest solution in $[0,1]$ of
(\ref{equationdeterminationsurvie3}) which in this example becomes
\begin{align}
1= & \int_0^{+\infty}y(4-y)e^{(z-1)y(4-y)a-\widehat{E}(y,x)a^2/2}da\nonumber\\
= & \frac{\sqrt{2\pi}y(4-y)}{\sqrt{\widehat{E}(y,x)}}
\exp\left(\frac{(z-1)^2y^2(4-y)^2}{2\widehat{E}(y,x)}\right)
\left(1-\Phi\left(-\frac{(z-1)y(4-y)}{\sqrt{\widehat{E}(y,x)}}\right)\right).
\label{equationzexemple2}
\end{align}

\noindent The solution of (\ref{equationzexemple2}) cannot be obtained explicitly, but
can be computed numerically for every $x$ and $y$ in $[0,4]$  as obtained in Figure \ref{figtracefex2}.\\

\noindent Conversely to Example 1, the TSS process does not converge to a limit, even if it
 seems to stay in the neighborhood of the evolutionary singularity $x^*\approx 3.2$.
 Moreover,  looking at the underlying microscopic process (Figure \ref{figmicrosc}),
 we can observe splitting of the population into smaller groups once the neighborhood of
 this point has been reached. To understand this fact, let us study more carefully the
 directional derivatives of the solution of (\ref{equationzexemple2}).

\begin{figure}[!ht]
\begin{center}\begin{tabular}[!ht]{cc}
(a) & (b) \\
\begin{minipage}[b]{.33\textwidth}\centering
\vspace{0.7cm}\includegraphics[width=0.7\textwidth,height=0.17\textheight,angle=0,trim=1.1cm 0cm 0cm 1.1cm]
{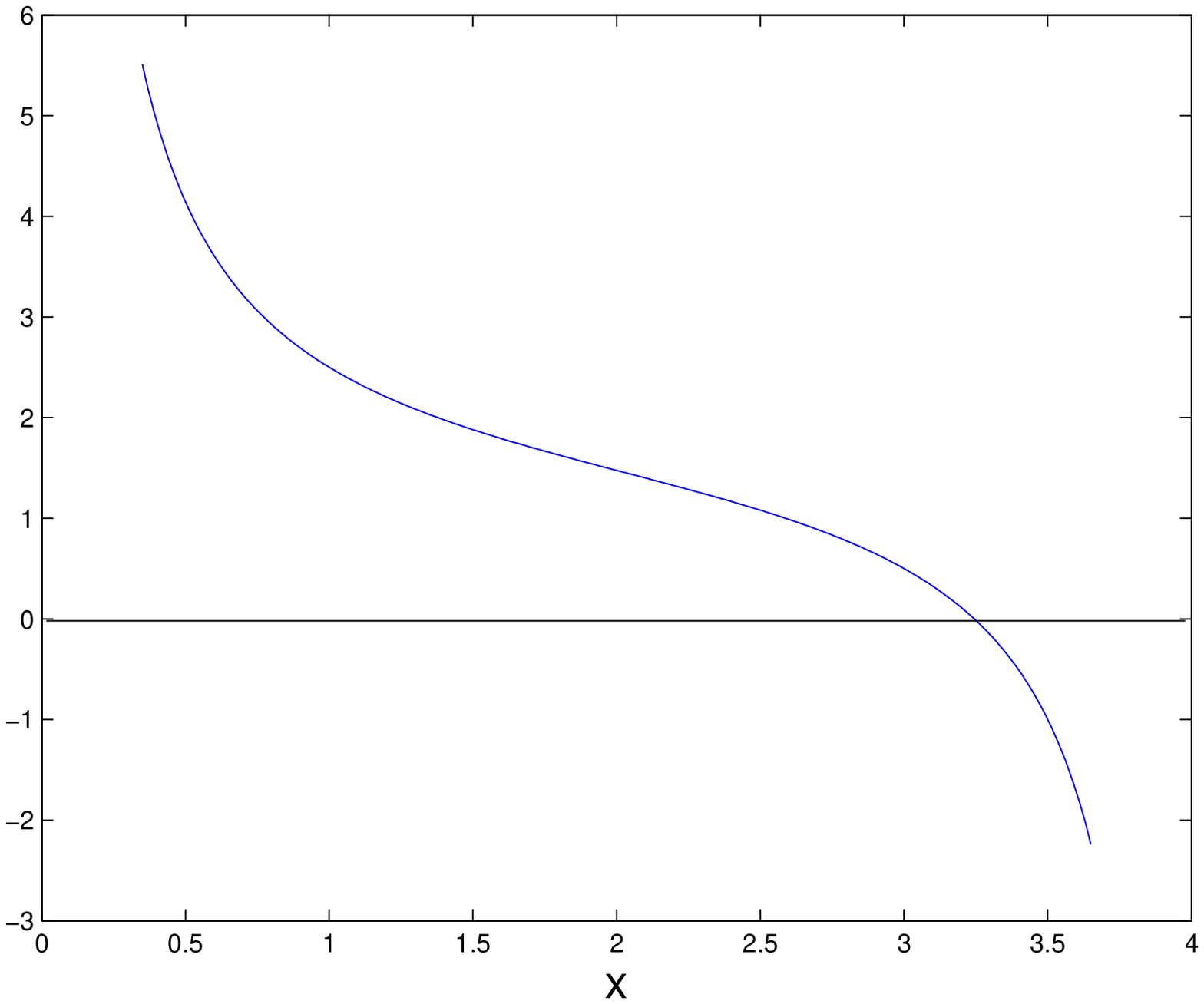}
\end{minipage} &
\begin{minipage}[b]{.33\textwidth}\centering
\vspace{-0.4cm}\includegraphics[width=0.7\textwidth,height=0.16\textheight,angle=0,trim=1cm 1cm 1cm 1cm]{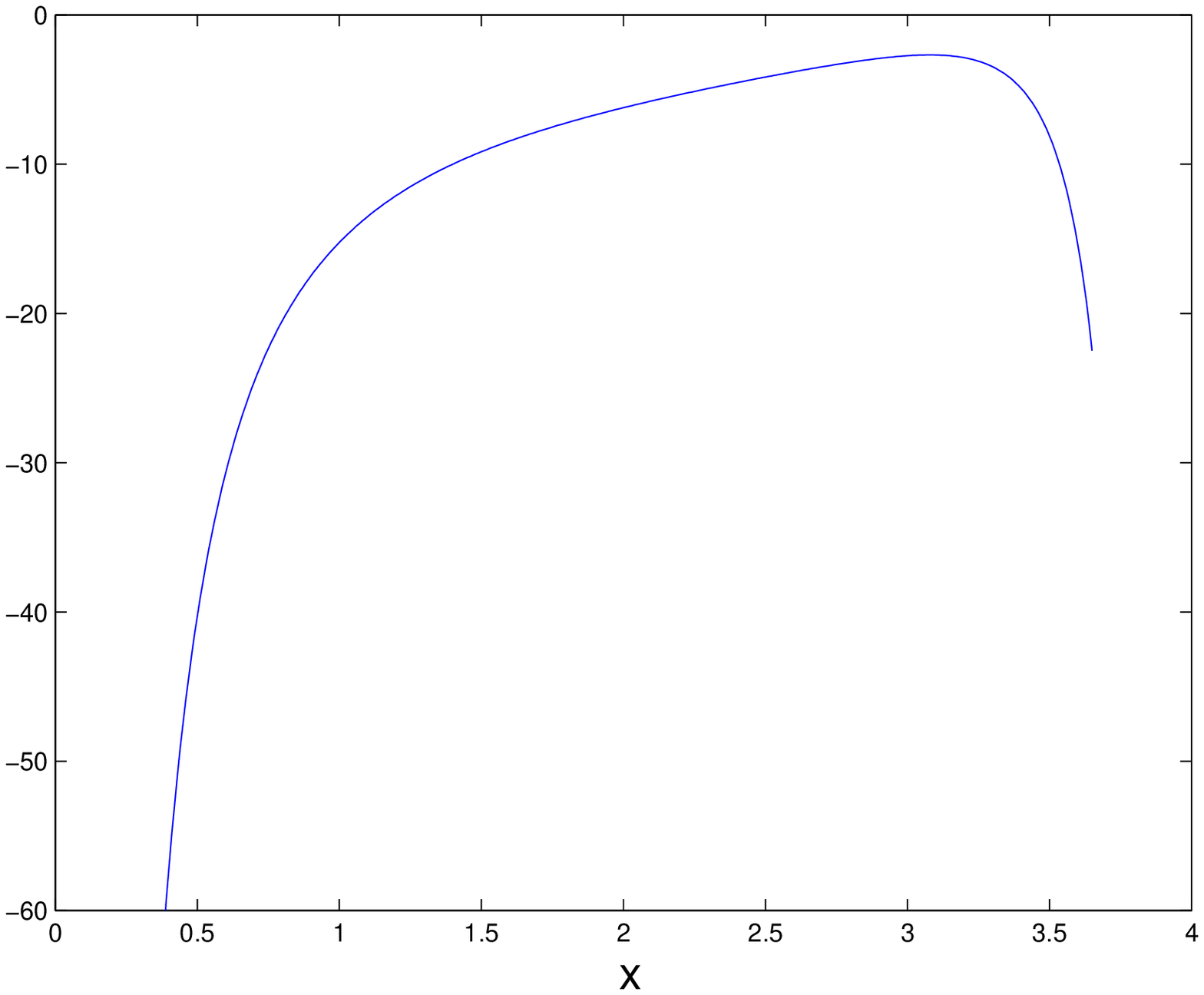}
\end{minipage}\\
(c) & (d) \\
\hspace{-0.5cm}\begin{minipage}[b]{.33\textwidth}\centering
\vspace{0.5cm}\includegraphics[width=0.78\textwidth,height=0.16\textheight,angle=0,trim=1cm 1cm 1cm 1cm]{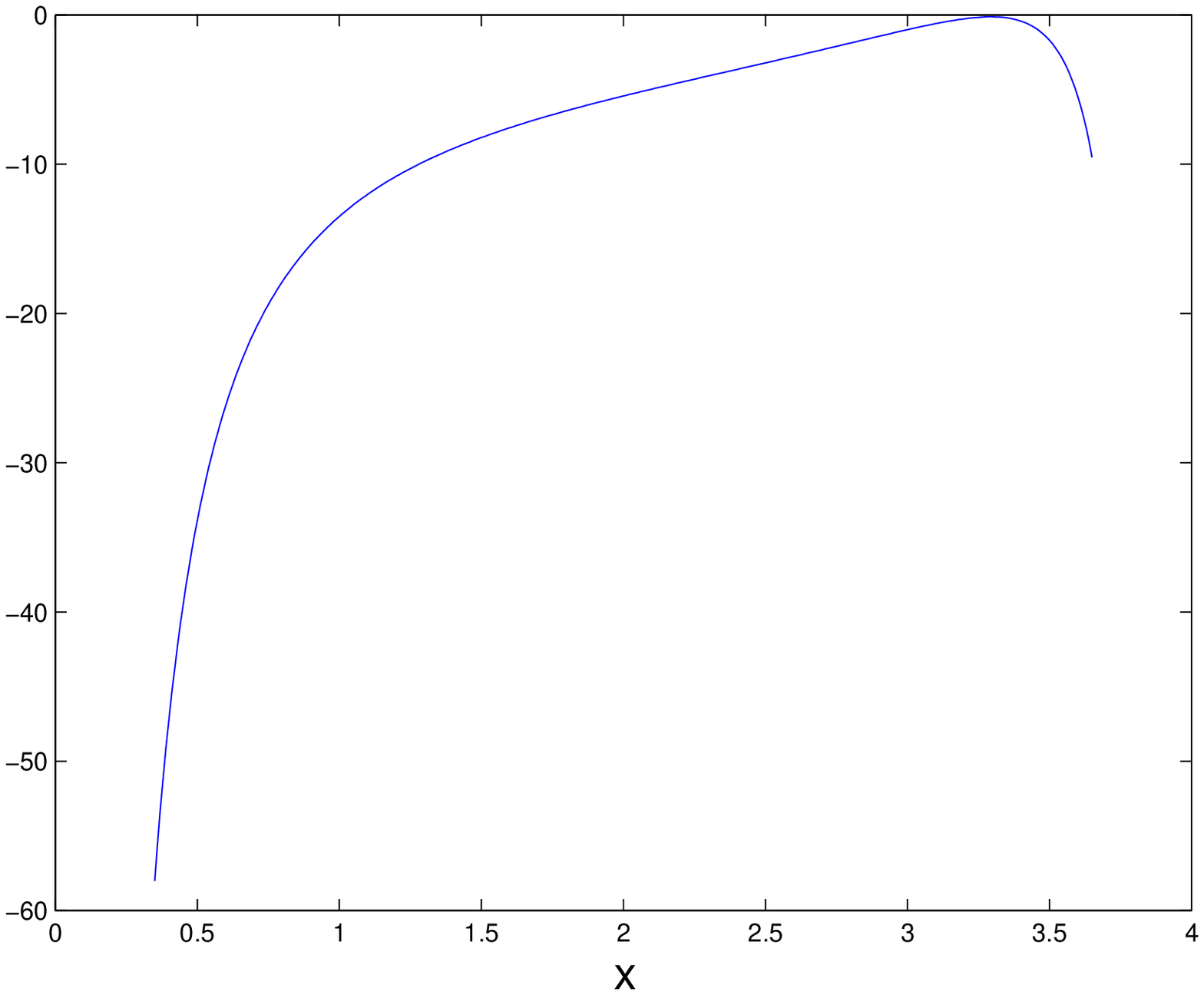}
\end{minipage}
 & \begin{minipage}[b]{.33\textwidth}\centering
\vspace{0.5cm}\includegraphics[width=0.7\textwidth,height=0.16\textheight,angle=0,trim=1cm 1cm 1cm 1cm]{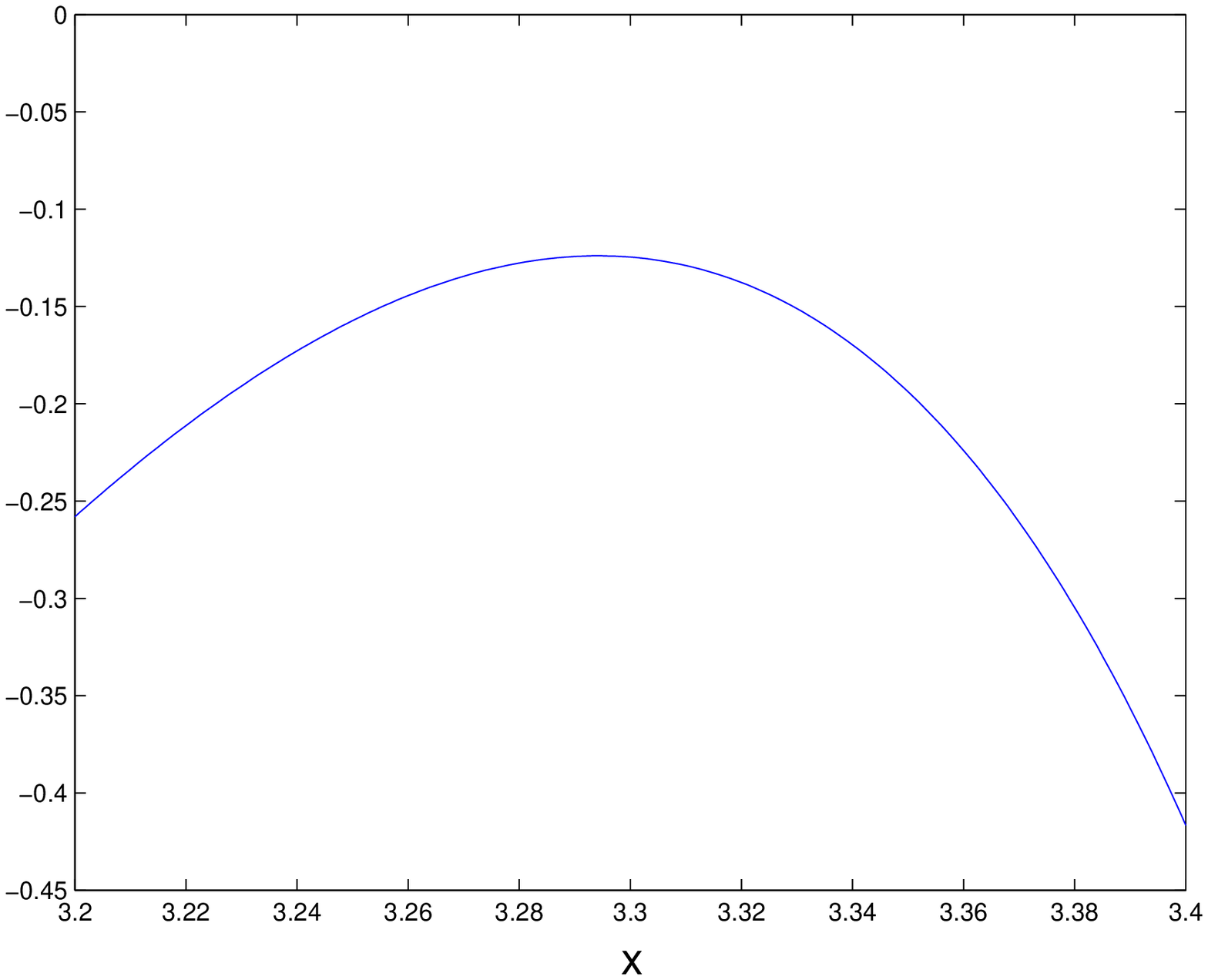}
\end{minipage}
\end{tabular}\end{center}
\caption{\textit{(a): Graph of $x\mapsto \partial_1 g(x,x)$ defined in (\ref{gradientfitnessex2}). (b): Graph of $x\to  \partial_{22}^{2}g(x,x)$. (c), (d): graphs of
$ x\to \partial_{11}^2 g(x,x)$ on $(0,4)$ and on $[3.2,3.4]$.}}\label{figderivee}
\end{figure}

\begin{prop}\label{etudefitnessex2}
Recall that $g$ has been defined in Lemma \ref{g}. Then, for all
$x\in (0,4)$,
\begin{align}
& \partial_1 g(x,x)=\frac{1.4 (x+1.4)(x-3.2)}{x(4-x)}\label{gradientfitnessex2}\\
& \partial^2_{11}g(x,x)=-\frac{5.8(x^2+ 1.8 x + 1.1)(x^2- 6.6 x + 10.9)}{x^2 (4-x)^2 }
\nonumber\\
& \partial^2_{22}g(x,x)=-\frac{5.8(x^2+1.8 x+1.5)(x^2-6.6 x+11.1)}{x^2(4-x)^2}.\nonumber
\end{align}
\end{prop}

\begin{proof}
In order
to obtain these derivatives, we replace as previously $y$ by
$y+\varepsilon_y$ and $x$ by $x+\varepsilon_x$ in
(\ref{equationzexemple2}), expand the right hand side in series with {\sc Maple}
 in $\varepsilon_x$ and $\varepsilon_y$ and identify the terms according to their order.
 \end{proof}

\noindent We remark that the function $x \to \partial_1 g(x,x)$
vanishes on the unique point $x^*=3.2$ of the interval $[0,4]$, which confirms
the simulations. $\partial_{11}^2 g(x,x)$
remains negative on $[0,4]$ and the function $y\mapsto g(y,x)-1$
is hence locally concave at the neighborhood of $(x,x)$, for each $x\in (0,4)$.
 In particular, $x^*$ corresponds to a local minimum of the fitness function.
 Moreover, we can check that $\partial^2_{22}g(x^*,x^*)\approx -2.9<
 \partial_{11}^2 g(x^*,x^*)\approx -0.3$ which entails that $x^*$ is a
 \textit{branching point} in the terminology of
  \cite{diekmann, metzgeritzmeszenajacobsheerwaarden, geritzmetzkisdimeszena}.
   These results can also be read on the PIP (Figure \ref{figtracefex2} (b))
    which shows that the population can still be invaded once it has reached the
    neighborhood of the evolutionary singularity $x^*$.

\section{Conclusion}

\noindent In this work, we have considered a population structured by traits and by a continuous
age, varying with time. The ecology is approximated by a
partial differential equation in both trait and age. Based on a measure-valued process approach and generalizing \cite{champagnat3, champagnatferrieremeleard}, we have obtained new approximations extending the
 Trait Substitution Sequence and the Canonical Equation of Adaptive Dynamics to trait and age-structured populations.
  To our knowledge, these equations have not been proposed so far in the biological literature.
  Our approach emphasizes that the object of interest here is the whole age-distribution of the monomorphic population at equilibrium, and not only the equilibrium trait.\\

\noindent Adding an age structure to trait structured population models opens the way to new problems dealing with life history feature. In the examples, we have seen that age-structure does not seem to change the qualitative behavior of the evolutionary approximation but refines the computation of the evolutionary stable state. As in the case without age-structure, the choice of the competition kernel has a strong influence on the patterns that can be observed on the microscopic simulations, and on the nature of the evolutionary singularities.

\begin{center}\textbf{Acknowledgments}\end{center}
The authors thank Nicolas Champagnat, Régis Ferrière and Pierre Collet for many
fruitful discussions.

\appendix

\section{Proof of Theorem \ref{theoremtss} and of Proposition \ref{propextinction}}\label{sectiontraitsubstitutionsequence}

\noindent The skeleton of the proof of Theorem \ref{theoremtss} is similar to the one in Champagnat \cite{champagnat3} for populations without
age-structure, but with additional difficulties. In \cite{champagnat3}, the distance between
a population and its equilibrium is obtained by comparisons of the
trait values and of the sizes. Here, age-distributions have to be taken into account. We need to compare the measure $Z^n_t$ with the stationary solution of (\ref{pde2}). Thus, we consider the space of finite measures $\mathcal{M}_F(\widetilde{\X})$ which we embed with the weak
convergence topology, induced by the Dudley metric (see Rachev
\cite{rachev} p79):
$$\forall \mu,\,\nu\in \mathcal{M}_F(\widetilde{\X}),\,D(\mu,\nu)=\sup_{{\scriptsize
 \begin{array}{c}f\mbox{ 1-Lip}\\
\|f\|_\infty\leq 1\end{array}}}\left|\langle \mu,f\rangle-\langle
\nu,f\rangle\right|.$$To prove Theorem \ref{theoremtss}, we establish that $\forall t\in \R_+$, $\forall
\Gamma\subset \X$ mesurable, $\forall \varepsilon>0$
\begin{eqnarray}
 \lim_{n\rightarrow +\infty}\mathbb{P}
 \left(Z^n_{t/nu_n}\mbox{ is monomorphic of trait } y\in \Gamma,\,
 D(Z^n_{t/nu_n},\widehat{\xi}_y)<\varepsilon\right)=\mathbb{P}_{x_0}
 \left(X_t\in \Gamma\right),\label{equationtheoremeprincipal}
\end{eqnarray}where $(X_t)_{t\in \R_+}$ has been defined in
Theorem \ref{theoremtss}. We use the
exponential deviations and estimates of times of exit of domains
established in \cite{trangdesdev} one the one hand, and
comparisons with linear age-structured birth and death processes
on the other hand. This last point is detailed through the main
steps of the proof of Proposition \ref{propextinction} in the
sequel. For a complete proof, we refer to \cite{chithese} (Chapter
6).

\noindent Let $\varepsilon>0$ and let us introduce the following stopping times:
\begin{align*}
R_\varepsilon^n  = & \inf\left\{t\geq 0\,\,\,\,|\,\,\,\,D\left( 1_{\{x_0\}}(x)Z^n_t(dx,da),\widehat{\xi}_{x_0}(dx,da)\right)\geq \varepsilon \right\}\\
S^n_\varepsilon  = & \inf\left\{t\geq 0\,\,\,\,|\,\,\,\,\int_{\widetilde{\X}}1_{\{y\}}(x)Z^n_t(dx,da)\geq \varepsilon\right\}\\
S^n_0  = & \inf\left\{t\geq
0\,\,\,\,|\,\,\,\,\int_{\widetilde{\X}}1_{\{y\}}(x)Z^n_t(dx,da)=0\right\}.
\end{align*}Recall that the time of first mutation $\tau$ and the time of return to a monomorphic state $\theta$ have been defined in Proposition \ref{propextinction}. By large deviations results, and by comparison results, we can prove that (see \cite{chithese}, Section 6.2.4)
\begin{equation}\exists \rho>0,\, \exists n_0\in \N^*,\,\forall n\geq n_0,\, \mathbb{P}^n_{x_0,q^n_0,y}\left(\frac{\rho}{nu_n}<\tau\wedge R^n_\varepsilon\right)\geq 1-2\varepsilon.\label{estimeetauetr}
\end{equation}On the period $[0,\tau\wedge R^n_\varepsilon\wedge
S^n_\varepsilon\wedge S^n_0]$, the mutant population is of mass
smaller than $\varepsilon$ and evolves in a resident population
that is closed to $\widehat{\xi}_{x_0}$. The extinction probability $z_0(y,x)$ that appears in Theorem \ref{theoremtss}
comes from comparison of the mutant population process with the
following linear birth and death age-structured (non-renormalized)
processes: for sufficiently large $n\in \N^*$ (such that $1-u_n
p>1-\varepsilon$), $\forall t\in [0,\tau\wedge
R^n_\varepsilon\wedge S^n_\varepsilon],\, \forall f\in
\mathcal{B}_b(\R_+,\R_+),\, $
$$\frac{1}{n}\langle Z^{1, \varepsilon}_{t},f\rangle \leq \int_{\widetilde{\X}}1_{\{y\}}(x)f(a)Z^n_t(dx,da)\leq
\frac{1}{n}\langle Z^{2,\varepsilon}_{t},f\rangle,$$where:
\begin{enumerate}
\item $(Z^{1,\varepsilon}_t)_{t\in \R_+}$ is the birth and death
process with  birth rate $b_1(a)= (1-\varepsilon)b(y,a)$ and death
rate
$d_1(a)=d(y,a)+\int_{\R_+}U((y,a),(x_0,\alpha))\widehat{m}(x_0,\alpha)d\alpha
+\bar{U}\varepsilon,$ \item $(Z_t^{2,\varepsilon})_{t\in \R_+}$ is
the birth and death process with
  birth rate $b_2(a)=b(y,a)$ and death rate $d_2(a)=d(y,a)+\int_{\R_+}U((y,a),(x_0,\alpha))\widehat{m}(x_0,\alpha)d\alpha
-\bar{U}\varepsilon.$
\end{enumerate}Let us define for $i\in \{1,2\}$ :
\begin{align*}
S^{i}_{\varepsilon n}  =  \inf\left\{t\geq
0\,\,\,\,|\,\,\,\,\langle Z^{i,\varepsilon}_{t},1\rangle\geq
\varepsilon n\right\},\,\,\quad  S^{i}_0  =  \inf\left\{t\geq
0\,\,\,\,|\,\,\,\,\langle
Z^{i,\varepsilon}_{t},1\rangle=0\right\}.
\end{align*}
The probabilities that these processes reach the level $\varepsilon
n$ before getting extinct are given by the following lemma, proved
at the end of the section:
\begin{lemme}\label{lemmebranchementlineaire}Let $\varepsilon>0$, and let $i\in \{1,2\}$. Let us consider the linear birth and death age-structured process $(Z^{i,\varepsilon}_t)_{t\in \R_+}$ introduced above and starting from $Z^{i,\varepsilon}_0(da)=\delta_0(da)$. Let $(t_n)_{n\in \N^*}$ be a positive real sequence such that $\lim_{n\rightarrow +\infty}t_n/\log n=+\infty$.
\begin{enumerate}
\item If $\int_0^{+\infty} b_i(a)e^{-\int_0^a
d_i(\alpha)d\alpha}da\leq 1$:
\begin{eqnarray}
\lim_{n\rightarrow+\infty}\mathbb{P}\left(S^{i}_0\leq t_n\wedge
S^{i}_{\varepsilon n}\right)=1\label{branching1}
\end{eqnarray}
\item If $\int_0^{+\infty} b_i(a)e^{-\int_0^a
d_i(\alpha)d\alpha}da> 1$:
\begin{align}
\lim_{n\rightarrow+\infty} & \mathbb{P}\left(S^i_0\leq t_n\wedge S^i_{\varepsilon n}\right)=z^i_0(y,x_0)\label{branching3}\\
\lim_{n\rightarrow+\infty} & \mathbb{P}\left(S^i_{\varepsilon
n}\leq t_n\leq S^i_0\right)=1-z^i_0(y,x_0),\label{branching4}
\end{align}where $z^i_0(y,x_0)$ is the smallest solution in $[0,1]$ of the equation analogous to (\ref{equationdeterminationsurvie}), where the birth and death rates are replaced by $b_i(a)$ and $d_i(a)$.
\hfill$\Box$
\end{enumerate}
\end{lemme}

\noindent The probability that the mutant population gets extinct before having reached the mass $\varepsilon$, before the occurrence of a new mutation, and before the resident population deviates from its equilibrium is then lower bounded by:
\begin{align}
\mathbb{P}^n_{x_0, q_0^n,y}&\left(S^{n}_0<\tau\wedge
R^n_\varepsilon \wedge S^{n}_\varepsilon\right)
  \geq
\mathbb{P}^n_{x_0, q_0^n,y}\left(S^{n}_0<\frac{\rho}{nu_n} \wedge S^{n}_\varepsilon,\,\,\, \frac{\rho}{nu_n}<\tau\wedge R^n_\varepsilon\right) \nonumber\\
 \geq & \mathbb{P}^n_{x_0, q_0^n,y}\left(S^{2}_0<\frac{\rho}{nu_n} \wedge S^{2}_{\varepsilon n},\,\,\, \frac{\rho}{nu_n}<\tau\wedge R^n_\varepsilon\right)
 \geq   z_0(y,x_0)-C\varepsilon,\label{etape1demoinvasion}
\end{align}by (\ref{estimeetauetr}), by Lemma \ref{lemmebranchementlineaire} and by showing that when $\varepsilon\rightarrow 0$, $z_0^i(y,x)\rightarrow z_0(y,x)$. The continuity of the extinction probability when the birth and death rates are perturbed by $\varepsilon$ is obtained thanks to the implicit function theorem (see Lemma 6.2.5 of \cite{chithese}). Similarly, the probability that the mutant population reaches the mass $\varepsilon$ before extinction, before the occurrence of another mutant or before the deviation of the resident population to its equilibrium is lower bounded by:
\begin{align}
\mathbb{P}^n_{x_0, q_0^n,y} & \left(S^{n}_\varepsilon<\tau\wedge
R^n_\varepsilon \wedge S^{n}_0\right)
   \geq
\mathbb{P}^n_{x_0, q_0^n,x_0}\left(S^{n}_\varepsilon<\frac{\rho}{nu_n} \wedge S^{n}_0,\,\,\, \frac{\rho}{nu_n}<\tau\wedge R^n_\varepsilon\right) \nonumber\\
   \geq &
\mathbb{P}^n_{x_0, q_0^n,x_0}\left(S^{1}_{\varepsilon
n}<\frac{\rho}{nu_n} \wedge
S^{1}_0,\,\frac{\rho}{nu_n}<\tau\wedge R^n_\varepsilon\right)
  \geq  1-z_0(y,x_0)-C\varepsilon.\label{etape2demoinvasion}
\end{align}

\noindent If $z_0(y,x_0)=1$, we obtain by (\ref{etape1demoinvasion}) that:
\begin{align}
\mathbb{P}^n_{x_0, q_0^n,y}\left(S^{n}_0<\tau\wedge
R^n_\varepsilon \wedge S^{n}_\varepsilon\right)  = &
\mathbb{P}^n_{x_0, q_0^n,y}\left(\theta <\tau\wedge
R^n_\varepsilon
,\,V_0=x_0,\,D(Z^n_{\theta},\widehat{\xi}_{V_0})<\varepsilon\right)
\geq   1-C\varepsilon.\label{rajouttss4}
\end{align}Since this is valid for every $\varepsilon>0$, Proposition \ref{propextinction} is proved in this case.
\noindent If $z_0(y,x_0)\in ]0,1[$, we show by following the proof of Lemma 3
in \cite{champagnat3} that once the mutant population has reached
the mass $\varepsilon$, it replaces the resident population with
probability one. Indeed, the microscopic process in this case
follows its deterministic \textit{dimorphic} approximation. Since
$z_0(y,x_0)<1$, we necessarily have (Proposition
\ref{proprappelg}):
\begin{equation}\int_{\R_+} b(y,a)\exp\left(-\int_0^a \widehat{d}(y,\alpha,x_0)d\alpha\right)da>1,\quad \mbox{ and}\label{tsschapetape1}
\end{equation}by the assumption of non coexistence in the long term of two traits (Ass. \ref{dimstan}):
\begin{equation}\int_0^{+\infty} b(x_0,a)\exp\left(-\int_0^a \widehat{d}(x_0,\alpha,y)d\alpha\right)da<1.\label{tsschapetape2}
\end{equation}Then, the deterministic approximation $\xi$ converges to $\delta_{y}(dx) \widehat{m}(y,a)da$. Its neighborhood is reached by the microscopic process in
finite time with a probability that tends to 1 when $n\rightarrow
+\infty$. When this happens, we can approximate the dynamics of the
resident population by comparing it with a linear birth and death
age-structured process with birth rate $b(x_0,a)$ and death rate
$\widehat{d}(x_0,a,y)$ as we did for the mutant population after
its introduction in the system. Because of
(\ref{tsschapetape2}), these linear branching processes can be
chosen sub-critical and we can then show that the resident
population gets extinct with a probability that tends to 1 when
$n\rightarrow +\infty$. This gives us that:
\begin{align*}
\lim_{n\rightarrow+\infty}\mathbb{P}^n_{x_0, q_0^n,y}
\left(\theta\leq \tau,\,\, V=y\right)
 \geq &  1-z_0(y,x_0)-C\varepsilon.
\end{align*}
Since (\ref{etape1demoinvasion}) gave us that:
\begin{align*}
\lim_{n\rightarrow +\infty}\mathbb{P}^n_{x_0, q_0^n,y}
\left(\theta\leq \tau,\,\, V=x_0\right)
 \geq &  z_0(y,x_0)-C\varepsilon,
\end{align*}and since $\mathbb{P}^n_{x_0,q_0^n,y}
\left(\theta\leq \tau,\,\, V=x_0\right)+\mathbb{P}^n_{x_0,
q_0^n,y} \left(\theta\leq \tau,\,\, V=y\right)=1,$
Proposition \ref{propextinction} is proved.

\subsection{Proof of Lemma \ref{lemmebranchementlineaire}}

For the linear birth and death processes, there is no accumulation
of birth and death events and
$$\lim_{n\rightarrow+\infty }t_n\wedge S^i_{\varepsilon n}=+\infty,\,\,\mathbb{P}-p.s.$$By dominated convergence, the left hand side of (\ref{branching1}) and (\ref{branching3}) converges to $\mathbb{P}\left(S^i_0<+\infty\right)$, which solves an equation similar to (\ref{equationdeterminationsurvie}). The result is then obtained from Proposition \ref{proprappelg}.

\noindent Let us now consider (\ref{branching4}):
\begin{align}
\mathbb{P}\left(S^i_{\varepsilon n}\leq t_n\leq S^i_0\right)= &
\mathbb{P}\left(S^i_{\varepsilon n}\leq t_n\mbox{ and }
S^i_0=+\infty\right)+\mathbb{P}\left(S^i_{\varepsilon n}\leq
t_n\leq S^i_0<+\infty\right).\label{rajoutannexeproclin1}
\end{align}The second term of (\ref{rajoutannexeproclin1}) is upper-bounded by $\mathbb{P}\left(t_n\leq S^i_0<+\infty\right),$ which tends to 0 when $n\rightarrow +\infty$ by the choice of $t_n$. Let us now turn to the first term. Under the assumptions of Point 2, there exists a unique $\lambda>0$ such that:
$$\int_0^{+\infty}b_i(a)e^{-\lambda a-\int_0^a d_i(\alpha)d\alpha}da=1.$$
Let $Y_1$ and $W$ be the number of children and lifelength of an individual with birth rate $b_i(a)$ and death rate $d_i(a)$.
From the sufficient conditions of Doney \cite{doney}, when $t\rightarrow +\infty$,
$(e^{-\lambda t}\langle Z^{i,\varepsilon}_t,1\rangle)_{t\in \R_+}$
converges almost surely and in mean square to a strictly positive random
variable on $\{S^i_0=+\infty\}$ if $\mathbb{E}(Y_1^2)<+\infty$ and $\mathbb{E}\left(\left(\int_0^W e^{-\lambda a}b_i(a)da\right)\log\left(\int_0^W e^{-\lambda a}b_i(a)da \right)\right)<+\infty$. Recall Assumptions \ref{hypothesetaux}, and let $\varepsilon>0$ be sufficiently small so that $d_2$ remains bounded below by $\underline{d}>0$. Since the density of $W$ is $d_i(w)\exp(\int_0^w d_i(\alpha)d\alpha)$ and since conditionally to $W$, $Y_1$ is a Poisson random variable with parameter
$\int_0^W b_i(a)da$, we have
\begin{align*}
\mathbb{E}\left(Y_1^2\right)= & \int_0^{+\infty}\left(\int_0^w
b_i(a)da\right) \,d_i(w) e^{-\int_0^w d_i(a)da}dw \leq
\bar{b}(\bar{d}+\varepsilon\bar{U}+\bar{U}\widehat{M}_x)\int_0^{+\infty}
w e^{-\underline{d}w}dw<+\infty.
\end{align*}
The function $x\in \R^* \mapsto x\log(|x|)$ tends to 0 when $x\rightarrow 0$. On $\R_+$, this function first decreases from 0 to $-e^{-1}$ and then increases. Thus:
\begin{multline}
-\frac{1}{e}\leq \mathbb{E}\left(\left(\int_0^W e^{-\lambda a}b_i(a)da\right)\log\left(\int_0^W e^{-\lambda a}b_i(a)da \right)\right)
\leq  \mathbb{E}\left(\max(0,\bar{b} W \log(\bar{b}
W))\right)\\
\leq  \bar{b}(\bar{d}+\varepsilon \bar{U}+\bar{U}\widehat{M}_x)\int_0^{+\infty}w\left|\ln(\bar{b}w)\right|e^{-\underline{d}w}dw<+\infty.\nonumber
\end{multline}
Doney's sufficient conditions are then satisfied. On $\{S^i_0=+\infty\}$ we thus have \begin{equation} \lim_{t\rightarrow +\infty}\frac{\log \langle
Z^{i,\varepsilon}_t,1\rangle}{t}=\lambda>0.\label{limitelogannexe}
\end{equation}
Let us consider $n>1/\varepsilon$, so that $\log(\varepsilon n)>0$. Since $\lim_{n\rightarrow +\infty}S^i_{\varepsilon n}=+\infty$ almost surely, we have by (\ref{limitelogannexe}) that almost surely on $\{S^i_0=+\infty \}$
\begin{align}
\lim_{n\rightarrow +\infty}\frac{\log \varepsilon n
}{S^i_{\varepsilon n}}\geq \lim_{n\rightarrow +\infty}\frac{\log
\langle Z_{S^i_{\varepsilon n}-},1\rangle }{S^i_{\varepsilon
n}}=\lambda>0.
\end{align}
Then:
\begin{align}
\lim_{n\rightarrow +\infty}\mathbb{P}\left(S^i_{\varepsilon n}
\leq t_n,\quad S^i_0=+\infty\right)
=& \lim_{n\rightarrow +\infty} \mathbb{P}\left(\frac{S^i_{\varepsilon n}}{\log (\varepsilon n)} \leq \frac{t_n}{\log(\varepsilon n)},\quad S^i_0=+\infty\right)\nonumber\\
= &
\mathbb{P}\left(S^i_0=+\infty\right)=1-z^i_0(y,x),\label{limiteprobasvarepsilonnannexe}
\end{align}as by choice of $t_n$, $\lim_{n\rightarrow +\infty}t_n/(\log \varepsilon n)=+\infty$.

\section{Stability of the nontrivial equilibrium in Example 2}\label{annexestabiliteex2}

Let us prove that the nontrivial equilibrium (\ref{solutionstationnaireex2}) of Equations (\ref{eq1ex2})-(\ref{eq2ex2}) is asymptotically stable. In view of Theorem 4.12 in Webb \cite{webb} (p207), it is sufficient to study the
eigenvalues of the operator $B$ associated to the linearization of these equations in the neighborhood of this equilibrium. This leads us to study existence of non
trivial solutions for the following equations:
\begin{align}
 & \frac{d u}{d a}(a)=-\left(a\widehat{E}(x)+\lambda\right)u(a)-a\widehat{m}(x,a)E(x),\label{equationslinearisees1}\\
 & u(0)=\int_0^{+\infty}x(4-x)u(a)da,\label{equationslinearisees2}\\
 & E(x)=\int_0^{+\infty}U(x,x)(1+e^{-\alpha})u(\alpha)d\alpha.\label{equationslinearisees3}
\end{align}Equation (\ref{equationslinearisees1}) implies that:
\begin{align}
u(a)= & e^{-\widehat{E}(x)\frac{a^2}{2} -\lambda
a}\left(u(0)-\widehat{m}(x,0)E(x)\int_0^a \alpha e^{\lambda
\alpha}d\alpha\right).\label{exploitationeq1}
\end{align}Plugging this in Equation (\ref{equationslinearisees2}):
\begin{align}
u(0)=\int_0^{+\infty}x(4-x)e^{-\frac{\widehat{E}(x)a^2}{2}-\lambda
a}\left(u(0)-\widehat{m}(x,0)E(x)\int_0^a \alpha e^{\lambda
\alpha}d\alpha\right)da.\label{exploitationeq2}
\end{align}
\noindent If $\lambda=0$, then (\ref{exploitationeq2}) and the balance condition (\ref{balanceconditionex2}) lead us to:
\begin{align*}
0=\widehat{m}(x,0)E(x)x(4-x)\int_0^{+\infty}\frac{a^2}{2}e^{-\frac{\widehat{E}(x)a^2}{2}}
da,
\end{align*}which is never satisfied, since the right hand side is strictly positive. Hence $\lambda=0$ is not an eigenvalue of $B$.\\
\noindent If $\lambda\not= 0$, (\ref{exploitationeq2}) gives
\begin{align}
u(0)=-\frac{\widehat{m}(x,0)E(x)\int_0^{+\infty}e^{-\frac{\widehat{E}(x)a^2}{2}}\left(e^{-\lambda
a}+ \lambda a -1\right) da}{\lambda^2
\int_0^{+\infty}e^{-\frac{\widehat{E}(x)a^2}{2}}\left(1-e^{-\lambda
a}\right)da}.\label{u0}
\end{align}
\noindent If $\lambda>0$, then (\ref{u0}) has a sign opposite to $E(x)$ and we deduce from (\ref{exploitationeq1}) that so has $u(a)$. This contradicts (\ref{equationslinearisees3}) unless both members of the equation are zero. This is excluded since $u$ is an eigenfunction and hence is not constant equal to zero. Hence there is no eigenvalue on the nonnegative real axis.\\
\noindent Let us now consider $\lambda=\lambda_1+\textit{i}\lambda_2$ with $\lambda_1>0$. Since
$\int_0^{+\infty}ae^{-\widehat{E}(x)a^2/2}da=1/\widehat{E}(x)$,
\begin{align*}
u(0)= &
\frac{\widehat{m}(x,0)E(x)}{\lambda^2}\left(1-\frac{\lambda}{\widehat{E}(x)
\int_0^{+\infty}e^{-\frac{\widehat{E}(x)a^2}{2}}\left(1-e^{-\lambda
a}\right)da}\right),
\end{align*}and hence (\ref{exploitationeq1}) becomes
\begin{align}
u(a)= &
e^{-\widehat{E}(x)\frac{a^2}{2}}\frac{\widehat{m}(x,0)E(x)}{\lambda}
\left(-a+\frac{1}{\lambda}-\frac{e^{-\lambda
a}}{\widehat{E}(x)\int_0^{+\infty}e^{-\frac{\widehat{E}(x)\alpha^2}{2}}
\left(1-e^{-\lambda
\alpha}\right)d\alpha}\right).\label{ua}
\end{align}Using (\ref{ua}) in (\ref{equationslinearisees3})
\begin{multline}
0= \int_0^{+\infty}\frac{C\nu}{1+\nu}\widehat{m}(x,0)(1+e^{-a})
e^{-\frac{\widehat{E}(x)a^2}{2}}\\
\times \left(-a+\frac{1}{\lambda}-\frac{e^{-\lambda
a}}{\widehat{E}(x)\int_0^{+\infty}e^{-\frac{\widehat{E}(x)\alpha^2}{2}}\left(1-e^{-\lambda
\alpha}\right)d\alpha}\right)da-\lambda=:\Lambda(\lambda,x).\label{defgrandphivalppes}
\end{multline}
\unitlength=0.7cm
\begin{figure}[ht]
  \begin{center}
    \begin{picture}(10,7)
    \put(1,3){\line(1,0){8.5}}
    \put(1,0){\line(0,1){6}}
    \put(9,5.5){\vector(-1,0){8}}
    \put(9,0.5){\vector(0,1){5}}
    \put(1,0.5){\vector(1,0){8}}
    \put(1.1,2){\vector(0,-1){1.5}}
    \put(2.1,3){\vector(-1,-1){1}}
    \put(1.1,4){\vector(1,-1){1}}
    \put(1.1,5.5){\vector(0,-1){1.5}}
    \put(0.65,3.25){$0$}
    \put(9.5,5.5){$\Gamma$}
    \end{picture}
  \end{center}
  \caption{The Jordan curves used to study the zeros of $\lambda\mapsto \Lambda(\lambda,x)$}\label{figurejordan}
\end{figure}
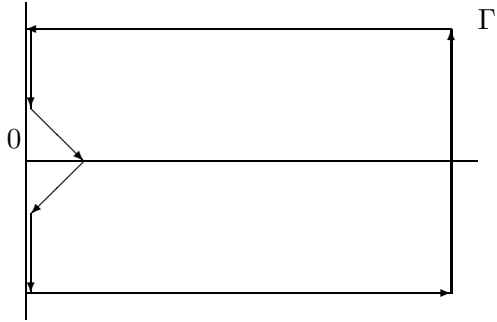
\noindent We wish to check that for any given $x$ in $[0,4]$, the complex
 function $\lambda\mapsto \Lambda(\lambda,x)$ does not have roots in the half plane
  $\mathcal{P}^+$ of complex numbers with nonnegative real part. Let $x\in [0,4]$ be fixed.
  When $|\lambda|\rightarrow +\infty$, the integral in (\ref{defgrandphivalppes}) converges
   to a finite value. Hence, the roots of $\Lambda(\lambda,x)$ in $\mathcal{P}^+$,
    if they exist, belong necessarily to a compact set. In order to show that there are
    no roots on compact sets of $\mathcal{P}^+$, we use the argument principle
     (see Henrici \cite{henrici}, Section 4.10). The function $\lambda\mapsto
      \Lambda(\lambda,x)$ is analytic on $\mathcal{P}^+\setminus \{0\}$.
       Let $\Gamma\,:\tau\in [0,1]\mapsto \Gamma(\tau)$ be a positively
       oriented Jordan curve of $ \mathcal{P}^+\setminus\{0\}$.
       Assume that there are $J$ zeros $(z_j)_{j\in \lbrac 1, J\rbrac}$
       with multiplicities $(m_j)_{j\in \lbrac 1, J\rbrac}$ contained in $\Gamma$. Then
\begin{align}
& \sum_{j=1}^J m_j  =\frac{1}{2\pi
i}\int_{\Gamma}\frac{\Lambda'(\lambda,x)}{\Lambda(\lambda,x)}d\lambda=
\frac{1}{2\pi i}\int_{\Lambda(\Gamma,x)}\frac{1}{w}dw=
n(\Lambda(\Gamma,x),0)\\
& \mbox{where }
n(\Lambda(\Gamma,x),0):=\frac{[\mbox{arg}(\lambda)]_{\Lambda(\Gamma,x)}}{2\pi}=
\frac{\phi(1)-\phi(0)}{2\pi},\label{argument}
\end{align}where $\phi(\tau)$ is a continuous version of the argument of
$\Lambda(\Gamma(\tau),x)$. Our purpose is to show that for chosen Jordan curves in $ \mathcal{P}^+\setminus\{0\}$, the right member of (\ref{argument}) is zero. This will entail that no zeros are contained in the chosen curves.
\noindent We consider the Jordan curves drawn in Figure \ref{figurejordan}. Since the computations can not been carried explicitly, we compute numerically $n(\Lambda(\Gamma,x),0)$ by following the algorithm proposed by Henrici (\cite{henrici}, end of Section 4.6). We then let $x$ vary along a grid between 0 and 4. The numerical results tell us that the variation of the argument $(\phi(1)-\phi(0))/(2\pi)$ remains zero for every $x\in [0,4]$. We have thus checked numerically that there is no complex solution $\lambda$ of (\ref{defgrandphivalppes}) with nonnegative real part, and the nontrivial equilibrium (\ref{solutionstationnaireex2}) is asymptotically stable.

{\footnotesize
\bibliographystyle{plain}
\bibliography{Biblioarticles,Bibliolivres,Biblioconf,Bibliounpub,Bibliophdthesis}

\providecommand{\noopsort}[1]{}\providecommand{\noopsort}[1]{}\providecommand{%
\noopsort}[1]{}\providecommand{\noopsort}[1]{}
\begin{thebibliography}{10}

\bibitem{athreyaney}
K.B. Athreya and P.E. Ney.
\newblock {\em Branching Processes}.
\newblock Springer edition, 1970.

\bibitem{busenbergiannelli}
S.~Busenberg and M.~Iannelli.
\newblock A class of nonlinear diffusion problems in age-dependent population
  dynamics.
\newblock {\em Nonlinear Analysis Theory, Methods and Applications},
  7(5):501--529, 1983.

\bibitem{champagnat}
N.~Champagnat.
\newblock Convergence and existence for polymorphic adaptive dynamics jump and
  degenerate diffusion models.
\newblock Preprint Laboratoire MODAL'X 03/7, 03 2003.

\bibitem{champagnat3}
N.~Champagnat.
\newblock A microscopic interpretation for adaptative dynamics trait
  substitution sequence models.
\newblock {\em Stochastic Processes and their Applications}, 2006.

\bibitem{champagnatferrieremeleard2}
N.~Champagnat, R.~Ferri\`{e}re, and S.~M\'{e}l\'{e}ard.
\newblock Individual-based probabilistic models of adpatative evolution and
  various scaling approximations.
\newblock In {\em Proceedings of the 5th seminar on Stochastic Analysis, Random
  Fields and Applications}, Probability in Progress Series, Ascona, Suisse,
  2006. Birkhauser.

\bibitem{champagnatferrieremeleard}
N.~Champagnat, R.~Ferri\`{e}re, and S.~M\'{e}l\'{e}ard.
\newblock Unifying evolutionary dynamics: from individual stochastic processes
  to macroscopic models via timescale separation.
\newblock {\em Theoretical Population Biology}, 2006.

\bibitem{charlesworth}
B.~Charlesworth.
\newblock {\em Evolution in Age structured Population}.
\newblock Cambridge University Press, 2 edition, 1994.

\bibitem{dawson}
D.~A. Dawson.
\newblock Mesure-valued markov processes.
\newblock In Springer, editor, {\em Ecole d'Et\'{e} de probabilit\'{e}s de
  Saint-Flour XXI}, volume 1541 of {\em Lectures Notes in Math.}, pages 1--260,
  New York, 1993.

\bibitem{dieckmannheinoparvinen}
U.~Dieckmann, M.~Heino, and K.~Parvinen.
\newblock The adaptive dynamics of function-valued traits.
\newblock {\em Journal of Theoretical Biology}.
\newblock in press.

\bibitem{dieckmannlaw}
U.~Dieckmann and R.~Law.
\newblock The dynamical theory of coevolution: a derivation from stochastic
  ecological processes.
\newblock {\em Journal of Mathematical Biology}, 34:579--612, 1996.

\bibitem{diekmann}
O.~Diekmann.
\newblock A beginner's guide to adaptive dynamics.
\newblock {\em Banach Center Publications}, 63:47--86, 2003.

\bibitem{doney}
R.A. Doney.
\newblock Age-dependent birth and death processes.
\newblock {\em Z.Wahrscheinlichkeitstheorie verw.}, 22:69--90, 1972.

\bibitem{ernandedieckmanheino}
B.~Ernande, U.~Dieckman, and M.~Heino.
\newblock Adaptive changes in harvested populations: plasticity and evolution
  of age and size at maturation.
\newblock {\em Proc. R. Soc. Lond. B}, 271:415--423, 2004.

\bibitem{vonfoerster}
H.~Von Foerster.
\newblock Some remarks on changing populations.
\newblock In Grune~\& Stratton, editor, {\em The Kinetics of Cellular
  Proliferation}, pages 382--407, New York 1959.

\bibitem{fourniermeleard}
N.~Fournier and S.~M\'{e}l\'{e}ard.
\newblock A microscopic probabilistic description of a locally regulated
  population and macroscopic approximations.
\newblock {\em Ann. Appl. Probab.}, 14(4):1880--1919, 2004.

\bibitem{geritzmetzkisdimeszena}
S.A.H. Geritz, J.A.J. Metz, E.~Kisdi, and G.~Mesz\'{e}na.
\newblock The dynamics of adaptation and evolutionary branching.
\newblock {\em Physical Review Letters}, 78:2024--2027, 1997.

\bibitem{gurtinmaccamy}
M.E. Gurtin and R.C. MacCamy.
\newblock Nonlinear age-dependent population dynamics.
\newblock {\em Arch. Rat. Mech. Anal.}, 54:281--300, 1974.

\bibitem{henrici}
P.~Henrici.
\newblock {\em Applied and computational complex analysis}.
\newblock John wiley \& sons edition, 1997.

\bibitem{henson}
S.M. Henson.
\newblock A continuous age-structured insect population model.
\newblock {\em Journal of Mathematical Biology}, 39:217--243, 1999.

\bibitem{hofbauersigmund}
J.~Hofbauer and R.~Sigmund.
\newblock Adaptive dynamics and evolutionary stability.
\newblock {\em Appl. Math. Letters}, 3:75--79, 1990.

\bibitem{kendall}
D.G. Kendall.
\newblock Stochastic processes and population growth.
\newblock {\em J. Roy. Statist. Sec., Ser. B}, 11:230--264, 1949.

\bibitem{kingsbury}
A.~Kingsbury.
\newblock Pink salmon.
\newblock Alaska Department of Fish and Game,
  \verb"www.adfg.state.ak.us/pubs/notebook/fish/pink.php", 1994.

\bibitem{kisdi}
E.~Kisdi.
\newblock Evolutionary branching under asymmetric competition.
\newblock {\em J. Theor. Biol.}, 197(2):149--162, 1999.

\bibitem{mckendrick}
A.G. McKendrick.
\newblock Applications of mathematics to medical problems.
\newblock {\em Proc. Edin. Math.Soc.}, 54:98--130, 1926.

\bibitem{metzencyclop}
J.A.J. Metz.
\newblock Fitness.
\newblock In {\em Encyclopedia of Ecology}, S.E. Jorgensen Ed. Elsevier.

\bibitem{metzgeritzmeszenajacobsheerwaarden}
J.A.J. Metz, S.A.H. Geritz, G.~Mesz\'{e}na, F.A.J. Jacobs, and J.S.~Van
  Heerwaarden.
\newblock Adaptative dynamics, a geometrical study of the consequences of
  nearly faithful reproduction.
\newblock {\em S.J. Van Strien \& S.M. Verduyn Lunel (ed.), Stochastic and
  Spatial Structures of Dynamical Systems}, 45:183--231, 1996.

\bibitem{mischlerperthameryzhik}
S.~Mischler, B.~Perthame, and L.~Ryzhik.
\newblock Stability in a nonlinear population maturation model.
\newblock {\em Mathematical Models \& Methods in Applied Science}, 12:1--22,
  2002.

\bibitem{murray}
J.D. Murray.
\newblock {\em Mathematical Biology}, volume~19 of {\em Biomathematics}.
\newblock Springer, 1993.
\newblock Third Edition.

\bibitem{oelschlager}
K.~Oelschl\"{a}ger.
\newblock Limit theorem for age-structured populations.
\newblock {\em The Annals of Probability}, 1990.

\bibitem{parvinendieckmannheino}
K.~Parvinen, U.~Dieckmann, and M.~Heino.
\newblock Function-valued adaptive dynamics and the calculus of variations.
\newblock {\em Journal of Mathematical Biology}, 52:1--26, 2006.

\bibitem{perthameryzhik}
B.~Perthame and L.~Ryzhik.
\newblock Exponential decay for the fragmentation or cell-division equation.
\newblock {\em Journal of the Differential Equations}, 210:155--177, 2005.

\bibitem{promislow}
D.E.L. Promislow.
\newblock Senescence in natural population of mammals: a comparative study.
\newblock {\em Evolution}, 45(8):1869--1887, 1991.

\bibitem{rachev}
S.T. Rachev.
\newblock {\em Probability Metrics and the Stability of Stochastic Models}.
\newblock John Wiley \& Sons, 1991.

\bibitem{rotenberg}
M.~Rotenberg.
\newblock Transport theory for growing cell populations.
\newblock {\em Journal of Theoretical Biology}, 103:181--199, 1983.

\bibitem{thieme}
H.R. Thieme.
\newblock {\em Mathematics in Population Biology}.
\newblock Princeton Series in Theoretical and Computational Biology. Princeton
  University Press, {S}imon {A}. {L}evin edition, 2003.

\bibitem{chithese}
V.C. Tran.
\newblock {\em Mod\`{e}les particulaires stochastiques pour des probl\`{e}mes
  d'\'{e}volution adaptative et pour l'approximation de solutions
  statistiques}.
\newblock PhD thesis, Universit\'{e} Paris X - Nanterre.
\newblock \verb"http://tel.archives-ouvertes.fr/tel-00125100".

\bibitem{trangdesdev}
V.C. Tran.
\newblock Large population limit and time behaviour of a stochastic particle
  model describing an age-structured population.
\newblock {\em ESAIM: P\&S}, 2007.
\newblock In press.

\bibitem{webb}
G.F. Webb.
\newblock {\em Theory of Nonlinear Age-Dependent Population Dynamics},
  volume~89 of {\em Monographs and Textbooks in Pure and Applied mathematics}.
\newblock Marcel Dekker, inc., New York - Basel, 1985.

\end{thebibliography}
}
\end{document}